%% file: PBW-Basis_main.tex
\documentclass[11pt]{amsart}
\usepackage[colorlinks=true, pdfstartview=FitV, linkcolor=blue, citecolor=blue, urlcolor=blue, breaklinks=true]{hyperref}
\usepackage{amsmath,amsfonts,amssymb,amsthm,amscd,comment,euscript}
\usepackage{xcolor}
\usepackage{tikz}
\usepackage{tikz}
\usepackage[all]{xy}
\usepackage{tikz}
\usepackage[hypcap=false]{caption}
\usepackage{array}
\usepackage{bookmark}
\usepackage{graphicx}
\usepackage[lflt]{floatflt} %bilder vom text umfließt
\usepackage{dsfont}
\newtheorem{thm}{Theorem}[subsection]
\newtheorem*{thmon}{Main Theorem}
\newtheorem{cor}[thm]{Corollary}
\newtheorem{remon}{Remark}
\newtheorem{lem}[thm]{Lemma}
\newtheorem{prop}[thm]{Proposition}

\newtheorem{defi}[thm]{Definition}
\newtheorem{rem}[thm]{Remark}

\newtheorem{exa}[thm]{Example}

\newcounter{cnt}
 \makeatletter
\def\mydggeometry{\makeatletter\dg@YGRID=1\dg@XGRID=20\unitlength=0.003pt\makeatother}
\makeatother \theoremstyle{remark}
\numberwithin{equation}{section}

\let\bwdg\bigwedge
\def\bigwedge{{\textstyle\bwdg}}
\theoremstyle{definition}

\theoremstyle{definition}

\newtheorem*{theoA}{Theorem A}
\newtheorem*{theoB}{Theorem B}

\newcommand\al{\alpha}

\newcommand\Lg{\mathfrak{g}}
\newcommand\Lh{\mathfrak{h}}

\newcommand\Ln{\mathfrak{n}}
\newcommand{\bp}{{\bf p}}
\newcommand{\bs}{{\bf s}}

\newcommand{\bt}{{\bf t}}
\newcommand{\bq}{{\bf q}}
\newcommand{\br}{{\bf r}}

\newcommand{\bm}{{\bf m}}
\newcommand{\bn}{{\bf n}}
\newcommand{\bx}{{\bf x}}
\newcommand{\bb}{{\bf b}}
\newcommand{\bu}{{\bf u}}
\newcommand{\bh}{{\bf h}}
\DeclareMathOperator{\height}{ht}

\newcommand{\Z}{\mathbb Z }
\newcommand{\N}{\mathbb N }
\newcommand{\R}{\mathbb R }

\newcommand{\B}{\mathbb B }
\newcommand{\C}{\mathbb C }
\newcommand{\Bw}{\mathbb B_{\omega} }

\newcommand{\Bwn}{\mathbb B_{\omega_n}}
\newcommand{\rr}{\Delta_{+}^{\omega}} %Die Menge der relevanten Wurzeln.
\newcommand{\rrk}{\Delta_{+}^{\omega_k}}
\newcommand{\rrn}{\Delta_{+}^{\omega_n}}
\newcommand{\rrnme}{\Delta_{+}^{\omega_{n-1}}}
\newcommand{\rrp}{\Delta_{+}}
\newcommand{\wk}{\omega_k}

\newcommand{\rrl}{\Delta_{+}^{\lambda}}
\newcommand{\prr}{\mathcal{P}(\Delta)}
\newcommand{\prrl}{\mathcal{P}(\Delta_{+}^{\lambda})}
\newcommand{\prk}{\mathcal{P}(\Delta_{+}^{\omega_k})}
\newcommand{\prn}{\mathcal{P}(\Delta_{+}^{\omega_n})}
\newcommand{\dsw}{\overline{D}_{\omega}}
\newcommand{\dswk}{\overline{D}_{\omega_k}}
\newcommand{\dswn}{\overline{D}_{\omega_n}}
\newcommand{\dswnme}{\overline{D}_{\omega_{n-1}}}
\newcommand{\dsl}{\overline{D}_{\lambda}}
\newcommand{\maxx}{\mathrm{max}}     %Maximale Wurzel
\newcommand{\minn}{\mathrm{min}}     %Minimale Wurzel
\newcommand{\pmw}{P(m\omega)}     	 %Polytop zum gewicht m*omega
\newcommand{\smw}{S(\lambda)}        %die ganzzahligen Punkte im Polytop
\newcommand{\sw}{S(\omega)}

\newcommand{\supp}{\mathrm{supp}}    %Support-Abb.
\newcommand{\suppm}{\mathrm{supp_{\mathrm{m}}}}
\newcommand{\ms}{M_{\bs}}  %Die Menge M_s
\newcommand{\Vw}{V(\omega)}  %Der Höchstgewichtmodul V(\omega)
\newcommand{\vw}{v_{\omega}} %Der Höchstgewichtvektor
\newcommand{\vwe}{v_{\omega_1}} %Der Höchstgewichtvektor
\newcommand{\vwk}{v_{\omega_k}} %Der Höchstgewichtvektor
\newcommand{\vmw}{v_{\lambda}} %Der Höchstgewichtvektor

\newcommand{\Vwk}{V(\omega_k)}

\newcommand{\Vmw}{V(\lambda)}
\newcommand{\Vmwa}{V(\lambda)^{a}}
\newcommand{\unm}{U(\mathfrak{n}^-)}
\newcommand{\unp}{U(\mathfrak{n}^+)}
\newcommand{\snm}{S(\mathfrak{n}^-)}
\newcommand{\Sn}{S(\mathfrak{n}^{-})}
\newcommand{\g}{\mathfrak{g}}
\newcommand{\Bl}{\mathbb{B}_{\lambda}}
\newcommand{\Vl}{V(\lambda)}
\newcommand{\vl}{v_{\lambda}}
\newcommand{\Vla}{V(\lambda)^{a}}
\newcommand{\Il}{I(\lambda)}

\newcommand{\imw}{I(\lambda)}
\newcommand{\Imw}{I(\lambda)}
\newcommand{\Hgl}{H(\Ln_{\omega_i}^-)_{\Lg}}
\newcommand{\Hgw}{H(\Ln_{\omega_i}^-)_{\Lg}}
\newcommand{\Unp}{U(\mathfrak{n}^+)}
\newcommand{\Un}{U(\mathfrak{n}^-)}

\newcommand{\w}{\omega}
\newcommand{\Snml}{S(\mathfrak{n}_{\lambda}^{-})}
\newcommand{\Lnml}{\mathfrak{n}_{\lambda}^{-}}
\newcommand{\bps}{\overline{\bp}}
\newcommand{\ail}{\alpha_{i_l,j_l}}
\newcommand{\aim}{\alpha_{i_m,j_m}}
\newcommand{\air}{\alpha_{i_r,j_r}}
\newcommand{\es}{\mathrm{es}}

\newcommand{\mf}{\mathfrak}

\begin{document}

\title[PBW filtration: FFL modules via Hasse diagrams]{PBW filtration: Feigin-Fourier-Littelmann modules via Hasse diagrams}
\author{Teodor Backhaus and Christian Desczyk}
\address{
%Teodor Backhaus:\newline
Mathematisches Institut, Universit\"at zu K\"oln,\newline
Weyertal 86-90, D-50931 K\"oln, Germany}
\email{tbackha@math.uni-koeln.de}
\email{cdesczyk@math.uni-koeln.de}
%\address{Christian Desczyk:\newline
%Mathematisches Institut, Universit\"at zu K\"oln,\newline
%Weyertal 86-90, D-50931 K\"oln, Germany}
%\email{cdesczyk@math.uni-koeln.de}

%\thanks{The work of T.B. was partially supported by the DFG Priority Program SPP 1388 on Representation Theory}

\subjclass[2010]{}
\begin{abstract}
We study the PBW filtration on the irreducible highest weight representations of simple complex finite-dimensional Lie algebras. This filtration is induced by the standard degree filtration on the universal enveloping algebra. For certain rectangular weights we provide a new description of the associated graded module in terms of generators and relations. We also construct a basis parametrized by the integer points of a normal polytope. The main tool we use is the Hasse diagram defined via the standard partial order on the positive roots. As an application we conclude that all representations considered in this paper are Feigin-Fourier-Littelmann modules.
\end{abstract}

\maketitle \thispagestyle{empty}
\input{PBW-Basis_Introduction}
\input{PBW-Basis_PBW_filtration}
\input{PBW-Basis_helpful}
\input{PBW-Basis_Spanning}
\input{PBW-Basis_Types}
\input{PBW-Basis_linearindependence}
\include{PBW-Basis_Appendix}
\input{PBW-Basis_acknowledments}

% References
%%%%%%%%%%%%%%%%%%%%%%%%%%%%%%%%%%%%%%%%%%%%%%%%%%%%%%%%%%%%%%%%%%%
%\bibliographystyle{alpha}
%\bibliography{biblistfusion}
\def\cprime{$'$}

%biblistfusion

%}
\end{document}

%% file: PBW-Basis_Introduction.tex
\section*{Introduction}
We recall briefly the construction of the PBW filtration. We consider a simple complex finite-dimensional Lie algebra $\g$ and a triangular decomposition ${\g=\mathfrak{n}^+\oplus \mathfrak{h}\oplus \mathfrak{n}^-}$. We denote by $\Vl$ the irreducible finite-dimensional module of highest weight $\lambda$ and by $\vl$ a highest weight vector, then we have $\Vl=\unm\vl$. The degree filtration $\unm_s$ on  the universal enveloping algebra $\unm$ over $\Ln^-$ is defined by: 
\begin{equation*}
\unm_s= \mathrm{span}\{x_1\cdots x_l\mid x_i\in \mathfrak{n}^-,\ l\leq s\}.
\end{equation*}
This filtration induces the PBW filtration on $\Vl$, where the $s$-th filtration component is given by $\Vl_s=\unm_s\vl$. The associated graded space $\Vla$, with respect to the PBW filtration, is a $\Sn$-module generated by $\vl$, where $\Sn$ is the symmetric algebra over $\Ln^-$. Then we have for $\Il\subseteq\Sn$ the annihilator of the generating element: 
\begin{equation*}
\Vla=\Sn\vl\cong\Sn/\Il.
\end{equation*}
There are some natural questions (see also \cite{FFoL11a}):
\begin{itemize}
	\item Is it possible to describe $\Vla$ explicitly as a $\Sn$-module, i.e. is it possible to describe the generators of the ideal $\Il$?
	\item Is it possible to find an explicit combinatorial description of a monomial basis of $\Vla$?
\end{itemize}
%For us that means that there is a explicit combinatorial description of the basis. 
%If the basis is parametrized by the integer points of a normal polytope we will call this basis a \textit{Feigin-Fourier-Littelmann} or just FFL basis of $\Vla$ and the module an FFL module.\\
We will call such a basis a \textit{Feigin-Fourier-Littelmann} or just FFL basis and $\Vla$ a FFL module, if the bases of $V(m\lambda)^{a}$, $m\in\Z_{\geq 0}$ are parametrized by the integer points of a normal polytope $P(m)$. \\ For both questions there is a positive answer in the cases of $\mathfrak{sl}_n$ and $\mathfrak{sp}_{2n}$ for arbitrary dominant integral weights
%\Lg$ of type $\mathtt{A_n}$, $\mathtt{C_n}$ with $\lambda\in P^+$ and $n\in\N$ arbitrary
(see \cite{FFoL11a} and \cite{FFoL11b}). Further the second question is positively answered for $\mathtt{G_2}$ (see \cite{G11}). In this paper we focus on certain rectangular weights and prove the following theorem:
\newpage
\begin{thmon} Let $\Lg$ be a simple complex finite-dimensional Lie algebra and ${\lambda = m \omega_i,}$ $ m \in \Z_{\geq 0}$ be a rectangular weight, where $\Lg$ and $\w_i$ appear in the same row of Table \ref{introcases}. Further let $V(\lambda)^a \cong S(\mf{n}^-) / I(\lambda).$ Then there is a positive answer for both questions above, in particular:
\begin{itemize}
	\item $I(\lambda) = S(\mf{n}^-)\left( \Unp \circ \mathrm{span} \{f_{\beta}^{\langle\lambda,\beta^{\vee}\rangle + 1} \mid \beta \in \rrp\}\right).$
	\item $V(\lambda)^a$ is a FFL module.
\end{itemize}
Here we denote with $\Delta_+$ the set of positive roots of $\Lg$. 
\begin{center}
\begin{minipage}{\linewidth}
%\captionof{table}{}
%    \label{fundamentalweights}
\centering
$\begin{array}{|c|c||c|c|}
\hline
\text{Type of $\Lg$}&\text{weight $\omega$}&\text{Type of $\Lg$}&\text{weight $\omega$}\\
\hline\hline
\mathtt{A_n}&\w_k,\ 1\leq k\leq n&\mathtt{E_6}&\w_1,\ \w_6\\
\hline
\mathtt{B_n}&\w_1,\ \w_n&\mathtt{E_7}&\w_7\\
\hline
\mathtt{C_n}&\w_1&\mathtt{F_4}&\w_4\\
\hline
\mathtt{D_n}&\w_1,\ \w_{n-1},\ \w_n&\mathtt{G_2}&\w_1\\
\hline
\end{array}$
\captionof{table}{Solved cases}
    \label{introcases}
  \end{minipage}
\end{center}
%To be precise: For a simple Lie algebra $\Lg$ listed in the Table \ref{introcases} we solved the problem for all rectangular weights of the form $\lambda=m\omega_i$, where $m\in\Z_{\geq 0}$ and the fundamental weight $\omega_i$ occurs in corresponding row of Table \ref{introcases}.
\end{thmon}
\begin{remon} The Theorem above implies the existence of a normal polytope $P(m\w_i)$ such that the integer points $S(m\w_i)$ parametrize a basis of $V(m\w_i)$. This polytope is the $m$-th Minkowski sum of the polytope $P(\omega_i)$ corresponding to $V(\w_i)$. In general this is not true for different fundamental weights, because the number of integer points in the Minkowski sum is too small. For example in the case of $\mathfrak{g} = \mathfrak{sl}_5$, we have $|(P(\w_1) + P(w_2) + P(\w_3) + P(\w_4)) \cap \Z_{\geq 0}^N| = 1023$ and $\dim V(\w_1+\w_2+\w_3+\w_4) = 1024$.
 %the Minkowski sum of the polytopes corresponding to $\w_i$, $i=1,2,3,4$ does not provide a basis of $V(\w_1+\w_2+\w_3+\w_4)$.
\end{remon}
\begin{remon}
The bases obtained in \cite{FFoL11a}, which were conjectured by Vinberg (see \cite{V05}) and obtained in \cite{FFoL11b} are different from our bases. This is due to a different choice of the total order on the monomials in $\snm$. As a consequence the induced normal polytopes are also different. Nevertheless in the cases $(\mathtt{A_n},\w_k)$ the corresponding projective toric varieties are isomorphic. In contrast, these are in general not isomorphic to the toric varieties corresponding to Gelfand-Tsetlin polytopes investigated in \cite{GL} and \cite{KM05}.
%\begin{thmon} 
%There is a positive answer for both questions from above in the following cases:
%\begin{center}
%\begin{minipage}{\linewidth}
%%\captionof{table}{}
%%    \label{fundamentalweights}
%\centering
%$\begin{array}{|c|c||c|c|}
%\hline
%\text{Type of $\Lg$}&\text{weight $\omega$}&\text{Type of $\Lg$}&\text{weight $\omega$}\\
%\hline\hline
%\mathtt{A_n}&\w_k,\ 1\leq k\leq n&\mathtt{E_6}&\w_1,\ \w_6\\
%\hlinee
%\mathtt{B_n}&\w_1,\ \w_n&\mathtt{E_7}&\w_7\\
%\hline
%\mathtt{C_n}&\w_1&\mathtt{F_4}&\w_4\\
%\hline
%\mathtt{D_n}&\w_1,\ \w_{n-1},\ \w_n&\mathtt{G_2}&\w_1\\
%\hline
%\end{array}$
%\captionof{table}{Solved cases}
    %\label{introcases}
  %\end{minipage}
%\end{center}
%To be precise: For a simple Lie algebra $\Lg$ listed in the Table \ref{introcases} we solved the problem for all rectangular weights of the form $\lambda=m\omega_i$, where $m\in\Z_{\geq 0}$ and the fundamental weight $\omega_i$ occurs in corresponding row of Table \ref{introcases}.
%\end{thmon}
\end{remon}
We explain briefly the methods used in our paper. Our main tool is the Hasse diagram of $\Lg$ given by the standard partial order on the positive roots of $\Lg$. We associate to this directed graph a normal polytope $P(\lambda)=P(m\w_i)\subset \R^N_{\geq 0}$ via the directed paths. If the Hasse diagram satisfies certain properties, the set of integer points $S(\lambda)=P(\lambda)\cap \Z^N_{\geq 0}$ parametrizes a FFL basis of $\Vla$. So we reduce 
%the description of $V(\lambda)^a$ and $V(\lambda)$ in terms of generators and relations and in terms of a basis
the questions above
to the combinatorics of the Hasse diagram and provide a general procedure which uses the structure of the Hasse diagram. As an important application we show that the modules $V(m\omega_i), m \in \Z_{\geq 0}$ are FFL modules, where $\omega_i$ appears in Table \ref{introcases}.\\
Except for the cases listed in Table \ref{introcases} it is much more involved to obtain a polytope which parametrizes a FFL basis. Even in the cases ($\mathtt{B_n,\w_1}$), ($\mathtt{F_4,\w_4}$) and ($\mathtt{G_2,\w_1}$) we have to change the Hasse diagram slightly, to be able to apply our procedure.\\
The property of being a FFL module implies some nice consequences. For example the corresponding degenerate flag varieties are normal and Cohen-Macaulay. Further there is an explicit representation theoretical description of the corresponding homogeneous coordinate rings. Another important property is the interpretation of the describing polytopes as Newton-Okounkov bodies (see \cite{FFoL13a} and for more details on Newton-Okounkov bodies see \cite{KK} and \cite{HK}). 
\vspace{2mm}

In the recent years it turned out that the PBW theory has a lot of connections to many areas of representation theory. For example to the geometric representation theory: Schubert varieties (\cite{IL14}, \cite{CLaL14}) and degenerate flag varieties (\cite{FFiL}, \cite{F1}, \cite{F2}, \cite {CIFR12} and \cite{C}). Further there are connections to combinatorial representation theory for example to Schur functions (\cite{Fo}), combinatorics of crystal basis (\cite{K1}, \cite{K2}) and Macdonald polynomials (\cite{CF13}, \cite{FM14}). A purely combinatorial research on the FFL polytopes can be found in \cite{ABS}. A general formular for the maximal degree of $V(\lambda)^a$ for arbitrary dominant integral weights $\lambda$ is provided in \cite{BBDF14}. %Moreover that is a very active area of research for example 
\vspace{2mm}

Our paper is organized as follows:\\In Section 1 we introduce the constructions and tools we use. Furthermore we state our Main Theorems and provide the connection to FFL modules. In Section 2 we prove that all polytopes considered in this paper are normal. Sections 3, 4 and 5 are devoted to the proof of our Main Theorems. In Section 4 we calculate explicitly FFL bases of $V(\omega)$ for all cases listed in Table \ref{introcases}. Finally in the Appendix we give some explicit examples of Hasse diagrams and normal polytopes.  

%% file: PBW-Basis_PBW_filtration.tex
\section{PBW Filtration}\label{test}
\subsection{Definitions}\label{Definitions}
Let $\g$ be a simple complex finite-dimensional Lie algebra and let $\g=\mathfrak{n}^+\oplus \mathfrak{h}\oplus \mathfrak{n}^-$ be a triangular decomposition.\\ For a dominant integral weight $\lambda$ we denote by $\Vl$ the irreducible $\g$-module with highest weight $\lambda$. We fix a highest weight vector $\vl\in\Vl$. Then we have $\Vl=\unm\vl$. The degree filtration $\unm_s$ on $\unm$ is defined by: 
\begin{equation}\label{degfilt}
\unm_s= \mathrm{span}\{x_1\cdots x_l\mid x_i\in \mathfrak{n}^-,\ l\leq s\}.
\end{equation}
In particular, $\unm_0= \C\mathds{1}$. So we have an increasing chain of subspaces:\\ $\unm_0\subseteq\unm_1\subseteq\unm_2\subseteq\dots $. The filtration (\ref{degfilt}) induces a filtration on $\Vl$: $\Vl_s=\unm_s\vl$, the PBW filtration.\\ We consider the associated graded space $\Vla$ of $\Vl$ defined by:
\begin{equation}\label{vla}
\Vla=\bigoplus_{s  \, \in \,  \Z_{\geq 0}} \Vl_s/\Vl_{s-1},\ \Vl_{-1}=\{0\}.
\end{equation}
Let $\Delta_+\subset \Lh^*$ be the set of positive roots of $\Lg$ and $\Phi_+=\{\alpha_1\dots,\alpha_n \}\subset\Delta_+$ the subset of simple roots, where $n\in \N$ is the rank of the Lie algebra $\Lg$. Further we denote by $f_{\beta}\in \Ln^-$ the root vector corresponding to $\beta\in\Delta_+$. Let $\langle \lambda,\beta^{\vee}\rangle=\frac{2(\lambda,\beta)}{(\beta,\beta)}$, where $\beta^{\vee} = \frac{2\beta}{(\beta,\beta)}$ is the coroot of $\beta$ and $(\cdot,\cdot)$ is the Killing form. We define
\begin{equation*}
\Ln_{\lambda}^{-} := \text{span}\{f_{\beta} \mid \langle \lambda,\beta^{\vee}\rangle \geq 1\} \subset \Ln^{-}.
\end{equation*}
Throughout this paper we focus on certain rectangular weights $\lambda=m\omega_i$, $m\in \Z_{\geq 0}$ (see Table \ref{introcases}).\\ 
%For most of the results we have restrictions see Table \ref{fundamentalweights}.
Let $\beta = \sum_{j=1}^n n_j \al_j,\ n_j \in \Z_{\geq 0}$ be a positive root with $n_i \geq 1$. Then we have for the coroot $\beta^{\vee}= \sum_{j=1}^n n_j^{\vee} \al_j^{\vee}$: $n_i^{\vee} \geq 1$.  Conversely starting with a coroot $\beta^{\vee}$, with $n_i^{\vee}\geq 1$ we have for the corresponding positive root $\beta$: $n_i \geq 1$. Hence, independent of the choice of $m\geq 1$: $$\Ln_{\w_i}^-=\Ln_{m\w_i}^{-} \subset \Ln^{-}$$ is the Lie subalgebra spanned by those root vectors $f_\beta$, where $\al_i$ is a summand of $\beta$.\\ From the PBW-Theorem we get $U(\Ln_{\lambda}^-)^a = S(\mathfrak{n}_{\lambda}^{-}) = \C[f_\beta \mid \langle \lambda,\beta^{\vee}\rangle \geq 1]$, where $S(\mathfrak{n}_{\lambda}^{-})$ is the symmetric algebra over $\Ln_{\lambda}^-$.
\begin{rem}
$(i)$ We have $V(\lambda) = U(\Ln_{\lambda}^-)v_{\lambda}.$ The action of $U(\Ln_{\lambda}^-)$ on $\Vl$ induces the structure of a $S(\mathfrak{n}_{\lambda}^{-})$-module on $\Vla$ and 
\begin{equation}\label{cyclic}
\Vla=S(\Ln^{-})\vl = S(\mathfrak{n}_{\lambda}^{-})\vl.
\end{equation}
$(ii)$ The action of $\unp$ on $\Vl$ induces the structure of a $\unp$-module on $\Vla$. Note for $e_{\al} \in \Ln^+ \hookrightarrow U(\Ln^+), f_{\beta} \in \mathfrak{n}_{\lambda}^{-} \hookrightarrow S(\mathfrak{n}_{\lambda}^{-})$, $[e_\al,f_\beta]$ is not in general an element of $S(\mathfrak{n}_{\lambda}^{-})$, but for $f_{\nu} \in \Sn \setminus S(\mathfrak{n}_{\lambda}^{-})$ we have $f_{\nu}v_\lambda = 0$. That follows from the well known description (see \cite{Hum72}) of $V(\lambda)$:
\begin{equation}\label{wkd}
V(\lambda)=\Un/\langle f_{\beta}^{\langle \lambda,\beta^{\vee}\rangle + 1} \mid \beta\in \Delta_+\rangle.
\end{equation}
\end{rem}
%\begin{proof}
%We know that $\unm^{a}\simeq \snm$ (see (xxx) Carter). 
%\end{proof}
Equation (\ref{cyclic}) shows that $\Vla$ is a cyclic $\Snml$-module and hence there is an ideal $I_{\lambda} \subseteq \Snml$ such that $\Vla\simeq\Snml/ I_{\lambda}$, where $I_{\lambda}$ is the annihilating ideal of $\vl$. We have therefore the following projections:
\begin{equation*}
S(\Ln^-) \rightarrow S(\Ln^-)/\left< f_{\beta} \mid \langle \lambda,\beta^{\vee}\rangle = 0\right> = \Snml \rightarrow \Snml/I_{\lambda}.
\end{equation*}
Hence, although we work with $\Lnml$, we actually consider $\Ln^-$-modules. So our aims in this paper are 
\begin{itemize}
	\item To describe $\Vla$ as a $\Snml$-module, i. e. describe explicitly generators of the ideal $I_{\lambda}.$
	\item To find a basis of $\Vla$ parametrized by integer points of a normal polytope $P(\lambda)$ (see $\eqref{poly}$).
\end{itemize}
%\begin{rems}\mbox{} 
%\begin{itemize}
	%\item  A posteriori we know that $\{f_{\beta_i}^{\langle\lambda,\beta_i^{\vee}\rangle+1}\vl\ |\ \beta_i\in \Delta_+ \}\subset \Il$ (see (xxx)). 
	%\item We have: $\Sn\simeq \C[f_{\beta_i}|\ \beta_i\in \Delta_+]$ (see (xxx)).
%\end{itemize}
%\end{rems}
To achieve these goals we have to introduce further terminology.
We denote the set of positive roots associated to $\Ln_{\lambda}^-$ by
\begin{equation}\label{relroots}
\rrl=\{ \beta \in \Delta_+|\ \langle \lambda,\beta^{\vee}\rangle \geq 1\}=:\{\beta_1,\dots,\beta_N \} \subseteq \Delta_+,\ |\rrl|=N\in \Z_{\geq 0}.
\end{equation}
\begin{exa}
We write $(r_1,r_2,\dots,r_n)$ for the sum: $\sum_{k=1}^n r_k\alpha_k$. Let $\Lg$ be of type $\mathtt{A_4}$ and $\lambda=\omega_3$, the third fundamental weight. Then we have:
\begin{align*}
%\Delta_+=&\{(1,1,1,1),(0,1,1,1),(1,1,1,0),(0,0,1,1),(0,1,1,0),(1,1,0,0),(0,0,0,1),\\
%&(0,0,1,0),(0,1,0,0),(1,0,0,0) \}\\
\Delta_+^{\omega_3}= \{&\beta_1=(1,1,1,1),\beta_2=(0,1,1,1),\beta_3=(1,1,1,0),\\
&\beta_4=(0,0,1,1),\beta_5=(0,1,1,0),\beta_6=(0,0,1,0)\}\subset\Delta_+.
\end{align*}
%So the set $\Delta_+^{\omega_3}$ consits of all positive roots with $a_3>0$.
\end{exa}
We choose a total order $\prec$ on $\rrl$:
\begin{equation}\label{totalorder}
    \beta_1\prec\beta_2\prec\cdots\prec\beta_{N-1}\prec\beta_N.
\end{equation}
We assume that this order satisfies the following conditions:
\begin{itemize}
\item[(i)] Let $\geq$ be the standard partial order on the positive roots, then
\begin{equation*} 
\beta_i > \beta_j \Rightarrow \beta_i \prec \beta_j.
\end{equation*}
\item[(ii)]  Let $\beta_i=(r_1,\dots,r_n), \beta_j = (t_1,\dots,t_n)$ and we
define the height as the sum over these entries: 
$\height(\beta_i) = \sum_{i = 1}^{n} r_i, \height(\beta_j) = \sum_{i = 1}^n t_i$. Then 
\begin{equation*} 
\height(\beta_i) > \height(\beta_j) \Rightarrow  \beta_i \prec \beta_j.
\end{equation*}
\item[(iii)] If $\beta_i$ and $\beta_j$ are not comparable in the sense of $(i)$ and $(ii)$, then\\ $\beta_i \prec \beta_j \Leftrightarrow$  $\beta_i$ is greater than $\beta_j$ lexicographically, i.e. there exists $1 \leq k \leq n,$ such that $r_k > t_k$ and $r_i = t_i$ for $1 \leq i < k$.
\end{itemize}
\begin{rem}\label{fixedbasis} The explicit order of the roots depends on the Lie algebra and the chosen weight, see Section $\ref{uebergreifend2}$.
But in all cases considered in this paper we have $\beta_1=\theta$, the highest root of $\Lg$ and $\beta_N$ is the simple root $\al_i$. 
%In $(iii)$ it is allowed to choose a different ordering of the simple roots.
%with $\langle \w_i,\al^{\vee}\rangle = 1.$
\end{rem}
In order to make our equations more readable we write for $1\leq i\leq N$: $f_i=f_{\beta_i}$ and $s_i=s_{\beta_i}$. We associate to the multi-exponent $\bs=(s_{i})_{i=1}^{N}\in\Z_{\geq 0}^{N}$ the element    
\begin{equation}\label{eleinsnm}
	f^{\bs}=\prod_{i=1}^{N}f_{i}^{s_{i}}\in \Snml,
\end{equation}
and define the degree of $f^{\bs}\vl \neq 0$ in $\Vla$ by $\mathrm{deg}(f^{\bs}\vl)=\mathrm{deg}(f^{\bs}) = \sum_{i=1}^{N}s_{i}$, or $\mathrm{deg}(f^{\bs}\vl) = 0$ if $f^{\bs}\vl = 0$.
We extend $\prec$ to the homogeneous lexicographical total order on the monomials of $\Snml$ (resp. multi-exponents).\\
Let $\bs,\bt \in \Z_{\geq 0}^N$ be two multi-exponents. We say  $f^\bs \succ f^\bt$ or $\bs \succ\bt$ if
\begin{itemize}
\item $\mathrm{deg}(f^\bs) > \mathrm{deg}(f^\bt)$ or
\item $\mathrm{deg}(f^\bs) = \mathrm{deg}(f^\bt)$ and $\exists \, 1\leq k\leq N:(s_k>t_k) \wedge \,\forall \, k<j\leq N:(s_j=t_j)$.
\end{itemize}
%, induced by ``$\succ$''.(we will use again ``$\succ$'' as relation symbol):
%\begin{equation}
%f_{\beta_1}^{s_1}\cdots f_{\beta_N}^{s_N}\succ f_{\beta_1}^{p_1}\cdots f_{\beta_N}^{p_N} \Leftrightarrow \begin{cases} \sum_{i=1}^N s_i > \sum_{i=1}^N p_i,\ \text{or} \\
%\sum_{i=1}^N s_i = \sum_{i=1}^N p_i,\ \text{and}\ \exists\ 1\leq k\leq N:\\
%k<j\leq N:s_j=p_j \wedge s_k>p_k.
%\end{cases}
%\end{equation}
For example: $f_{1}^1f_{2}^2f_{3}^0\prec f_{1}^2f_{
2}^0f_{3}^1\prec f_{1}^1f_{2}^0f_{3}^2$.
\begin{rem}\label{degreeinvariance}
Because the action of $\Ln^+$ on $\Vl$ is induced by the adjoint action, we 
know that $\Vl_s, s\in\Z_{\geq 0}$ is stable under the action of $\Ln^+$: for $e\in\Ln^+$ and $x_1\cdots x_s\vl\in\Vl_s$ we have
\begin{equation*}
e.x_1\cdots x_s\vl= \sum_{i\,=\,1}^s x_1\cdots x_{i-1}[e,x_i]x_{i+1}\cdots x_s\vl\in\Vl_s. 
\end{equation*}
Hence $\Vl_s$ is a $\Unp$-module. So for $f^{\bt}\vl$ in $\Vla = \bigoplus_{s \geq 0} V(\lambda)_s / V(\lambda)_{s-1}$ we have $\mathrm{deg}(u f^\bt\vl)\in\{0,\mathrm{deg}(f^\bt\vl)\}$ for all $u\in\Unp$.
\end{rem}
The next Lemma is devoted to give a better understanding of the module $\Vla$, but we will not need it to prove our main statements.
%As vector spaces $\Vl$ and $\Vla$ are isomorphic, so we know that $\dim\Vla<
%\infty$. Hence there exists a maximal $r\in\Z_{\geq 0}$ with $\Vla_r:=\Vl_{r}/
%\Vl_{r-1}\neq \{0\}$.
\begin{lem}
Let $f^{\bm}\in\snm$ with $f^{\bm}\vl\neq 0$ in $\Vla$ and weight $\mathrm{wt}(f^{\bm})=\lambda-w_0(\lambda)$, where $w_0$ is the longest element in the Weyl group of $\Lg$ and $w_0(\lambda)$ is the lowest weight of $\Vl$. Then
\begin{equation*}
\mathrm{deg}(f^{\bn})\leq\mathrm{deg}(f^{\bm}),\ \forall f^{\bn}\vl\neq 0\in\Vla.
\end{equation*}
\end{lem} 
\begin{proof}
%%We know $\Vl=\Un\vl$ and $\dim\Vl<\infty$. So there is
Let $v_{w_0(\lambda)}$ be a lowest weight vector  
%with  $\mathrm{wt}(v_{w_0(\lambda)})=w_0(\lambda)$ 
such that:
\begin{equation*}
\Vl=\Unp v_{w_0(\lambda)}.
\end{equation*}
Hence we can interpret $\Vl$ as a lowest weight module. The lowest weight $\w_0(\lambda)$ is in the Weyl group orbit of $\lambda$, thus $\dim\Vl_{w_0(\lambda)}=1=\dim\Vl_{\lambda}$. So there is a minimal $s\in\Z_{\geq 0}$ such that: $\Vl_{w_0(\lambda)}\subseteq \Vl_s$. Further there exists a scalar $c\in \C$ with $f^{\bm}\vl=cv_{w_0(\lambda)}$.\\ For an arbitrary element $f^{\bn}\vl\neq 0\in \Vla$ we fix the order of the factors to obtain $f^{\bn}\vl\in\Vl$. Then there exists an element $x\in \Unp$ such that: $f^{\bn}\vl=x(f^{\bm}\vl)$. This implies with Remark \ref{degreeinvariance}: $\mathrm{deg}(f^{\bn})\leq\mathrm{deg}(f^{\bm})$.
\end{proof}
Associated to the set $\mathfrak{n}_{\lambda}^-$ we define a directed graph $H(\mathfrak{n}_{\lambda}^-)_{\mathfrak{g}}:=(\rrl
,E)$. The set of vertices is given by $\rrl$ and the set of edges $E$ 
is constructed as follows:
\begin{equation*}
	\forall\, 1\leq i,j\leq N:\ (\beta_i\xrightarrow{k}\beta_j) \in E 
\Leftrightarrow \exists\ \alpha_k\in \Phi_+:\ \beta_i-\beta_j=\alpha_k.
\end{equation*}
We call this directed graph \textit{Hasse diagram} of $\Lg$ associated to $
\lambda$. For our further considerations $H(\Ln_{\lambda}^-)_{\Lg}$ is the most important tool. 
\begin{exa}\label{exaHdA4} The Hasse diagram $H(\mathfrak{n}_{\w_3}^-)_{\mathfrak{sl}_5}$ is given by:\end{exa}

\begin{center}
\begin{tikzpicture}
  \node (1) at (0,0) {$\beta_1$};
  \node (2) at (-1,-1) {$\beta_2$};
  \node (3) at (1,-1) {$\beta_3$};
  \node (4) at (-2,-2) {$\beta_4$};
  \node (5) at (0,-2) {$\beta_5$};
	\node (6) at (-1,-3) {$\beta_6$};
	\node (1) at (0,0) {$\beta_1$};
%  \node (2) at (0,-2) {$\beta_2$};
%  \node (3) at (2,0) {$\beta_3$};
%  \node (4) at (0,-4) {$\beta_4$};
%  \node (5) at (2,-2) {$\beta_5$};
%	\node (6) at (2,-4) {$\beta_6$};
	\node (7) at (5,-0.35) {$\beta_1=(1,1,1,1)$};
	\node (8) at (5,-0.85) {$\beta_2=(0,1,1,1)$};
	\node (9) at (5,-1.35) {$\beta_3=(1,1,1,0)$};
	\node (10) at (5,-1.85) {$\beta_4=(0,0,1,1)$};
	\node (11) at (5,-2.35) {$\beta_5=(0,1,1,0)$};
	\node (12) at (5,-2.85) {$\beta_6=(0,0,1,0)$};
   \draw[->] (1) edge node[left, pos=0.2] {\tiny{1}} (2);
   \draw[->] (1) edge node[right, pos=0.2] {\tiny{4}} (3);
   \draw[->] (2) edge node[left, pos=0.2] {\tiny{2}} (4);
   \draw[->] (3) edge node[left, pos=0.2] {\tiny{1}} (5);
   \draw[->] (2) edge node[right, pos=0.2] {\tiny{4}} (5);
   \draw[->] (4) edge node[right, pos=0.2] {\tiny{4}} (6);
   \draw[->] (5) edge node[left, pos=0.2] {\tiny{2}} (6);  
\end{tikzpicture}
\end{center}
We define an ordered sequence of roots in $\rrl$: $(\beta_{i_1},\dots,\beta_{i_r})$ with $\beta_{i_j} \prec \beta_{i_{j+1}}$ to be a \textit{directed path} from $\beta_{i_1}$ to $\beta_{i_r}$. 
%Further we define the \textit{support of a directed path} $(\beta_{i_1},\dots,\beta_{i_r})$ as follows:
%Let us denote by $W=(\beta_{i_1},\dots,\beta_{i_r})$ a \textit{directed path from $\beta_{i_1}$ to $\beta_{i_r}$} in $\Hgl$. For our purposes we want to allow the trivial path $(\emptyset)$ and any ordered subset of a directed path. In the example above $(\beta_1, \beta_2, \beta_4, \beta_8)$ and $(\beta_1, \beta_2, \beta_8)$ are directed paths. We define the support of a directed path $U$ 
%\begin{equation*}
%\mathrm{supp}(\beta_{i_1},\dots,\beta_{i_r})=\{\beta\in\rrl \mid \ \beta\ \textrm{occurs in the path } (\beta_{i_1},\dots,\beta_{i_r})\}.
%\end{equation*}
%\begin{equation*}
%\mathrm{supp}(\beta_{i_1},\dots,\beta_{i_r})=\{\beta\in\rrl \mid \exists \, 1 \leq j \leq r: \beta = \beta_{i_j} \}.
%\end{equation*}
\begin{rem} For our purposes we want to allow the trivial path $(\emptyset)$ and any ordered subsequence of a directed path to be a directed path again. So in Example \ref{exaHdA4} $(\beta_1, \beta_2, \beta_4, \beta_6)$ and $(\beta_1, \beta_2, \beta_6)$ are two possible directed paths. 
\end{rem}
In general it is possible that two edges in $H(\Ln_{\lambda}^-)_{\Lg}$, one ending in a root $\beta$ and one starting in $\beta$, have the same label: 
\begin{equation*}
\gamma\xrightarrow{k}\beta\xrightarrow{k}\delta.
\end{equation*}
We call this construction a \textit{k-chain (of length 2)}.\\[2mm]
%\begin{lem}
%Every subgraph $\Hgl$ for arbitrary $\Lg$ and weight $\lambda$ with $\langle \lambda,\theta^{\vee}\rangle =1$ has no chains. Except for $\Lg=\mathtt{F_4}, \lambda=\omega_4$.
%\end{lem} 
%\begin{proof}
%see section (xxx)
%\end{proof}
Associated to $H(\Ln_{\lambda}^-)_{\Lg}$ we construct two subsets $D_{\lambda},\overline{D}_{\lambda} \subset \prrl$ of the power set of $\rrl$: For $\bp \in \prrl$ we define
\begin{equation}\label{dyckpaths}
\bp \in D_{\lambda} :\Leftrightarrow \bp = \{\beta_{i_1},\dots,\beta_{i_r}\},  
\end{equation}
for a directed path $(\beta_{i_1},\dots,\beta_{i_r})$ in $H(\Ln_{\lambda}^-)_{\Lg}$.
So from now on by $\eqref{dyckpaths}$ we interpret $\bp \in D_{\lambda}$ as a directed path in $H(\Ln_{\lambda}^-)_{\Lg}$.
%by the following relation: Let $\gamma,\delta\in \rrl$, then 
%\begin{equation}\label{defD1}
%\exists\ \bp\in D: \gamma,\delta\in \bp \Leftrightarrow \gamma-\delta= \begin{cases} \sum_{i=1}^n l_i\alpha_i &, l_i\in \Z_{\geq 0}\\ \sum_{i=1}^n l_i\alpha_i &, l_i\in \Z_{\leq 0} \end{cases}.
%\end{equation}
\begin{rem}\label{defdyck}
Let $\beta_i,\beta_j \in \rrl$ be arbitrary. Then there exist a $\bp \in D_{\lambda}$ with $\beta_i,\beta_j \in \bp$ if and only if $\beta_i-\beta_j$ or $\beta_j - \beta_i$ is a non-negative linear combination of simple roots.
%We call $D_{\lambda}$ the set of $Dyck$ $paths$. We will provide a more general definition of Dyck paths in Section $\ref{helpful}$.
\end{rem}
\begin{rem}\label{whydyck}
A staircase walk from (0,0) to (n,n) beyond the diagonal in a $n \times n$-lattice is a called Dyck path. In the general $\mathtt{A_n}$-case (\cite{FFoL11a}) the constructed directed paths are Dyck paths in this sense. To be consistent with their notation we call our directed paths $D_{\lambda}$ also Dyck paths.
%We call $D_{\lambda}$ the set of $Dyck$ $paths$ but these objects are not the same objects known from combinatorics. We will provide a more general definition of Dyck paths in Section $\ref{helpful}$.\\ The notion of Dyck paths is motivated by the fact that $D_{m\w_i}=D_{\w_i}$. In the $\mathtt{A_n}$-case the Hasse diagram is of rectangular shape. Then the set $D_{\w_i}$ consists of all possible directed paths in the Hasse diagram $\Hgl$.
\end{rem}
Further we define the set of $co$-$chains$ by
\begin{equation}\label{dstrich}
\overline{D}_{\lambda} := \{ \bps \in \prrl \mid |\bps \cap \bp| \leq 1,\, \forall\ \bp \in D_\lambda\}.
\end{equation}
%Further we define $\overline{D}_{\lambda}$ to be the set containing sets of the form 
%\begin{equation*}
%\bps = \{ \beta_{i_1},\dots,\beta_{i_r} \in \rrl \mid \forall \bp \in D_{\lambda}: \beta_{i_k},\beta_{i_j}\notin \bp, \forall 1 \leq k \neq j \leq r\}.
%\end{equation*}
%In other words $\overline{D}_{\lambda}$ describes sets of roots $\bps$ such that for all $\bp \in D_{\lambda}: | \bps \cap \bp | \leq 1$. We call $\bps\in\overline{D}_{\lambda}$ a $co$-$chain$.
If necessary we use an additional index $\dsl^{\,\text{type of }\Lg}$, to distinguish which type of $\Lg$ we consider. We want to consider the integral points of a polytope which is connected to $D_{\lambda}$ in a very natural way. Fix $\lambda=m\w_i$, with $m\in\Z_{\geq 0}$. Let
\begin{equation}\label{poly}
P(m\w_i)=\{ \textbf{x}\in \mathbb R_{\geq 0}^N \mid \sum_{\beta_j\,\in\, \bp}x_j\leq m,\hspace{2mm} \forall\ \bp\in D_{\w_i} \},
\end{equation}
be the \textit{associated polytope} to $D_{\w_i}$. Denote by $S(m\w_i)$ the integer points in $P(m\w_i)$: $S(m\w_i)=P(m\w_i)\cap \mathbb Z_{\geq 0}^N$. We define the map 
\begin{equation*}
\supp_1:S(\w_i) \rightarrow \mathcal{P}(\Delta_{+}^{\w_i}), \, \supp_1(\bs) = \{\beta_j \mid s_j > 0\}.
\end{equation*}
%by $\supp_1(\bs) = \{\beta_i \mid s_i > 0\}.$ 
For $\bs \in S(\w_i)$ we have with $\eqref{dstrich}$ immediately $\supp_1(\bs) \in \overline{D}_{\w_i}$. Conversely every $\bps \in \overline{D}_{\w_i}$ has a non-empty pre-image. 
With $\bs \in \{0,1\}^N$ we conclude that $\supp_1$ is injective and that we have the immediate proposition:
\begin{prop}\label{bijSWDL}
The map $\supp_1:S(\w_i) \rightarrow \overline{D}_{\w_i}$ is a bijection. $\hfill \Box$
\end{prop}
Hence in Section $\ref{uebergreifend2}$ it is sufficient to determine the co-chains in $\Hgl$ to find the elements in $S(\omega_i)$. Now we are able to formulate our main statements.
\subsection{Main statements}\label{main}
Let $\Lg$ be a simple complex finite-dimensional Lie algebra and $\lambda= m\omega_i$ be a rectangular weight, with $\langle \omega_i,\theta^{\vee}\rangle=1$ and $m\in \Z_{\geq 0}$, where $\theta$ is the highest root of $\Lg$. Further we assume that $\Hgw$ has no $k$-chains of length 2. In the following table we list up all cases where these assumptions are satisfied. Additionally in the cases $(\mathtt{B_n,\w_1}), (\mathtt{F_4,\w_4})$ and $(\mathtt{G_2,\w_1})$, we can rewrite $\Hgw$ in a diagram without $k$-chains of length 2:
%This is the case for $\mathtt{A_n}:\omega_i, 1\leq i \leq n$, $\mathtt{B_n}:\omega_n$, $\mathtt{C_n}:\omega_1$, $\mathtt{D_n}:\omega_{n-1},\omega_n$,  $\mathtt{E_6}:\omega_1,\omega_6$, $\mathtt{E_7}:\omega_7$. 
%In addition to this, in the cases $\mathtt{B_n},\mathtt{G_2}:\omega_1$ and $\mathtt{F_4}:\omega_4$, we are able to rewrite $\Hgw$ in a diagram without chains. 
\begin{center}
\begin{minipage}{\linewidth}
%\captionof{table}{}
%    \label{fundamentalweights}
\centering
$\begin{array}{|c|c||c|c|}
\hline
\text{Type of $\Lg$}&\text{weight $\omega_i$}&\text{Type of $\Lg$}&\text{weight $\omega_i$}\\
\hline\hline
\mathtt{A_n}&\w_k,\ 1\leq k\leq n&\mathtt{E_6}&\w_1,\ \w_6\\
\hline
\mathtt{B_n}&\w_1,\ \w_n&\mathtt{E_7}&\w_7\\
\hline
\mathtt{C_n}&\w_1&\mathtt{F_4}&\w_4\\
\hline
\mathtt{D_n}&\w_1,\ \w_{n-1},\ \w_n&\mathtt{G_2}&\w_1\\
\hline
\end{array}$
\captionof{table}{Solved cases}
    \label{fundamentalweights}
  \end{minipage}
\end{center}
%
%\begin{center}
%\begin{minipage}{\linewidth}
%%\captionof{table}{}
%%    \label{fundamentalweights}
%\centering
%$\begin{array}{|c|c|}
%\hline
%\text{Type of $\Lg$}& \text{weigth $\omega$}\\
%\hline\hline
%\mathtt{A_n}&\w_k,\ 1\leq k\leq n\\
%\hline
%\mathtt{B_n}&\w_1,\ \w_n\\
%\hline
%\mathtt{C_n}&\w_1\\
%\hline
%\mathtt{D_n}&\w_1,\ \w_{n-1},\ \w_n\\
%\hline
%\mathtt{E_6}&\w_1,\ \w_6\\
%\hline
%\mathtt{E_7}&\w_7\\
%\hline
%\mathtt{F_4}&\w_4\\
%\hline
%\mathtt{G_2}&\w_1\\
%\hline	
%\end{array}$
%\captionof{table}{Solved cases}
%    \label{fundamentalweights}
%  \end{minipage}
%\end{center}
Let  $I(m\w_i) \subset \snm$ be the ideal such that $V(m\w_i)^{a}=\Sn / I(m\w_i).$
\begin{theoA}\label{theoremA}
\begin{equation*}
I(m\w_i)=\snm\left(\unp\circ\mathrm{span}\{f_{\beta}^{\langle m\w_i,\beta^{\vee}\rangle +1}\mid \beta\in \Delta_+\} \right).
\end{equation*}
\end{theoA}
\begin{proof}
This statement follows by Theorem \ref{thereomAA}. 
%get this statement immediately by the proof of the straightening law in Subcsection $\ref{spanning}$ and from Theorem B.
\end{proof}
\begin{theoB}
$\mathbb{B}_{m\w_i}=\{f^{\bs}v_{m\w_i} \mid \bs\in S(m\w_i) \}$ is a FFL basis of $V(m\w_i)^{a}$.
\end{theoB}
\begin{proof}
In Section \ref{helpful} we show that the polytope $P(m\w_i)$ is normal. By Theorem $\ref{thmspan}$ we conclude that $\mathbb{B}_{m\w_i}$ is a spanning set for $V(m\w_i)^{a}$. After fixing the order of the factors, with Theorem $\ref{theoremBB}$ we have a FFL basis of $V(m\w_i)$. Because this basis is monomial and $V(m\w_i) \cong V(m\w_i)^{a}$ as vector spaces, we conclude that $\mathbb{B}_{m\w_i}$ is a FFL basis of $V(m\w_i)^{a}$.
%In Section $\ref{uebergreifend2}$ we construct the polytopes $P(\lambda)$ and the Hasse diagrams $\Hgl$ explictly. 
%%and show  by direct calculation, that we already have basis for $V(\omega)$, where $\lambda = m\omega, m \in \Z_{\geq 0}$.\\
%In Subsection $\ref{spanning}$ we show that $\mathbb{B}_{\lambda}$ is a generating set for $V(\lambda)^a$ and also for $V(\lambda)$ after fixing the factors in the products $f^{\bs}$. Very important to this proof is the form of the consructed Hasse diagram, for example that there are no chains in $\Hgl$ and the notion of Dyck paths.  With this statement we show for each type by direct calculation, that we already have basis for $V(\omega)$, where $\lambda = m\omega, m \in \Z_{\geq 0}$, again in Section $\ref{uebergreifend2}$.\\
%In particular we show that this basis is given by $f^\bs\vw, \bs \in S(\omega)$. With the results on the Minkowski property from Section $\ref{helpful}$, we show that the constructed polytopes $P(\lambda)$ are normal. In particular we show
%\begin{equation*}
%S(m\omega) = S((m-1)\omega) + S(\omega), m \in \Z_{\geq 1},
%\end{equation*}
%and $+$ is here the Minkowski sum. In Subsection $\ref{linearinde}$ we use the notion of essenial monomials following $FFL$ and the Minkowski sum property to show that $\B_{\lambda}$ is also linear independent and therefore a basis.
%\\
%We proove the spanning property in section (xxx). To proove the linear independence of $\mathbb{B}_{m\omega}$ we show in section 3 (xxx), that the polytopes $\pmw$ are normal polytopes.
\end{proof}
\subsection{Applications}\label{applications} To state an important consequence of Theorem A and Theorem B we give the definitions of $essential$ $monomials$ due to Vinberg (see \cite{V05}, \cite{G11}) and $Feigin$-$Fourier$-$Littelmann$ (FFL) modules due to \cite{FFoL13a}. Let $\lambda$ be a dominant integral weight. Recall that we have a homogeneous lexicographical total order $\prec$ on the set of multi-exponents induced by the order on $\rrl$: $$\beta_1 \prec \beta_{2} \prec \dots \prec \beta_N.$$ In the following we fix a ordering on the factors in a vector
\begin{equation}\label{fixedorder}
f^{\bp}\vmw = f_{N}^{p_N}f_{N-1}^{p_{N-1}} \dots f_{1}^{p_1}\vmw.
\end{equation}
\begin{defi}
(i) We call a multi-exponent $\bp \in \Z_{\geq 0}^{N}$ $essential$ if 
\begin{equation*}
f^{\bp}\vmw \notin \mathrm{span}\{f^{\bq}\vmw \mid \bq \prec \bp\}.
\end{equation*}
%In this case we call $f^{\bp}$ an $essential$ $monomial$ and $f^{\bp}\vmw$ an essential vector.\\
(ii) Define $\es(\Vmw) \subset \Z_{\geq 0}^{N}$ to be the set of essential multi-exponents.

 %and let
%\begin{equation*}
%F_{\bp}(\Vmw) := \langle f^{\bq}\vmw \mid \bq \preceq \bp\rangle, \, F_{\bp}(\Vmw)^{-} := \langle f^{\bq}\vmw \mid \bq \prec \bp\rangle.
%\end{equation*}
\end{defi}
%These spaces define an increasing filtration on $\Vmw$, in particular if $\bq \prec \bp$ then $F_{\bq}(\Vmw) \subseteq F_{\bp}(\Vmw)$. We denote the associated graded space by 
%\begin{equation*}
%\Vmwt = \underset{\bp \, \in \, \Z_{\geq 0}^{N}}{\bigoplus}F_{\bp}(\Vmw) / F_{\bp}(\Vmw)^{-}.
%\end{equation*}
By \cite[Section 1]{FFoL13a}
%$\Vmwt$ is a cyclic $\Snml$-module with induced structure given by the action of $\Unl$ on $\Vmw$. Further 
$\{f^{\bp}\vmw\mid \bp \in \es(\Vmw)\}$ is a basis of $\Vmwa$ and of $\Vmw$.
\vspace{2mm}

Let $M=\Un v_M$ and $M^{'}=\Un v_{M^{'}}$ be two cyclic modules. Then we denote with $M \odot M^{'} := \Un (v_{M}\otimes v_{M^{'}}) \subset M \otimes M^{'}$ the $Cartan$ $component$ and we write $M^{\odot n} := M \odot \cdots \odot M$ (n-times).

%Let $\Vmw \odot V(\lambda{'}) := \Un (\vmw \otimes v_{\lambda'}) \subset \Vmw \otimes V(\lambda')$ be the $Cartan$ $component$ for $\lambda,\lambda'$ dominant integral weights and let $\dim \Vl^{\odot n} = \Vl \odot \cdots \odot \Vl$ (n-times).
\begin{defi}
We call a cyclic module M a FFL module if:
\begin{enumerate}
\item[(i)] There exists a normal polytope $P(M)$ such that $\es(M) = S(M)$, where $S(M)$ is the set of lattice points in $P(M)$.
\item[(ii)] $\forall n \in \N: \dim M^{\odot n} = |n S(M)|$, where $n S(M)$ is the n-fold Minkowski sum of $S(M)$.
\end{enumerate}
\end{defi}
%\begin{cor}
%Let $\mathbb{U} = exp(\mathfrak{n}^-)$, then $\Vmw$ is a favourable $\mathbb{U}$-module in the sense of FFL.
%\end{cor}
\begin{cor}
For the cases of Table \ref{fundamentalweights} $V(m\w_i)$ is a FFL module.
\end{cor}
\begin{proof}
Proposition $\ref{normalpolytope}$ shows that $P(m\w_i)$ is a normal polytope. By Theorem B a basis of $V(m\w_i)$ is given by $\B_{m\w_i}$, hence with Lemma $\ref{lemequa}$ we have $S(m\w_i) = \es(V(m\w_i))$.\\
Let $n \in \N$ be arbitrary, then $\dim V(m\w_i)^{\odot n} = \dim V(nm\w_i)$. Again by Theorem B we have $\dim V(nm\w_i)) = |S(nm\w_i))|$. Because $P(nm\w_i))$ is a normal polytope and therefore satisfies the Minkowski sum property, we conclude $|S(nm\w_i))| = |n S(m\w_i))|$.
\end{proof}
\begin{rem}
We note that in \cite{FFoL13a} the FFL modules are called favourable modules. 
\end{rem}

%% file: PBW-Basis_helpful.tex
\section{Normal polytopes}\label{helpful}
%In this section we provide sufficient conditions on polytopes $P(m) \subset \R^K\!, K \in \Z_{\geq 0}, m \in \Z_{\geq 0}$
%In this section we will give sufficient conditions on a $convex$ $lattice$ polytope $P$, i.e. $P \subset \R^K$, for some $K \in \Z_{\geq 0}$ and $P$ is the convex hull of finitely many points in $\Z^K \subset \R^K$, to be a $normal$ polytope. A polytope $P(m), m \in \Z_{\geq 0}$ is called normal, if the set of lattice points in the dilation $mP(1)$ is the m-fold Minkowski sum of the lattice points in $P(1)$. In \cite[Lemma 8.7]{FFoL13a} it is shown, if for $S(m) := P(m) \cap \Z_{\geq 0}^{K}, m \in \Z_{\geq 1}$

%In this section we provide sufficient conditions on polytopes $P(m) \subset \R^K\!, m \in \Z_{\geq 0}$ and fixed $K \in \Z_{\geq 0}$ to satisfy the Minkowski sum property:
%\begin{equation}\label{minkprop}
%S(m) = S(m-1) + S(1),\ m\geq 1,
%\end{equation}
%where $S(m) := P(m) \cap \Z_{\geq 0}^K$ is the set of integer points in $P(m)$ and $+$ is the Minkowski sum. In \cite[Lemma 8.7]{FFoL13a} it is shown that a bounded polytope defined by inequalities with integer coefficients and satisfying the property $\eqref{minkprop}$, is a \textit{normal} polytope.\\ A polytope $P(m), m \in \Z_{\geq 0}$ is called normal, if the set of integer points in the dilation $mP(1)$ is the m-fold Minkowski sum of the integer points in $P(1)$. So the property $\eqref{minkprop}$ is what we are aiming for.
Our goal in this section is to show, that the polytopes defined in $\eqref{poly}$ are normal.\\
A convex lattice polytope $P \subset \R^{K}, K \in \Z_{\geq 0}$, i.e. P is the convex hull of finitely many integer points, is called normal, if the set of integer points in the m-th dilation $mP$ is the m-fold Minkowski sum of the integer points in $P$.\\
To achieve our goal we will prove the normality condition for a larger class of polytopes in a more abstract setting than in Section $\ref{test}$.

%In this section we provide sufficient conditions on a polytope $P\subset \R^K\!, K \in \Z_{\geq 0}$ to be normal. A convex lattice polytope $P$ is called normal, if the set of lattice points in the dilation $mP$ is the m-fold Minkowski sum of the lattice points in $P$. In \cite[Lemma 8.7]{FFoL13a} it is shown that a bounded polytope P, which is defined by inequalities with integer coefficients and satisfies the property \begin{equation}\label{minkprop}
%mP\cap \Z_{\geq 0}^K = ((m-1)P)\cap \Z_{\geq 0}^K + P\cap \Z_{\geq 0}^K,\ m\geq 1,
%\end{equation}
%is a convex lattice polytope and therefore a normal polytope, where $+$ is the Minkowski sum. 
%

\subsection{General setting}
Let $\Delta = \{z_1, z_2, \dots, z_K\}$ be a finite, non-empty set with a total order: $z_1 \succ z_2 \succ \dots \succ z_K.$
%and with the property $\left|\rr\right|\geq 1$.
We extend $\succ$ to the (non-homogeneous) lexicographic order on $\mathcal{P}(\Delta)$, the power set of $\Delta$. Let $D=\{\bp_1,\dots,\bp_t\}\subset \prr$ be an arbitrary subset. 
%, where $p_j\neq \emptyset$ for $j\in \{1,\dots,t\}.$ 
\begin{rem} (i) To illustrate this non-homogeneous lexicographical order we give for $K \geq 3$ an example:
$$\{z_{1},z_{2}\} \succ \{z_1\} \succ \{z_2,z_3\}$$
(ii) Let $\bp=\{z_{i_1},\dots,z_{i_r} \}\in \prr$ be an arbitrary set. We always assume without loss of generality (wlog): $z_{i_1}\succ\cdots\succ z_{i_r}$. 
\end{rem} 
We can associate a collection of polytopes to $D$ in a natural way:
\begin{equation}\label{polytopm}
P(m)=\{ \mathbf{x}\in \R_{\geq 0}^{K} \mid \sum_{z_{j} \,\in\, \bp}x_{j}\leq m,\hspace{2mm} \forall \bp\in D \}, \,m \in \mathbb Z_{\geq 0} .
\end{equation}
%Moreover we will need the following definitions:
To work with these polytope, in particular with the elements in $D$, we define the following.
\newpage
\begin{defi}\label{def1}\mbox{}
\begin{itemize}
	\item[(1)] For $\bp\in \prr$ define $\bp_{\mathrm{min}}=\underset{\succ}{\mathrm{min}}\{z \in \bp \}$ and  $\bp_{\maxx}$ analogously.
	\item[(2)] Let $\bp,\bq\in \prr$, $\bp=\{z_{i_1},\dots,z_{i_r} \}$, $\bq=\{z_{j_1},\dots,z_{j_s} \}$ with $\bp_{\minn}=\bq_{\maxx}$. Then we define the $concatenation$ of $\bp$ and $\bq$ 
%at $\alpha_{i_r}\, (= \alpha_{j_1})$
by
\begin{equation*}
\bp \cup \bq=\{z_{i_1}, z_{i_2} \dots,z_{i_r}\!=\! z_{j_1}, z_{j_2}, \dots,z_{j_s} \}\in \prr.
\end{equation*}	 
	\end{itemize}
\end{defi}
%Now we are able to define a polytope associated to $D$ in a natural way
%, for which we want to show $\eqref{minkprop}$
 %the integral points of a polytope which is connected to $D$ in a natural way. For that let 
%\begin{equation}\label{polytopm}
%P(m)=\{ \mathbf{x}\in \R_{\geq 0}^{K} \mid \sum_{z_{j} \,\in\, \bp}x_{j}\leq m,\hspace{2mm} \forall \bp\in D \}, \,m \in \mathbb Z_{\geq 0} .
%\end{equation}
%be the to $D$ \textit{associated polytope}.
%The integer points $S(m)$ in $P(m)$ are defined as follows: $S(m)=P(m)\cap \mathbb Z_{\geq 0}^{K}$.
%We point out that the only difference between this definition and $\ref{poly}$ is that we have not chosen a interpretation for the Dyck paths.
\subsection{Normality condition}

\begin{defi}\label{helpfulldef}
%Let $\rr$, $\pr$, $D$, $\pmw$ and $\smw$ like in the general setup and
Assume $D\subset \prr$ has the following properties:
\begin{enumerate}
	\item Subsets of elements in $D$ are again in $D$: 
	\begin{center}$\forall A\subset \bp\in D: A\in D$.\end{center}
	\item Every $z \in \Delta$ lies at least in one element of $D$:\begin{center}$\underset{\bp \,\in \, D}\bigcup \bp=\Delta$ \end{center}
	%(usual union of sets).
	\item The concatenation of two elements in $D$, if possible, lies again in $D$:\begin{center} $\forall \bp,\bq\in D$ with $ \bp_{\minn}=\bq_{\maxx}$:
	%like in $\mathrm{Definition\ (\ref{def1})}$:
$\bp \cup \bq\in D$. \end{center}
\end{enumerate} 
Then we call $D\subset\prr$ a set of \textit{Dyck paths}. 
\end{defi}
We define for $m\in\Z_{\geq 0}$, $\supp_{\textrm{m}}: S(m)\to \prr,$ by $$ \bt = (t_z)_{z \in \Delta} \mapsto \supp_{\textrm{m}}(\bt)=\{z\in \Delta \mid\ t_{z}>0 \}.$$
Note that the map $\suppm$ is in general not injective. Furthermore we have $\supp_1(S(1))\subseteq\suppm(S(m))$, because of $S(1)\subseteq S(m)$ and $\suppm|_{S(1)}=\supp_1$.

\begin{rem}Let $D\subset\prr$ be a set of Dyck paths, then $P(m)$ defined in $\eqref{polytopm}$ is a bounded convex polytope for all $m \in \Z_{\geq 0}$.\\
By the definition of $P(m)$ and the second property of $D$, which guarantees that each $z \in \Delta$ lies in at least one Dyck path, we have $t_z \in \{0,1\}, \forall z \in \Delta$, for $\bt \in S(1)$. Hence $\supp_1$ is an injective map and we get an induced (non-homogeneous) total order on $S(1)$.
\end{rem}
Now we want to give a characterization of the image of $\supp_1$.
\begin{rem}
Let $D\subset\prr$ be a set of Dyck paths, then
\begin{equation}\label{charak}
\supp_1(S(1))=\{A\in \prr \mid |A\cap \bp|\leq 1, \forall \bp\in D \}=:\Gamma.
\end{equation}
"$\subseteq$": Assume there is an element $\bt\in S(1)$ with $\supp_1(\bt)=A\in \prr$ and $|A\cap \bp|> 1$ for some $\bp\in D$. Then we have $\sum_{z\in A\cap \bp}t_{z}>1$, since $t_{z}>0,\ \forall z\in A$. And so we have: $\sum_{z \in \bp}t_{z}>1$. But this is a contradiction to the assumption $\bt\in S(1)$.
\vspace{11pt}

"$\supseteq$": Let $B\in\Gamma$ be arbitrary. Associated to B we define $\bq^{B}\in \Z_{\geq 0}^{K}$ by $q_{z}^B=1$ if $z\in B$ and $q_{z}^B=0$ else. By the definition of $\Gamma$ we have for every Dyck path $\bp\in D$: $\sum_{z\in \bp} q_{z}^B\leq 1$. Hence $\bq^{B}\in S(1)$ with $\supp_1(\bq^B) = B$.

\end{rem}
Let $\bs \in S(m), m \in \Z_{\geq 0}, \bs \neq 0$ be an arbitrary non-zero element. Consider $\suppm(\bs)\in \prr$, we have $\mathcal{P}(\suppm(\bs))\subseteq \prr$. Let 
\begin{equation}\label{nabla}
\nabla =(\supp_1(S(1)) \cap \mathcal{P}(\supp_m(\bs))\subseteq \prr.
\end{equation}
Note that $\nabla$ is a total ordered, non-empty set, because $S(1)$ contains all unit vectors and $\bs \neq 0$ by assumption. So there is a unique maximal element (with respect to $\succ$), denoted by $\ms\in \nabla$.
\begin{lem}\label{hilfelem}
 Let $D$ be a set of Dyck paths, $\bs \in S(m)$ non-zero and $\mu \in M_{\bs}$. Then we have $s_\nu = 0$ for all $ \nu \in \Delta $ such that $(\nu \succ \mu \, \, \textrm{and} \, \,  \exists \bq \in  D\!: \nu,\mu \in \bq).$
\end{lem}
\begin{proof}
We assume the contrary. That means there exists $\nu \in \Delta$ with $\nu \succ \mu$, $s_\nu \neq 0$ and a Dyck path $\bp \in  D$ such that $\nu,\mu \in \bp$. Define 
\begin{equation*}
V := \{\tau \in M_{\bs} \mid \exists\, \bq \in  D: \nu, \tau \in \bq, \nu \succ \tau\} \subset M_{\bs}
\end{equation*}
and
$M'_{\bs} := (\{\nu\} \cup M_{\bs}) \setminus V.$
By assumption we have $\mu \in V$ and so $| V | \geq 1$. Further we have $M'_{\bs} \in \mathcal{P}(\supp_m({\bs}))$
and we want to show that $M'_{\bs} \in \supp_1(S(1))$.\\
We assume that this is not the case. So there exists some $\bb \in  D$ such that $| M_{\bs}' \cap \bb | > 1.$ By the definition of V this can only happen, if there exists a $\alpha \in M_{\bs}$ with $\alpha \succ \nu$ and $\alpha,\nu \in \bb$. The following picture is intended to give a better understanding of the foregoing situation.
\begin{center}
\begin{tikzpicture}
\node (1) at (0,0) {$\nu$};
\node (2) at (1,0) {$.$};
\node (3) at (-1,0) {$.$};
\node (4) at (1,1) {$\tau_1$};
\node (5) at (-1,-1) {$.$};
\node (6) at (2,0) {$.$};
\node (7) at (2,-1) {$\tau_2$};
\node (8) at (-2,-1) {$.$};
\node (9) at (-3,-1) {$\alpha$};
\node (10) at (3,0) {$\mu$};
\node (11) at (3.3,-0.05) {$,$};
\node (12) at (4.5,0.05) {$\tau_1, \tau_2 \in V.$};

\draw[->][thick] (5) edge node[right] {\bf b} (3);
\draw[->][thick] (3) edge (1);
\draw[->][thick] (1) edge (2);
\draw[->] (2) edge (4);
\draw[->][thick] (2) edge (6);
\draw[->] (6) edge (7);
\draw[->][thick] (9) edge (8);
\draw[->][thick] (8) edge (5);
\draw[->][thick] (6) edge node[above] {\bf p} (10);
\end{tikzpicture}
\end{center}
We can assume wlog that $\bb_{\min} = \nu$ and $\bp_{\max} = \nu$, because subsets of Dyck paths are again Dyck paths. So the concatenation $\bb \cup \bp \in  D$ is defined and we have $\alpha,\nu \in \bb \cup \bp$. But then, because of $\alpha,\nu \in M_{\bs}$:
$| M_{\bs} \cap \bb | > 1,$
which is a contradiction to $M_{\bs} \in \supp_1(S(1))$. \\
So for all $\bq \in  D$ we have $| M'_{\bs} \cap \bq | \leq 1$. By that and with $M'_{\bs} \in \prr$ we conclude $M'_{\bs} \in \supp_1(S(1)).$
Therefore $M'_{\bs} \in \nabla$ and by construction, because $\succ$ is a lexicographic order, $M'_{\bs} \succ M_{\bs}$, which is a contradiction to the maximality of $M_{\bs}$. So the assumption on the existence of $\nu$ was wrong, which proves the Lemma.
\end{proof}
\begin{prop}\label{helpfullprop}
Let $D\subset\prr$ be a set of Dyck paths, then we have for the integer points $S(m)$ of the polytopes $P(m)$ associated to $D$:
\begin{equation}\label{minksum}
S(m-1) + S(1)=S(m),\ \forall m\in \Z_{\geq 1}, 
\end{equation} where the left-hand side $(\mathrm{lhs})$ of $\mathrm{(\ref{minksum})}$ is the Minkowski sum of $S(m-1)$ and $S(1)$. 
\end{prop}
\begin{proof}
Let $m \geq 1$. From the definition of $P(m)$ and of the Minkowski sum follows $S(m-1) + S(1) \subset S(m)$. So it is sufficient to show that 
\begin{equation}
S(m-1) + S(1) \supset S(m)
\end{equation}
holds. For that let $\bs=(s_{z})_{z \in \Delta}\in S(m)\setminus S(m-1)$ be an arbitrary element. We show that there exists an integer point $\bt^1\in S(1) \setminus \{0\}$ such that: $\bs-\bt^1\in S(m-1)$. 
We define for $\ms$ defined as in $\eqref{nabla}$:
\begin{equation}\label{aleph}
\bt^1:=\supp_1^{-1}(\ms)\in S(1)\setminus \{0\}.
\end{equation} 
This element is unique because of the injectivity of $\supp_1$. Now we consider the integer point $\bs-\bt^1$. We know that there are no negative entries, because $s_{z}=0$ implies for all $A \in \nabla: z \notin A$ and so $t^1_{z}=0 $. Hence $\bs-\bt^1 \in S(m)$ and so the second step is to show that $\bs-\bt^1$ lies already in $S(m-1)$.\\
To achieve that we assume contrary $\bs - \bt^1 \in S(m) \setminus S(m-1)$, i.e. that there is a Dyck path $\bp\in D$ such that: $$\sum_{z \in \bp} (s_{z}-t_{z}^1)=m.$$ Since $\bs\in S(m)$ we have:
\begin{equation}\label{folg}
m=\sum_{z\, \in \, \bp} (s_{z}-t_{z}^1)= \underbrace{\sum_{z\, \in \, \bp} s_{z}}_{\leq\, m}-\underbrace{\sum_{z\, \in\, \bp}t_{z}^1}_{\geq\, 0}\Rightarrow \sum_{z\, \in\, \bp} s_{z}= m\ \mathrm{and}\  \sum_{z \,\in\, \bp}t_{z}^1=0.
\end{equation}
%\begin{eqnarray}\label{folg}
%m=\sum_{\alpha\, \in \, \bp} (\bs_{\alpha}-\bt_{\alpha}^1)= \underbrace{\sum_{\alpha\, \in \, \bp} \bs_{\alpha}}_{\leq\, m}-\underbrace{\sum_{\alpha\, \in\, \bp}\bt_{\alpha}^1}_{\geq\, 0}\nonumber\\
%\Rightarrow \sum_{\alpha\, \in\, \bp} \bs_{\alpha}= m\ \mathrm{and}\  \sum_{\alpha \,\in\, \bp}\bt_{\alpha}^1=0.
%\end{eqnarray}
We want to construct another Dyck path $\overline{\bp}\in D$ such that $\sum_{z \in \overline{\bp}}\ s_{z}>m$.\\ Let $\beta\in \Delta$ be maximal with the property $\beta \in \bp \wedge s_{\beta}>0 $. In particular, since $\sum_{z \in \bp}(s_z - t_{z}^1)=m$ we have $\bp \cap \ms = \emptyset$ and so $\beta \notin \ms$. We define
\begin{equation*}
\bp'=\bp \setminus\{ \gamma\in \bp \mid \gamma\succ \beta\},
\end{equation*} 
which is an element of $D$ since subsets of Dyck paths are again Dyck paths.
%\end{proof}
By construction we have
$$\underset{z \, \in \,  \bp'}{\sum} s_z = m = \underset{z \, \in \, \bp}{\sum} s_z .$$
There are two possibilities to extend the path $\bp'$ with a further Dyck path $\bp'' \in  D$: 
\begin{equation*}
(i)\,\bp_{\min}'' = \beta \textrm{ or }(ii)\, \bp_{\max}'' = \bp_{\min}.
\end{equation*}
To obtain a path $\bps= \bp''\cup \bp'$ (respectively $\bps=\bp'\cup \bp''$) with $\sum_{z \in \overline{\bp}}\ s_{z}>m$, the extension $\bp''$ has to satisfy the following condition: $\bp'' \cap M_{\bs} \neq \emptyset$.\\[2mm] 
Assume we are in the case $(ii)$. Then there exists $\tau\in \bp'' \cap M_{\bs}$ with $s_\tau>0$. Further we have $s_\beta>0$ and $\tau,\beta\in\bp'\cup \bp''=\bps\in D$. By construction we have $\beta \prec \tau$ and so Lemma \ref{hilfelem} implies that $s_\beta=0$. This is a contradiction to $s_\beta>0.$\\[2mm] 
%\begin{lem}\label{hilfelem}
 %Let $t_\mu^1 \neq 0,$ i.e. $\mu \in M_{\bs}$ then we have $s_\nu = 0$ for all $ \nu \in \Delta $ such that $(\nu \succ \mu \, \, \textrm{and} \, \,  \exists \bq \in  D\!: \nu,\mu \in \bq).$
%\end{lem}
So we want to show the existence of a path $\bp'' \in D$ with condition $(i)$ and $\bp'' \cap \ms \neq \emptyset$. We assume contrary there is no such Dyck path $\bp''$: 
\begin{equation}\label{assumptiontwo}
\forall \bq \in  D \, \, \textrm{with} \, \, \bq_{\min} = \beta: \bq \cap \ms = \emptyset.
\end{equation}
Under this assumption and by using Lemma $\ref{hilfelem}$ we will show:
\begin{equation}\label{bla}
\forall \bq \in  D \, \, \textrm{with} \, \, \beta \in \bq: \bq \cap M_{\bs} = \emptyset.
\end{equation}
Assume $\eqref{bla}$ is not true, so there is some $\beta \neq \tau \in \bq \cap M_{\bs}$ for $\bq \in D$ with $\beta \in \bq$. Then we have two cases.\\
Let $\tau \succ \beta$, then $\tau$ and $\beta$ lie in $\bq$. Now the path from $\tau$ to $\beta$ is again a Dyck path. But this is a contradiction to Assumption $\eqref{assumptiontwo}$. \\
Let $\beta \succ \tau$, by $\tau \in \bq \cap M_{\bs}$ we have $t_\tau^1 \neq 0$. Then Lemma \ref{hilfelem} implies $s_\beta = 0,$ which is a contradiction to the choice of $\beta$.\\
Therefore $\eqref{bla}$ holds. Recall the properties of $M_{\bs}.$ We have 
\begin{equation*} 
M_{\bs} = \supp_{1}(\bt^{1})   \in \mathcal{P}(\Delta) \,\,\textrm{with}\,\, | M_{\bs} \cap \bq | \leq 1, \, \forall \bq \in  D.
\end{equation*}
Now consider $M_{\bs}' := M_{\bs} \cup \{\beta\} \in \mathcal{P}(\suppm(\bs)).$ We will show that $M_{\bs}' \in \supp_1(S(1)).$
For $\bq \in  D$ with $\beta \in \bq$ we have $| M_{\bs}' \cap \bq | = 1$ by $(\ref{bla})$.\\
For $\bq \in  D$ with $\beta \notin \bq$ we have $| M_{\bs}' \cap \bq | \leq 1$ by $| M_{\bs} \cap \bq | \leq 1$.\\
We conclude $M_{\bs}' \in \supp_1(S(1))$ and so
\begin{equation*}
M_{\bs}' \in \nabla = \supp_1(S(1)) \cap \mathcal{P}(\suppm(\bs)).
\end{equation*}
But with $M_{\bs}' \succ M_{\bs}$ we get a contradiction to the maximality of $M_{\bs}$.\\
So Assumption $\eqref{assumptiontwo}$ was wrong and there exists 
\begin{equation*}
\bp''  \in  D \, \, \textrm{with} \, \, \bp_{\min}'' = {\beta}: \bp'' \cap M_{\bs} \neq \emptyset.
\end{equation*}
We recall that $\beta \notin \ms$ and therefore $\tilde{\bp} \neq \{\beta\}.$
%We recall that for our fixed $\bp \in  D$ with $\sum_{\alpha\in\bp} (\bs_\alpha - \bt_\alpha^1) = m$ we have $\bp \cap M_{\bs} = \emptyset$ and so $\tilde{\bp} \neq \{\beta\}.$ 
Define the concatenation of $\bp''$ and $\bp'$ in $\beta$ as $\overline{\bp} :=  \bp'' \cup \bp' \in  D$ which is indeed defined because $ \bp''_{\min} = \beta = \bp'_{\max}$. From Definition $\ref{helpfulldef}(3)$ we know that $\overline{\bp}$ is a Dyck path. Now by construction we conclude
$$ \underset{z \, \in \, \overline{\bp}}{\sum} s_z = {\underset{> \, 0}{\underbrace{\underset{z \, \in \,  \bp''}{\sum} s_z}}} + \underset{ = \, m}{\underbrace{\underset{z \, \in \, \bp'}{\sum s_z}}} > m.$$
But this is a contradiction to the choice of $\bs \in S(m)$ and the assumption\\ $\sum_{z\in \bp} (s_z - t_z^1) = m$ was wrong. We conclude $\bs - \bt^1 \in S(m-1)$ and with $\bt^1\in S(1)$ we have $\bs \in S(m-1) + S(1)$. Finally we get $S(m) \subset S(m-1) + S(1).$
\end{proof}
\subsection{Consequences}
%Assumption (2) of $\ref{helpfulldef}$ ensures that the polytope $P(m)$ defined in $\eqref{polytopm}$ is bounded. We alredy saw that $\eqref{minkprop}$ implies that $P(m)$ is a normal polytope.\\
%Furthermore by \cite[Lemma 8.7]{FFoL13a} we know that if $P(m)$ satisfies $\eqref{minksum}$, then is $P(m)$ a normal convex lattice polytope.
We recall the construction of the Hasse diagram and the Dyck paths from Section $\ref{test}$ and show that we can apply Proposition $\ref{helpfullprop}$ to this setup. Let $\lambda = m\omega_i$ as before and we set $\Delta = \Delta_+^{\w_i}, D = D_{\w_i}$. Then we have for the associated polytopes: 
\begin{equation*}
P(m) = P(m\omega_i).
\end{equation*}
For $\rrl = \{\beta_1,\dots,\beta_N\}$ we chose in Section $\ref{test}$ the order $\beta_1 \prec \dots \prec \beta_N$. To apply Proposition $\ref{helpfullprop}$ we can use the same order on the positive roots and extend this order to the (non-homogeneous) lexicographical order on $\mathcal{P}(\Delta_+^{\w_i})$ as before. 
%\begin{equation*}
%$\beta_1 \succ \beta_2 \succ \dots \succ \beta_N.$
%\end{equation*}
%and use the (non-homogenous) lexicographical order on $\prrl$.
We want to show that the Dyck paths defined in Section $\ref{test}$ are Dyck paths in the sense of Definition $\ref{helpfulldef}$.\\
%Then, direct from the construction of $\Hgl$, we can state some properties of $D_{\lambda}$:\\
%(i) There is a unique $starting$ point $\theta$ and a unique $ending$ point, $\alpha_i \in \Phi^+:\langle \omega_i,\alpha_{i}^{\vee}\rangle = 1$, such that we can extend every path in $\Hgl$ to a path starting in $\theta$ and ending in $\alpha_i$.\\
(1) Every $\bp' \subset \bp \in D_{\w_i}$ is again a Dyck path: We saw that any ordered subset of a directed path in $\Hgl$ is again a Dyck path.\\
(2) For each $\beta \in \Delta_+^{\w_i}$ there is at least one $\bp \in D_{\w_i}$ such that $\beta \in \bp$: The set of vertices in $\Hgl$ is exactly $\Delta_+^{\w_i}$. By construction we allow paths of cardinality one, so for example the path $(\beta)$ contains $\beta$.\\
%By Remark $\ref{defdyck}$ $\theta, \beta \in \bp$ for an $\bp \in D$ if $\theta - \beta$ is a non-negative linear combination of simple roots. But  is the highest root, so every $\beta$ is smaller than $\theta$ regarding the standard partial order and the Hasse diagrams we consider are connected.\\
%We saw that we can interpret $\bp$ as a directed path in $\Hgl$. Hence, because $\Hgl$ is a connected diagram, if we cancel out the $\beta \in \bp \setminus \bp'$, there is directed path associated to $\bp'$.\\
(3) Let $\bp
%=\{\beta_{i_1}, \dots,\beta_{i_r}\}
,\bp'
%=\{\beta_{j_1}, \dots,\beta_{j_s}\}
\in D_{\w_i}$ be two Dyck paths, such that $\bp_{\min} = \bp'_{\maxx}$. Then there are directed paths $W,W'$ in $\Hgl$ realizing $\bp$ and $\bp'$ such that the end point of $W$  is equal to the starting point of $W'$. We consider the directed path, which we obtain by the concatenation of the directed paths $W$ and $W'$. This directed path realizes $\bp \cup \bp'$. Hence $\bp \cup \bp'$ lies in $D_{\w_i}$.\\ 
%, i.e. $\beta_{i_r} = \beta_{j_1}$. By the construction and the change of order it is $\beta_{i_1} \succ \dots \succ \beta_{i_r}=\beta_{j_1}\succ \beta_{j_2}\succ \dots \succ \beta_{j_s}$.
%Again, directly via the realization as a connected diagram the concatenation $\bp \cup \bp'$ can be interpreted as a directed path in $\Hgl$ and therefore it lies in $D_{\lambda}$. 
With Proposition $\ref{helpfullprop}$ we get immediately for $S(m\omega_i) = P(m\omega_i) \cap \Z_{\geq 0}^N, m \in \Z_{\geq 0}:$
\begin{prop}\label{normalpolytope} 
$S(m\omega_i) = S((m-1)\omega_i) + S(\omega_i), m \in \Z_{\geq 1}.\hfill \Box$
\end{prop}
Finally we conclude that the polytopes constructed in $\eqref{poly}$ are normal convex lattice polytopes.
%
%Vielleicht Remark, dass man mit (i) - (iv) bewiesen hat, dass kommen von nem diagram, das assozierte polytop immer minksum hat und für nen polytop--> diagram --> polytop' gilt polytop = polytop', dann erfültt polytop minksum.
%\end{document}

%% file: PBW-Basis_Spanning.tex
\section{Spanning Property}\label{spanning}
\setcounter{subsection}{1}
%\setcounter{subsection}{1}
%\subsection{Spanning}\label{proofs}  
Let $\Lg$ be a simple complex finite-dimensional Lie algebra, $\lambda =m\omega$ be a rectangular dominant integral weight such that $\langle \omega, \theta^{\vee} \rangle \, = 1$, where $\theta$ is the highest root in $\Delta_{+}$ and $m \in \Z_{\geq 0}$. In this section we show that $\Bl = \{f^{\bs}\vl \mid \bs \in S(\lambda)\}$ is a spanning set for $\Vla$. 
%We show that the polytopes $\pmw$ we constructed, induce for each Cartan type bases for $\Vmw, m \in \Z_{\geq 0}$.\\
Recall that we have 
\begin{equation*}
\Vmwa \cong \Snml / I_{\lambda},
\end{equation*} where $I_{\lambda}$ is the annihilating ideal of $\vmw$. We know that $f_{\alpha}^{\langle \lambda, \alpha^{\vee} \rangle + 1 }\vmw$ is zero in $\Vmw$ (see $\eqref{wkd}$). Hence $f_{\alpha}^{\langle \lambda, \alpha^{\vee} \rangle + 1 }\vmw = 0$ in $\Vmwa$. By the action of $\Unp$ on $\Vmwa$ we obtain further relations. We will see that these relations are enough to rewrite every element as a linear combination of $f^{\bs} \vmw, \bs \in \smw$.\\
In our proof it is essential to have a Hasse diagram $H(\Ln_{\lambda}^-)_{\Lg}$ without $k$-chains. 
%To prove this, it is essential to have a Hasse diagram $\Hgw$ without chains.
A Dyck path is defined as before to be the set of roots corresponding to a directed path in $H(\Ln_{\lambda}^-)_{\Lg}$.\\
Let $\circ$ be the action of $\Unp$ on $S(\mathfrak{g})$ induced by the adjoint action of $\mathfrak{n}^+$ on $\mathfrak{g}$. Via the isomorphism 
%\begin{equation*}
$S(\mathfrak{n}^-) \cong S(\mathfrak{g})/ S(\mathfrak{g}) (S_+(\mathfrak{n}^+ \! \oplus \mathfrak{h}))$
%\end{equation*}
we obtain an action on $\Sn$, where $S_+(\mathfrak{n}^+ \! \oplus \mathfrak{h}) \subset S(\mathfrak{n}^+ \! \oplus \mathfrak{h})$ is the augmentation ideal. By
\begin{equation*}
\Snml \cong \Sn / \Sn(\textrm{span}\{f_\beta \mid \beta \in \Delta_{+}\setminus\Delta_{+}^{\lambda}\})
\end{equation*}
we get an action on $\Snml$. We denote this action again by $\circ$.
Since the action of $\Unp$ on $\Vmwa$ is induced by the action of $\Unp$ on $\Vmw$ (which is again induced by the adjoint action), we obtain that for all $e \in \Unp, f \in S(\mathfrak{n}_{\lambda}^-)$
\begin{equation}\label{circ}
e(f\vmw) = (e \circ f) \vmw,
\end{equation}
holds. Therefore we can restrict our further discussion on the $\Unp$-module $\Snml$. Equation $\eqref{circ}$ and $\Unp (f\vl) = \Unp(0) = \{0\}$ for all $f \in I_{\lambda}$ imply that $I_{\lambda}$ is stable under $\circ$. Furthermore, by Remark $\ref{degreeinvariance}$ 
%because $e_{\alpha}\circ f_{\beta} = f_{\beta-\alpha}$, if $\beta-\alpha \in \Delta^+$ and $e_{\alpha}\circ f_{\beta} = 0$, if $\beta-\alpha \notin \Delta^+$, 
the total degree of a monomial in $\Snml/I_{\lambda}$ is invariant or it is zero under $\circ$.
We denote as before $\rrl = \{\beta_1, \dots, \beta_N\}$ and use the same total order $\prec$ on the multi-exponents (resp. monomials) as defined in Section $\ref{test}$, which is induced by ${\beta_1 \prec \beta_2 \prec \dots \prec \beta_N}$.\\[2mm]
We define differential operators; for $\al, \beta \in \Delta_{+}$ let
\begin{equation*}
\partial_{\alpha} f_{\beta} :=
\begin{cases}
f_{\beta - \alpha}, \, &\textrm{if} \, \beta - \alpha \in \rrl\\
0, \,& \textrm{else.}
\end{cases}
\end{equation*} The operators satisfy
\begin{equation*}
\partial_\alpha f_\beta = c_{\alpha,\beta} [e_\alpha, f_\beta],
\end{equation*}
for constants $c_{\al,\beta} \in \C$. So instead of using $\circ$ we can work with these differential operators. 
We point out that we need the differential operators for arbitrary roots in $\Delta_+$.
\setcounter{thm}{0}
\begin{rem}
Here we want to illustrate the problem which occurs if we allow k-chains in our Hasse diagram. Let $\gamma\prec\beta\prec\delta$ the roots of a k-chain $\gamma\xrightarrow{k}\beta\xrightarrow{k}\delta$ and consider for $\ell \geq 2$:
\begin{equation}\label{probichain}
\partial_k^2f_{\gamma}^{\ell}= \partial_k (\ell f_{\beta}^1f_{\gamma}^{\ell-1})=\underbrace{c_0 \ell f_{\delta}^1f_{\beta}^0f_{\gamma}^{\ell-1}}_{\text{maximal monomial}}+\ c_1 \ell(\ell-1) f_{\beta}^2f_{\gamma}^{\ell-2},
\end{equation}
with $c_0 = c_{\gamma,\al_k}c_{\beta,\al_k}$ and $c_1= c_{\gamma,\al_k}^2$ where $c_{\gamma,\al_k}, c_{\beta,\al_k}$ are the structure constants corresponding to $[e_{\al_k},f_{\beta}]$ and $[e_{\al_k},f_{\gamma}]$ respectively.
So it is more involved to find a relation which contains $\beta$ and $\delta$.
\end{rem}
The next Lemma describes the action of the differential operators and gives an explicit characterization of the maximal monomial of $\partial_{\nu} f^{\bs}$ for certain $\nu \in \Delta_+$ and $\bs \in \Z_{\geq 0 }^N$.
\begin{lem}\label{diffoperatordescription} Assume $H(\Ln_{\lambda}^-)_{\Lg}$ has no $k$-chains.\\
(i)\, Let $\bp = \{ \beta_{i_1}, \dots, \beta_{i_r}\} \in D_{\lambda}$ with $\beta_{i_1} \prec \dots \prec \beta_{i_r}$ and $\nu \in \Delta_{+}$. Further let $\beta_{i_k}, k \leq r$ be maximal such that $\partial_{\nu}f_{\beta_{i_k}} \neq 0$. Let $\bs \in \Z_{\geq 0 }^N$ be a multi-exponent supported on $\bp$, i.e. $s_{\beta} = 0$ for $\beta \notin \bp$. Then the maximal monomial in $\partial_{\nu}^{\, l} f^{\bs} = \partial_{\nu}^{\, l}(f_{{i_1}}^{s_1} \dots f_{{i_r}}^{s_r})$, $l \leq	s_k$, is given by
\begin{equation*}
f_{{i_1}}^{s_1} \dots f_{{i_{k-1}}}^{s_{k-1}}(f_{{i_k-\nu}}^{l}f_{{i_k}}^{s_k-l}) f_{{i_{k+1}}}^{s_{k+1}} \dots f_{{i_r}}^{s_r}.
\end{equation*}
(ii)\, Let $\sum_{\bu \, \in \, \Z_{\geq 0}^N} c_{\bu} f^{\bu} \in S(\mathfrak{n}^-)$ and $\nu \in \Delta_+$. Let $\bh = \underset{\prec}{\max}\{\bu \mid \partial_{\nu} f^{\bu} \neq 0, c_{\bu} \neq 0\}$. Further let $\beta_k = \underset{\prec}{\max} \{ \beta \mid f_{\beta} \, \textrm{is a factor of} \, f^{\bu}, \partial_{\nu} f_\beta \neq 0,c_{\bu} \neq 0\}$ and assume $h_{\beta_k} > 0$. Then for $l  \leq h_{\beta_k}$ the maximal monomial in 
\begin{equation*}
\partial_{\nu}^{\,l} \sum_{\bu \, \in \, \Z_{\geq 0}^N} c_{\bu} f^{\bu} = \sum_{\bu \, \in \, \Z_{\geq 0}^N} c_{\bu} \partial_{\nu}^{\,l} f^{\bu}
\end{equation*}
appears in $\partial_{\nu}^{\,l} f^{\bh}$.
\end{lem}
\begin{proof}
$(i)$ Assume we have two roots $\beta_{i}, \beta_{j} \in \rrl$ with $\beta_{i} \prec\beta_{j}$ and $\beta_{i} - \nu$ and $\beta_{j} - \nu$ are again roots in $\rrl$. For $\beta_{i_l} - \nu \notin \rrl$ we have $\partial_{\nu} f_{\beta_{i_l}} = 0$, so we do not need to consider such roots $\beta_{i_l}\in\rrl$. So in order to prove $(i)$, because our monomial order is lexicographic, it is sufficient to show that 
%$\beta_{i} - \nu \prec \beta_{j} - \nu$ .
\begin{equation}\label{orderinvariance}
\beta_i \prec \beta_j \Rightarrow \beta_i - \nu \prec \beta_j - \nu.
\end{equation}
If $\beta_i > \beta_j$ with respect to the standard partial order we have $\beta_i - \nu > \beta_j - \nu$ and therefore $\beta_i - \nu \prec \beta_j - \nu$, by the choice of the total order $\eqref{totalorder}$ on $\rrl$.\\ If the roots are not comparable with respect to the standard partial order, the second step is to compare the heights of the roots. So if $\height(\beta_i) > \height(\beta_j)$ then $\height(\beta_i - \nu) > \height(\beta_j - \nu)$ and again $\beta_i - \nu \prec \beta_j - \nu$.\\If $\height(\beta_i) = \height(\beta_j)$, we have to consider $\beta_i = (s_1,\dots,s_n)$ and $\beta_j = (t_1,\dots,t_n)$ in terms of the fixed basis of the simple roots (see Remark \ref{fixedbasis}). Then there is a $1 \leq k \leq n$, such that $s_k > t_k$ and $s_i = t_i$ for all $1 \leq i < k$. Let $\nu = (u_1,\dots,u_n)$, then $\beta_i - \nu = (s_1-u_1,\dots,s_n-u_n)$ is lexicographically greater than $\beta_j - \nu = (t_1 - u_1, \dots, t_n - u_n).$ Thus $\beta_i - \nu \prec \beta_j - \nu$ and $\eqref{orderinvariance}$ holds.\\[2mm]
$(ii)$ We only have to consider the multi-exponents $\bs \in \Z_{\geq 0}^N$ such that $\partial_{\nu} f^{\bs} \neq 0$. Now let $\bt$ be the maximal multi-exponent with this property and let $l \leq t_{\beta_k}$. Then we have $\partial_\nu^{\, l} f^{\bt} \neq 0$ and by (i) the maximal monomial appearing in $\partial_\nu^{\, l} f^{\bt}$ is 
\begin{equation}\label{maximalmonomial}
f_{\beta_k - \nu}^{\, l} f_{\beta_k}^{t_{\beta_k} - \, l} \underset{\beta \neq \beta_k - \nu}{\prod_{\beta \in \rrl, \beta \neq \beta_k}}  f_{\beta}^{t_{\beta}}.
\end{equation}
The observation $\eqref{orderinvariance}$ tells us that $f_{\beta_k - \nu} = \max\{f_{\beta - \nu} \mid \partial_{\nu} f_{\beta} \neq 0, s_{\beta} > 0 \}$. So by the choice of $\bt$ and because our order is lexicographic, the element $\eqref{maximalmonomial}$ is the maximal monomial in $\sum_{\bs \, \in \, \Z_{\geq 0}^N} c_{\bs} \partial_{\nu}^{\, l} f^{\bs}$.
\end{proof}
\begin{prop}\label{important} 
%Let $\Lg$ be a simple Lie algebra and $\lambda=m\omega \in \Lambda^+$, $m \in \Z_{\geq 0}$ with $\langle\omega, \check{\theta}\rangle= 1$, where $\theta \in \Delta_{+}$ is the highest root.
Assume $H(\Ln_{\lambda}^-)_{\Lg}$ has no $k$-chains and let $\bp \in D_{\lambda}$ be a Dyck path, $\bs \in \Z_{\geq 0}^{N}$ be a multi-exponent supported on $\bp$. Suppose further $\langle \lambda,\theta^{\vee}\rangle=m$ and $\underset{\al \, \in \, \bp}{\sum}{s_{\al}} > m.$ Then there exist constants $c_{\bt} \in \C$, $\bt \in \Z_{\geq 0}^N$ such that:
\begin{equation}\label{straighteninglaw}
f^{\bs} + \underset{\bt \, \prec \, \bs}{\sum}{c_{\bt}f^{\bt}} \in I_{\lambda}.
\end{equation}
%for some $c_{\bt} \in \C$.
\end{prop}
%\begin{rem}
%\end{rem}
We follow an idea of \cite{FFoL11a,FFoL11b} who showed a similar statement in the cases $\mathfrak{sl}_n$ and $\mathfrak{sp}_n$ for arbitrary dominant integral weights.
%\begin{rem}
%We denote $\eqref{straighteninglaw}$ as a straightening law, because it implies
%\begin{equation*}
%f^{\bs} = - \underset{\bt \, \prec \, \bs}{\sum}{c_{\bt}f^{\bt}} \in \Snml/\Imw \cong \Vmwa.
%\end{equation*}
%\end{rem}
\begin{proof}
Let $\bp = \{\tau_0,\tau_1, \dots, \tau_{r}\} \in D_{\lambda}$ be an arbitrary Dyck path. By construction we have for $1 \leq i \leq r$: $\tau_{i-1} \prec \tau_{i}$. Because $\sum_{i = 0}^{r}{s_{\tau_i}} > m$ we have
\begin{equation*}
f_{_{\theta}}^{{\text{\scriptsize{$s_{\tau_0}+ \dots + s_{\tau_{r}}$}}}} \in I_{\lambda}.
\end{equation*}
%We define differential operators; for $\al, \beta \in \Delta_{+}$ let
%\begin{equation*}
%\partial_{\alpha} f_{\beta} :=
%\begin{cases}
%f_{\beta - \alpha}, \, &\textrm{if} \, \beta - \alpha \in \rrl\\
%0, \,& \textrm{else.}
%\end{cases}
%%\left\{
%%\begin{array}{l}
%%f_{\beta - \alpha}, \, \textrm{if} \, \beta - \alpha \in \Delta^{+}\\
%%0, \, \textrm{else.}
%%\end{array}
%\end{equation*} The operators satisfy
%\begin{equation*}
%\partial_\alpha f_\beta = c_{\alpha,\beta} [e_\alpha, f_\beta],
%\end{equation*}
%for constants $c_{\al,\beta} \in \C$. So instead of using $\circ$ we can work with these differential operators. 
%%For any index set $I$ we write for $i \in I, n \in \Z_{\geq 0}, x_i \in \Delta_+$:
%%%\begin{equation*}
%%%\partial_i := \partial_{x_i} 
%%%\textrm{ and }  \partial_i^{^\text{\scriptsize{$n$}}} := \underset{n \textrm{ times}}{\underbrace{\partial_i^{} \dots \partial_i^{}}} .
%%%\end{equation*}
%%\begin{equation*}
%%\partial_i := \partial_{x_i} 
%%\textrm{ and }  \partial_i^{^\text{\scriptsize{$n$}}} := \partial_i^{} \dots \partial_i^{} \,\, (\text{n-times}) .
%%\end{equation*}
%%and for $n \in \Z_{\geq 0}: \underset{n \textrm{ times}}{\underbrace{\partial_i^{} \dots \partial_i^{}}} = \partial_i^{^\text{\scriptsize{$n$}}}$.\\
%We point out that we need the differential operators for arbitrary roots in $\Delta_+$.\\ %But if the result lies in $\Delta_+\setminus\rrl$ we send it to zero.\\
By the construction of the Hasse diagram there is a Dyck path $\bp' \in D_{\lambda}$ with $\bp \subset\bp'$, such that there is no path $\bp''$ with $\bp' \subsetneq \bp''$.
%Further we can assume that $\bp'$ starts in the highest root $\theta$ and ends in the simple root corresponding to $\w$, which is always $\beta_N$.\\
Hence we can assume wlog
\begin{equation*}
\bp = \{\tau_0 = \theta, \tau_1, \dots, \tau_{r-1}, \tau_{r} = \beta_N\}.
\end{equation*}
%and for each pair $\tau_{i}, \tau_{i+1},0 \leq i \leq r-1$ there is no $\beta \in \Delta_{+}^{\lambda}$ with $\tau_{i} \prec \beta \prec \tau_{i+1}$.
%Let $\nu_{1}, \dots, \nu_{r} \in \rrp$ be the labels at the edges regarding $\bp$.\\
Let $\nu_{1}, \dots, \nu_{r} \in \rrp$, with $\nu_i\neq\nu_{i+1}$ be the labels at the edges of $\bp$.
%\begin{rem}
%Here we want to illustrate the problem which occurs if we allow k-chains in our Hasse diagram. Let $\gamma\prec\beta\prec\delta$ the roots of a k-chain $\gamma\xrightarrow{k}\beta\xrightarrow{k}\delta$ and consider for $\ell \geq 2$:
%\begin{equation}\label{probichain}
%\partial_k^2f_{\gamma}^{\ell}= \partial_k (\ell f_{\beta}^1f_{\gamma}^{\ell-1})=\underbrace{c_0 \ell f_{\delta}^1f_{\beta}^0f_{\gamma}^{\ell-1}}_{\text{maximal monomial}}+\ c_1 \ell(\ell-1) f_{\beta}^2f_{\gamma}^{\ell-2},
%\end{equation}
%with $c_0 = c_{\gamma,\al_k}c_{\beta,\al_k}$ and $c_1= c_{\gamma,\al_k}^2$ where $c_{\gamma,\al_k}, c_{\beta,\al_k}$ are the structure constants corresponding to $[e_{\al_k},f_{\beta}]$ and $[e_{\al_k},f_{\gamma}]$ respectively.
%So it is more involved to find a relation which contains $\beta$ and $\delta$.
%\end{rem}
We consider $f_{_{\theta}}^{^{\text{\scriptsize{$s_{\tau_0}+ \dots +s_{\tau_{r}}$}}}}\!\! \in I_{\lambda}$.
%We consider $f_{_{\theta}}^{{\text{\scriptsize{$s_{0}$+ \dots +$s_{{r}}$}}}} \in \imw$. 
%We consider $f_{\theta}^{s_0 + \dots + s_r} \in \imw$.
Because $I_{\lambda}$ is stable under $\circ$, we have for arbitrary $x_1, \, \dots, \, x_l \in \Delta_{+}$ and $f^{\bt} \in I_{\lambda}$:
$$
\partial_{x_1} \dots \partial_{x_l} f^{\bt} \in I_{\lambda}.
$$
We define 
\begin{equation}\label{A}
A := \partial_{_{\nu_r}}^{^{\text{\scriptsize{$s_{\tau_{r}}$}}}} \dots \, \partial_{_{\nu_2}}^{^{\text{\scriptsize{$s_{\tau_2}+ \dots +s_{\tau_{r}}$}}}}\partial_{_{\nu_1}}^{^{\text{\scriptsize{$s_{\tau_1}+ \dots +s_{\tau_{r}}$}}}} f_{_{\theta}}^{^{\text{\scriptsize{$s_{\tau_0}+ \dots +s_{\tau_{r}}$}}}} \in I_{\lambda}.
\end{equation}
\textbf{Claim:} There exist constants $c_{\bs} \neq 0, c_{\bt} \in \C, \bt \in \Z_{\geq 0}^N$ with $\bt \prec \bs$, such that: 
\begin{equation}\label{claim}
A = c_{\bs} f^{\bs} + \underset{{\bt} \, \prec \, \bs}{\sum} \, c_{\bt} f^{\bt} \in I_{\lambda}
\end{equation}
If the claim holds the Proposition is proven.\\[0.15cm]
%\begin{proof}[of the Claim]
$Proof$ $of$ $the$ $claim$. Now we need the explicit description of the Dyck paths given by the Hasse diagram.  Above we defined $\nu_{1}$ to be the label at the edge $\theta \overset{\nu_1}{\longrightarrow} \tau_1$ in $H(\Ln_{\lambda}^-)_{\Lg}$. Because we assumed that $H(\Ln_{\lambda}^-)_{\Lg}$ has no $\nu_1$-chains of length 2, there is no edge labeled by $\nu_{1}$ starting in the vertex $\theta - \nu_{1} = \tau_1$. That means $\partial_{_{\nu_1}}^{^{}} f_{_{\theta- \nu_{1}}}^{^{}} = 0$. Therefore we obtain
\begin{equation*}\label{stepone}
\partial_{_{\nu_1}}^{^{\text{\scriptsize{$s_{\tau_1}+ \dots +s_{\tau_{r}}$}}}} f_{_{\theta}}^{^{\text{\scriptsize{$s_{\tau_0}+ \dots +s_{\tau_{r}}$}}}} = a_0 \, f_{_{\theta}}^{^{\text{\scriptsize{$s_{\tau_0}$}}}} f_{_{\theta -  \nu_{1}}}^{^{\text{\scriptsize{$s_{\tau_1}+ \dots +s_{\tau_{r}}$}}}} \in I_{\lambda}
\end{equation*}
for some constant $a_0 \in \C\setminus\{0\}.$ 
%At this point we need the assumption that there are no chains in $\Hgw$. If there were a chain beginning in $\theta$, we would have at least two edges labeld by $\al_{1,1}$, one leaving $\theta$ and one leaving $\theta-\al_{1,1}$. In terms of differential operators that would mean:
%%\begin{equation*}
%$\partial_{_{1,1}}^{^{}} f_{_{\theta}}^{^{}} \neq 0 \text{ and } \partial_{_{1,1}}^{^{}} f_{_{\theta- \al_{1,1}}}^{^{}} \neq 0$
%%\end{equation*}
%and so
%\begin{equation*}
%\partial_{_{1,1}}^{^{2}} f_{_{\theta}}^{^{2}} = f_{_{\theta}}^{^{}}f_{_{\theta - \al_{1,1} - \al_{1,1}}}^{^{}} + f_{_{\theta-\alpha_{1,1}}}^{^{2}}.
%\end{equation*}
%But $\theta - \al_{1,1} - \al_{1,1} \leq \theta - \al_{1,1}$ as $f_{_{\theta - \al_{1,1} - \al_{1,1}}}^{^{}} \succ f_{_{\theta - \al_{1,1}}}^{^{}}$ and because $\succ$ is a lexicographical ordering, $f_{_{\theta}}^{^{}}f_{_{\theta - \al_{1,1} - \al_{1,1}}}^{^{}} \succ f_{_{\theta-\alpha_{1,1}}}^{^{2}}$. So there would be no way to have $f^s$, in particular no monomial with a root vector corresponding to $\theta-\al_{1,1}$(with power greater one), as the largest element in a linear combination of monomials obtained by differential operators, coming from simple roots, like claimed in $\eqref{claim}$.
Now $\nu_{2}$ is the label at the edge between the vertices $\tau_1$ and $\tau_2$. Again there is no $\nu_2$-chain in $H(\Ln_{\lambda}^-)_{\Lg}$, so $\partial_{\nu_2} f_{_{\theta -  \nu_{1} - \nu_{2}}} = 0$ and $\partial_{\nu_2} f_{_{\theta - \nu_{2}}} = 0$, so we have for $k = \min\left\{s_{\tau_0},s_{\tau_2}+ \dots + s_{\tau_{r}}\right\}$, $b_q \in \C\setminus\{0\}$:
\begin{equation}\label{smallerterms}
\begin{aligned}
%\partial_{_{12}}^{^{\text{\scriptsize{$s_{\tau_3}$+ \dots +$s_{\tau_{j}}$}}}} \partial_{_{11}}^{^{\text{\scriptsize{$s_{\tau_2}$+ \dots +$s_{\tau_{j}}$}}}} f_{_{\theta}}^{^{\text{\scriptsize{$s_{\tau_1}$+ \dots +$s_{\tau_{j}}$}}}} &=\\
&\partial_{_{\nu_2}}^{^{\text{\scriptsize{$s_{\tau_2}+ \dots +s_{\tau_{r}}$}}}} a_0 f_{_{\theta}}^{^{\text{\scriptsize{$s_{\tau_0}$}}}} f_{_{\theta -  \nu_{1}}}^{^{\text{\scriptsize{$s_{\tau_1}+ \dots +s_{\tau_{r}}$}}}}=\\
&b_0\,f_{_{\theta}}^{^{\text{\scriptsize{$s_{\tau_0}$}}}} f_{_{\theta -  \nu_{1}}}^{^{\text{\scriptsize{$s_{\tau_1}$}}}}f_{_{\theta -  \nu_{1} - \nu_{2}}}^{^{\text{\scriptsize{$s_{\tau_2}+ \dots +s_{\tau_{r}}$}}}} + \overset{k}{\underset{q \, = \, 1}{\sum}}b_q f_{_{\theta}}^{^{\text{\scriptsize{$s_{\tau_0} \!- q$ }}}} \!\! f_{_{\theta -  \nu_{1}}}^{^{\text{\scriptsize{$s_{\tau_1}\! + q$}}}}f_{_{\theta -  \nu_{1} - \nu_{2}}}^{^{\text{\scriptsize{$s_{\tau_2}+ \dots +s_{\tau_{r}} \!- \! q$}}}}f_{_{\theta - \nu_{2}}}^{^{\text{\scriptsize{$q$}}}}\!.
\end{aligned}
\end{equation}
%where $k = \min\left\{s_{\tau_0},s_{\tau_2}+ \dots + s_{\tau_{r}}\right\}$, $b_q \in \C\setminus\{0\}$. 
%These constants depend on $a_0$, but the order on the monomials do not depend on scalars.
For our purposes, we do not need to pay attention to the scalars unless they are zero. We also notice that the terms of the sum are only non-zero, if $\theta-\nu_{2} \in\rrl$.\\
%We want to show that the monomial $f_{_{\theta}}^{^{\text{\scriptsize{$s_{\tau_0}$}}}} f_{_{\theta -  \nu_{1}}}^{^{\text{\scriptsize{$s_{\tau_1}$}}}}f_{_{\theta -  \nu_{1} - \nu_{2}}}^{^{\text{\scriptsize{$s_{\tau_2} + \dots + s_{\tau_{r}}$}}}}$ is the largest (with respect to $\prec$) in $\eqref{smallerterms}$.
The first part of Lemma $\ref{diffoperatordescription}$ implies, that the monomial $f_{_{\theta}}^{^{\text{\scriptsize{$s_{\tau_0}$}}}} f_{_{\theta -  \nu_{1}}}^{^{\text{\scriptsize{$s_{\tau_1}$}}}}f_{_{\theta -  \nu_{1} - \nu_{2}}}^{^{\text{\scriptsize{$s_{\tau_2}$+ \dots +$s_{\tau_{r}}$}}}}$ is the largest (with respect to $\prec$) in $\eqref{smallerterms}$, because $\theta \prec \theta - \nu_1 \prec \theta - \nu_1 - \nu_2.$\\
By construction $\partial_{\nu_{i+1}} f_{\theta -  \nu_{1} - \nu_{2} - \dots - \nu_{i}}\neq 0$, because $\theta -  \nu_{1} - \nu_{2} - \dots - \nu_{i} - \nu_{i+1}$ is an element of $\Delta_{+}^{\lambda}$, for $i < r.$ So the second statement of Lemma $\ref{diffoperatordescription}$ implies that the largest element is obtained by acting in each step on the largest root vector. To be more precise, we consider the following equations:
\begin{equation*}
\begin{aligned}
\partial_{_{\nu_r}}^{^{\text{\scriptsize{$s_{\tau_{r}}$}}}}\dots \, \partial_{_{\nu_2}}^{^{\text{\scriptsize{$s_{\tau_2}+ \dots +s_{\tau_{r}}$}}}}\partial_{_{\nu_1}}^{^{\text{\scriptsize{$s_{\tau_1}+ \dots +s_{\tau_{r}}$}}}} f_{_{\theta}}^{^{\text{\scriptsize{$s_{\tau_0}+ \dots +s_{\tau_{r}}$}}}}
&=\\
a_0 \,\partial_{_{\nu_r}}^{^{\text{\scriptsize{$s_{\tau_{r}}$}}}}\dots \, \partial_{_{\nu_2}}^{^{\text{\scriptsize{$s_{\tau_2}+ \dots +s_{\tau_{r}}$}}}}
f_{_{\theta}}^{^{\text{\scriptsize{$s_{\tau_0}$}}}} f_{_{\theta -  \nu_1}}^{^{\text{\scriptsize{$s_{\tau_1}+ \dots +s_{\tau_{r}}$}}}}
&=\\
b_0 \, \partial_{_{\nu_r}}^{^{\text{\scriptsize{$s_{\tau_{r}}$}}}}\dots \, \partial_{_{\nu_3}}^{^{\text{\scriptsize{$s_{\tau_3}+ \dots +s_{\tau_{r}}$}}}}
f_{_{\theta}}^{^{\text{\scriptsize{$s_{\tau_0}$}}}} f_{_{\theta -  \nu_1}}^{^{\text{\scriptsize{$s_{\tau_1}$}}}}f_{_{\theta -  \nu_1 - \nu_2}}^{^{\text{\scriptsize{$s_{\tau_2}+ \dots +s_{\tau_{r}}$}}}}
&+ \sum{\text{smaller monomials}} =\\[-0.3cm]
&\begin{array}{l}
.\\[-0.1cm]
.\\[-0.1cm]
.\\[-0.1cm]
\end{array}\\[-0.1cm]
b'_0 \, f_{_{\theta}}^{^{\text{\scriptsize{$s_{\tau_0}$}}}} f_{_{\theta -  \nu_1}}^{^{\text{\scriptsize{$s_{\tau_1}$}}}}f_{_{\theta -  \nu_1 - \nu_2}}^{^{\text{\scriptsize{$s_{\tau_2}$}}}} \dots f_{_{\theta - \nu_1 - \nu_2 - \dots - \nu_r}}^{^{\text{\scriptsize{$s_{\tau_{r}}$}}}} &+ \sum{\text{smaller monomials}} \in I_{\lambda}.\\
%}
% \, \, \, b_{_l} \, f_{_1}^{^{\text{\scriptsize{$s_1$-$l$}}}} f_{_{12}}^{^{\text{\scriptsize{$s_{12}$+$l$}}}} f_{_{1112}}^{^{\text{\scriptsize{$l$}}}} f_{_{11122}}^{^{\text{\scriptsize{$s_{112}$+$s_{11122}$-$l$}}}},
\end{aligned}
\end{equation*}
for some $b'_0 \in \C\setminus\{0\}$. But the last term is exactly what we wanted to obtain, so for constants $c_{\bt} \in \C$, $c_{\bs} \in \C\setminus\{0\}$ we have by assumption that $s_{\alpha} = 0$ if $\al \notin \bp$:
\begin{equation*}
\begin{aligned}
\partial_{_{\nu_r}}^{^{\text{\scriptsize{$s_{\tau_{r}}$}}}}\dots \, \partial_{_{\nu_2}}^{^{\text{\scriptsize{$s_{\tau_2}+ \dots +s_{\tau_{r}}$}}}}\partial_{_{\nu_1}}^{^{\text{\scriptsize{$s_{\tau_1}+ \dots +s_{\tau_{r}}$}}}} f_{_{\theta}}^{^{\text{\scriptsize{$s_{\tau_0}+ \dots +s_{\tau_{r}}$}}}}
&=\\
c_{\bs}f_{_{\theta}}^{^{\text{\scriptsize{$s_{\tau_0}$}}}} f_{_{\tau_1}}^{^{\text{\scriptsize{$s_{\tau_1}$}}}}f_{_{\tau_2}}^{^{\text{\scriptsize{$s_{\tau_2}$}}}} \dots f_{_{\tau_{r}}}^{^{\text{\scriptsize{$s_{\tau_{r}}$}}}} + \underset{\bt \, \prec \, \bs}{\sum} \, c_{\bt} f^{\bt} &= \\
c_{\bs}f^{\bs} + \underset{\bt \, \prec \, \bs}{\sum} \, c_{\bt} f^{\bt} &\in I_{\lambda}.
%}
% \, \, \, b_{_l} \, f_{_1}^{^{\text{\scriptsize{$s_1$-$l$}}}} f_{_{12}}^{^{\text{\scriptsize{$s_{12}$+$l$}}}} f_{_{1112}}^{^{\text{\scriptsize{$l$}}}} f_{_{11122}}^{^{\text{\scriptsize{$s_{112}$+$s_{11122}$-$l$}}}},
\end{aligned}
\end{equation*}
%for some constant $c_{\bt} \in \C$. This proves the claim. 
%The Dyck path was chosen randomly, so the proof holds for every $p \in D$ and the Proposition is proven.
\end{proof}
\begin{thm}\label{thmspan} 
%If there exists for every multi-exponent $\bq \in \Z_{\geq 0}^N$, $\bq$ supported on some $\bp \in D_{\lambda}$ and $\bq \notin \smw$, a $straightening$ $law$ (see \ref{straighteninglaw}), then 
The set  \{$f^{\bs}v_{\lambda} \mid \bs \in \smw\}$ spans the module $\Vmw^a$.
\end{thm}
\begin{proof}
Let $m \in \Z_{\geq 0}$ and $\bt \in \Z_{\geq 0}^N$ with $\bt \notin \smw$. That means there exists a Dyck path $\bp \in D_{\lambda}$ such that $\underset{\beta \, \in \, \bp}{\sum} t_\beta > m.$ Define a new multi-exponent $\bt'$ by
\begin{equation*}
t_{\beta}' :=
\begin{cases}
t_{\beta}, \, &\textrm{if} \, \beta \in \bp,\\
0, \,& \textrm{else.}
\end{cases}
\end{equation*}
Because of $\underset{\beta \, \in \, \bp}{\sum} t_\beta' = \underset{\beta \, \in \, \bp}{\sum} t_\beta > m$ we can apply
Proposition $\ref{important}$
to $\bt'$ and get 
\begin{equation*}
f^{\bt'} = \underset{\bs' \, \prec \, \bt'}{\sum} c_{\bs'}f^{\bs'} \in \Snml/I_{\lambda},
\end{equation*}
for some $c_{\bs'} \in \C$. Because the order of the factors of $f^{\bt} \in \Snml$ is arbitrary and since we have a monomial order, we get 
\begin{equation}\label{algorithm}
f^{\bt} = f^{\bt'} \underset{\beta \, \notin \, \bp}{\prod} f_{\beta}^{t_\beta} = \underset{\bs \, \prec \, \, \bt}{\sum} c_{\bs} f^{\bs} \in \Snml/I_{\lambda},
\end{equation}
where $c_{\bs} = c_{\bs'}$ and $f^{\bs} = f^{\bs'}\prod_{\beta\,\notin\, \bp}f_{\beta}^{s_\beta}.$ Equation $\eqref{algorithm}$ shows that we can express an arbitrary multi-exponent as a sum of strictly smaller multi-exponents. We repeat this procedure until all multi-exponents in the sum lie in $\smw$. There are only finitely many multi-exponents of a fixed degree and the degree is invariant or zero under the action $\circ$. So after a finite number of steps, we can express $\bt$ in terms of $\br \in \smw$ for some $c_{\br} \in \C$:
\begin{equation*}
f^{\bt} = \underset{\br \, \in \, \smw}{\sum} c_{\br} f^{\br} \in \Snml/ I_{\lambda}.
\end{equation*}
\end{proof}
\begin{cor}\label{Spanvmw} Fix for every $\bs \in \smw$ an arbitrary ordering of the factors $f_\beta$ in the product $\prod_{\beta \, > \, 0} f_{\beta}^{s_\beta} \in \Snml.$ Let $f^{\bs} =  \prod_{\beta \, > \, 0}^{} f_{\beta}^{s_\beta} \in \Un$ be the ordered product. Then the elements $f^{\bs}v_{\omega}, \bs \in \smw$ span the module $\Vmw$.
\end{cor}
\begin{proof}
Let $f^{\bt}\vmw \in \Vmw$ with $\bt \in \Z_{\geq}^N$ arbitrary. We consider $f^{\bt}\vl$ as an element in $\Vmwa$. By Theorem $\ref{thmspan}$ we get
\begin{equation*}
f^{\bt}\vmw = \underset{\bs \, \in \, \smw}{\sum} c_{\bs}f^{\bs}\vmw \, \,\text{in} \, \, \Vmwa.
\end{equation*}
The ordering of the factors in a product in $\Snml$ is irrelevant, so we can adjust the ordering of the factors to the fixed ordering and get an induced linear combination:
\begin{equation*}
f^{\bt}\vmw = \underset{\bs \, \in \, \smw}{\sum} c_{\bs}f^{\bs}\vmw \, \,\text{in} \, \, \Vmw.
\end{equation*} 
\end{proof}

%% file: PBW-Basis_Types.tex
\section{FFL Basis of \texorpdfstring{$V(\omega)$}{V(w)}}\label{uebergreifend2}
Throughout this section we refer to the definitions in Subsection $\ref{Definitions}$. In this section we calculate explicit FFL bases of the highest weight modules $\Vw$, where $\omega$ occurs in Table $\ref{fundamentalweights}$. We will do this by giving characterizations of the co-chains $\overline{\bp}\in \dsw$ (see (\ref{dstrich})) and using the one-to-one correspondence 
between $\dsw$ and $S(\omega)$ (see Proposition $\ref{bijSWDL}$).\\ The results of this section, i.e. $\B_\omega = \{f^\bs\vw \mid \bs \in S(\omega)\}$ is a FFL basis of $\Vw$, provide the start of an inductive procedure in the proof of Theorem $\ref{theoremBB}$. With Proposition \ref{helpfullprop} we will be able to give an explicit basis of $V(m\omega)$, $m\in\mathbb \Z_{\geq 0}$, parametrized by the $m$-th Minkowski sum of $S(\w)$.

\subsection{Type \texorpdfstring{$\mathtt{A_n}$}{An}}
Let $\Lg$ be a simple Lie algebra of type $\mathtt{A_n}$ with $n\geq1$ and the associated Dynkin diagram
\begin{center}
\begin{tikzpicture}[scale=.4]
    \draw (-1,0) node[anchor=east]  {$\mathtt{A_n}$};
    \foreach \x in {0,...,4}
    \draw[thick,xshift=\x cm] (\x cm,0) circle (1 mm);
    \foreach \y in {0,...,2}
    \draw[thick,xshift=\y cm] (\y cm,0) ++(.3 cm, 0) -- +(14 mm,0);

    %\draw[thick,dashed] (4 cm,2 cm) circle (1 mm);
    %\draw[thick] (4 cm, 3mm) -- +(0, 1.4 cm);
    \draw (0,0) node[anchor=north]  {\tiny{1}};
    \draw (2,0) node[anchor=north]  {\tiny{2}};
    \draw (4,0) node[anchor=north]  {\tiny{3}};
    \draw (6,0) node[anchor=north]  {\tiny{4}};
    \draw (8,-0.1) node[anchor=north]  {\tiny{n}};
    \draw[thick,dashed] (6.3,0)--(7.7,0);
    %\draw (4,3) node[anchor=north]  {\tiny{2}};
  \end{tikzpicture}
\end{center}
The highest root is of the form $\theta=\sum_{i=1}^n \alpha_i$.
Since  a Lie algebra $\mathfrak{g}$ of type $\mathtt{A_n}$ is simply laced  we have $\theta^{\vee}= \sum_{i=1}^n \alpha_i^{\vee}$ and so $\langle \omega,\theta^{\vee}\rangle=1\Leftrightarrow \omega\in \{\wk\mid\ 1\leq k\leq n\}$. The positive roots of $\mathfrak{g}$ are described by: $\Delta_+=\{\alpha_{i,j}=\sum_{l=i}^j \al_l \mid 1\leq i\leq j\leq n \}$. So for the roots corresponding to $\Ln_{\w_k}^-$ we have:
\begin{equation}\label{relrootsAn}
\rrk=\{\alpha_{i,j}\in \Delta_+|\ 1\leq i\leq k\leq j\leq n\}\subset \Delta_+.
\end{equation}
%\begin{center}
%\begin{tikzpicture}
  %\node (1) at (0,0) {};
  %\fill[black] (1) circle (1pt);
  %\draw (0,0) node[anchor=south]  {\tiny{$\beta_1$}};
  %\node (2) at (-0.5,-0.5) {};
  %\fill[black] (2) circle (1pt);
  %\node (3) at (0.5,-0.5) {};
  %\fill[black] (3) circle (1pt);
  %\node (4) at (0.0,-1) {};
  %\fill[black] (4) circle (1pt);
  %\node (5) at (1,-1) {};
  %\fill[black] (5) circle (1pt);
  %\node (6) at (-1,-1) {};
  %\fill[black] (6) circle (1pt);
  %\node (7) at (0,-1.5) {.  .  .};
  %%\fill[black] (7) circle (1pt);
  %\node (8) at (0.5,-2) {};
  %\fill[black] (8) circle (1pt);
  %\node (9) at (-0.5,-2) {};
  %\fill[black] (9) circle (1pt);
  %\node (10) at (-1.5,-2) {};
 %\fill[black] (10) circle (1pt);
 %\node (11) at (1.5,-2) {};
  %\fill[black] (11) circle (1pt);
  %\node (12) at (-1,-2) {...};
  %\node (13) at (1,-2) {...};
%%  \fill[black] (12) circle (1pt);
%\end{tikzpicture}
%\end{center}
Before we define the total order on $\rrk$, we define a total order on $\Delta_+$:
\begin{eqnarray*}
  &\beta_1=\al_{1,n},&\\
  &\beta_2=\al_{2,n},\ \beta_3=\al_{1,n-1},&\\
	&\beta_4=\al_{3,n},\ \beta_5=\al_{2,n-1},\ \beta_6=\al_{1,n-2},&\\
  &\cdots,&\\
	&\beta_{n(n-1)/2+1}=\al_{n},\ \beta_{n(n-1)/2+2}=\al_{n-1},\cdots,\ \beta_{n(n+1)/2}=\al_{1}.&
\end{eqnarray*}
Now we delete every root $\beta_i\in\Delta_+\setminus\rrk$ and relabel the remaining roots. 
%in the Hasse diagram from top to bottom and from left to right with increasing integers to get a total order on $\rrk$.
%\begin{align*}
%\ail\prec\aim :\Leftrightarrow &\ \he(\ail)<\he(\aim)\\
%\mathrm{or}&\ \he(\ail)=\he(\aim)\ \mathrm{and}\ i_l>i_m.
%\end{align*}
For an example of this procedure see Appendix, Figure \ref{hasseA4} and Example \ref{exaHdA4}. In the following it is more convenient to use the description $\al_{i,j}$ instead of $\beta_k$. First we give a characterization of the co-chains $\bps\in\dswk\subset\prk$.
\begin{prop}
Let be $\bps=\{\alpha_{i_1,j_1},\dots, \alpha_{i_s,j_s}\}\in\prk$ arbitrary, then:
\begin{equation}\label{dspropAn}
\bps\in\dswk \Leftrightarrow \forall \alpha_{i_l,j_l},\alpha_{i_m,j_m}\in \bps,\ i_l\leq i_m: i_l<i_m \leq k \leq j_l<j_m.
\end{equation}
%\begin{equation}\label{dspropAn}
%\bps\in\dswk \Leftrightarrow i_1 < i_2 < \dots < i_s \leq k \leq j_1 < j_2 < \dots < j_s. 
%\end{equation}
Further we have: $\bps\in\dswk \Rightarrow s\leq \mathrm{min}\{k,n+1-k\}$.
\end{prop}
\begin{proof}
First we prove $\eqref{dspropAn}$: ``$\Leftarrow$'': Let $\bps=\{\alpha_{i_1,j_1},\dots, \alpha_{i_s,j_s}\}\in\prk$ be an element with the properties of the right-hand side (rhs) of (\ref{dspropAn}). Let $\ail,\aim\in \bps$, with $i_l<i_m$. Consider now:
\begin{equation*}
\ail-\aim= \sum_{r=i_l}^{j_l}\alpha_r-\sum_{r=i_m}^{j_m}\alpha_r=\sum_{r=i_l}^{i_m-1}\alpha_r-\sum_{r=j_l+1}^{j_m}\alpha_r.
\end{equation*}
Since $j_l<j_m$ holds, Remark \ref{defdyck} implies that there is no Dyck path $\bq\in D_{\omega_k}$ such that $\aim$ and $\ail$ are contained in $\bq$.
\vspace{2mm}

``$\Rightarrow$'': Let be $\bps\in \dswk$ and $\ail,\aim\in \bps$ with $\ail\neq \aim$. Further we have $i_l \leq j_l, i_m \leq j_m$. Assume wlog $i_m = j_m$, then $\al_{i_m,j_m} = \al_{k}$ and $i_l < j_l$. Hence
\begin{equation*}\label{step0}
\alpha_{i_l,j_l}-\alpha_{k}= \sum_{r=i_l}^{k-1}\alpha_r + \sum_{r=k+1}^{j_l}\alpha_r,
\end{equation*}
which is a contradiction to $\bps\in\dswk$ by Remark \ref{defdyck}. So $i_l < j_l, i_m < j_m$ and we assume wlog $i_l\leq i_m$.

\vspace{2mm}

\textbf{1. Step:} $i_l=i_m=:y$. Set $x=\min\{j_l,j_m\}$ and $\overline{x}=\max\{j_l,j_m\}$:
\begin{equation*}
\alpha_{y,\overline{x}}-\alpha_{y,x}= \sum_{r=y}^{\overline{x}}\alpha_r-\sum_{r=y}^{x}\alpha_r=\sum_{r=x+1}^{\overline{x}}\alpha_r.
\end{equation*}
Again this contradicts to $\bps\in\dswk$. Hence we have: $i_l<i_m$.
\vspace{2mm}

\textbf{2. Step:} $(i_l<i_m)\wedge (j_l=j_m=:x)$:
\begin{equation*}\label{step2}
\alpha_{i_l,x}-\alpha_{i_m,x}= \sum_{r=i_l}^{x}\alpha_r-\sum_{r=i_m}^{x}\alpha_r=\sum_{r=i_l}^{i_m-1}\alpha_r.
\end{equation*}
We conclude: $j_l\neq j_m$.
\vspace{2mm}

\textbf{3. Step:} $(i_l<i_m<j_m)\wedge (i_l<j_l)$. So there are three possible cases:
\begin{center}
(a) $i_l<j_l<i_m<j_m$, (b) $i_l<i_m<j_l<j_m$ and (c) $i_l<i_m<j_m<j_l$.
\end{center}
The case (a) can not occur because $k\leq j_l<i_m\leq k$ is a contradiction.
So let us assume $\ail,\aim$ satisfy the case (c), then we have:
\begin{equation*}
\alpha_{i_l,j_l}-\alpha_{i_m,j_m}= \sum_{r=i_l}^{j_l}\alpha_r-\sum_{r=i_m}^{j_m}\alpha_r=\sum_{r=i_l}^{i_m-1}\alpha_r+\sum_{r=j_m}^{j_l}\alpha_r.
\end{equation*}
Finally we conclude that for two arbitrary roots $\ail,\aim\in \bps\in \dswk$ with $i_l\leq i_m$ we have: $i_l<i_m<j_l<j_m$.\\[2mm] It remains to show that the cardinality $s$ of $\bps$ is bounded by $\min\{k,n+1-k\}$:
\vspace{2mm}

\textbf{1. Case:} $\min\{k,n+1-k\}=k$. Let $\air\in \bps$ be an arbitrary root in $\bps$. Then we know from (\ref{relrootsAn}) $1\leq i_r\leq k$. But we also know that for any two roots $\ail,\aim\in \bps$ we have $i_l\neq i_m$. So there are at most $k$ different roots in $\bps$.
\vspace{2mm}

\textbf{2. Case:} $\min\{k,n+1-k\}=n+1-k$. For two roots ${\ail,\aim\in\bps}$ we have $j_l\neq j_m$ and $k\leq j_l,j_m\leq n$. So the number of different roots in $\bps$ is bounded by $n+1-k$.
\vspace{2mm}

Finally we conclude: $|\bps|=s\leq \min\{k,n+1-k\}$.
\end{proof}
\begin{rem}\label{remprop}
Let $\overline{\bp}=\{\alpha_{i_1,j_1},\dots, \alpha_{i_s,j_s}\} \in \dswk$ then $\eqref{dspropAn}$ implies $$i_1 < i_2 < \dots < i_s \leq k \leq j_1 < j_2 < \dots < j_s.$$
Assume wlog $k = j_1 = j_2$, then there is Dyck path containing $\al_{i_1,j_1}$ and $\al_{i_2,j_2}$, because $\al_{i_1,j_1} - \al_{i_2,j_2} = \al_{i_1,i_2-1} \in \Delta_+$.
\end{rem}
Because of Corollary \ref{Spanvmw} we know that the elements $\{f^{\bs}\vwk \mid \bs \in S(\omega_k)\}$ span $V(\omega_k)$ and by Proposition $\ref{bijSWDL}$ there is a bijection between $S(\omega_k)$ and $\dswk$.
% span the highest weight $\mathtt{A_n}$-module $\Vw$.
We want to show that these elements are linear independent. To achieve that we will show that $|\dswk|=\dim\Vwk$. To be more explicit:
\begin{prop}\label{helpAn}
For all $1\leq k\leq n$ we have: $|\dswk|=\dim\Vwk=\binom{n+1}{k}$.
\end{prop}
\begin{proof} Let $V(\w_1)$ be the vector representation with basis $\{e_1,e_2,\dots,e_{n+1}\}$. Then $\bigwedge^k V(\omega_1)$ is a $U(\Lg)$-representation with $v_{\w_k}=e_1 \wedge e_2 \wedge \dots \wedge e_k$:
\begin{equation}\label{wedge}
f_{\al_{i_1,j_1}}v_{\w_k} = e_1 \wedge \dots \wedge e_{i_1 - 1} \wedge e_{j_1 + 1} \wedge e_{i_1 + 1} \wedge \dots \wedge e_k,
\end{equation}
and we have $\bigwedge^k V(\omega_1) \cong V(\w_k)$. We define $f_{\bps}v_{\w_k}:=f_{\al_{i_1,j_1}}f_{\al_{i_2,j_2}} \dots f_{\al_{i_m,j_m}}v_{\w_k}$ for $\bps=\{\al_{i_1,j_1},\al_{i_2,j_2}, \dots, \al_{i_m,j_m}\} \in \dswk$ and claim that the set $\{f_{\bps}v_{\w_k}\mid \bps\in \dswk\}$ is linear independent in $\bigwedge^k V(\omega_1)$. If the claim holds we have $|\dswk| \leq \dim V(\w_k)$ and with Corollary $\ref{Spanvmw}$ we conclude that $|\dswk| = \dim V(\w_k)=\binom{n+1}{k}$.\\[2mm]
$Proof$ $of$ $the$ $claim$. Assume we have $\bps_1 = \{\al_{i_1,j_1},\al_{i_2,j_2}, \dots, \al_{i_m,j_m}\}$ and $\bps_2=\{\al_{s_1,t_1},\al_{s_2,t_2}, \dots, \al_{s_{\ell},t_{\ell}}\}$ in $\dswk$ with linear dependent images under the action $\eqref{wedge}$, i. e. $f_{\bps_1}v_{\w_k}=\pm f_{\bps_2}v_{\w_k}$. Then we have $m=\ell$, $\{j_1,\dots,j_m\}=\{t_1,\dots,t_\ell\}$ and we can assume wlog: $m=k=\ell$. Hence: $f_{\bps_1}v_{\w_k}=e_{j_1}\wedge\dots\wedge e_{j_m}=\pm f_{\bps_2}v_{\w_k}$, with Remark \ref{remprop} we conclude $\bps_1=\bps_2$.  
\end{proof}
\begin{exa} The non-redundant inequalities of the polytope $P(m\omega_3)$ in the case $\Lg = \mathfrak{sl}_{5}$ are:
\begin{equation*}
P(m\w_3)=\left\{\bx\in\R^6_{\geq 0}\mid
\begin{aligned}
&x_1+x_2+x_4+x_6\leq m\\
&x_1+x_2+x_5+x_6\leq m\\
&x_1+x_3+x_5+x_6\leq m
\end{aligned}
\right\}.
\end{equation*}
Example \ref{exaHdA4} shows the corresponding Hasse diagram $H(\Ln_{\w_3}^-)_{\mathfrak{sl}_5}$.
\end{exa}
Proposition $\ref{helpAn}$ implies immediately for $1\leq k\leq n$:
\begin{prop}\label{basisAn}
The vectors $f^{\bs}v_{\w_k}, \bs \in S(\w_k)$ are a FFL basis of $V(\w_k)$. $\!\hfill \Box$
%The set $\B_{\omega_k} = \{f^{\bs}v_{\w_k} \mid \bs \in S(\w_k)\}$ is a FFL basis for $V(\w_k)$.\hfill $\Box$
\end{prop}
%With Prop. \ref{helpfullprop} we are able to give an explicit PBW-basis for $V(m\omega_k)$, for $m\in\mathbb N$ and $1\leq k\leq n$, by taking the $m$-th Minkowski sum of $\dswk$.
\subsection{Type \texorpdfstring{$\mathtt{B_n}$}{Bn}}
Let $\Lg$ be a simple Lie algebra of type $\mathtt{B_n}, n \geq 2$ with associated Dynkin diagram
\begin{center}
\begin{tikzpicture}[scale=.4]
    \draw (-1,0) node[anchor=east]  {$\mathtt{B_n}$};
    \foreach \x in {0,...,4}
    \draw[thick,xshift=\x cm] (\x cm,0) circle (1 mm);
    \foreach \y in {0,...,0,2}
    \draw[thick,xshift=\y cm] (\y cm,0) ++(.3 cm, 0) -- +(14 mm,0);

		\draw[thick] (6.3,0.1) -- +(1.4 cm, 0);
		\draw[thick] (6.3,-0.1) -- +(1.4 cm, 0);

		\draw (7.7,0) node[anchor=east] {\small{{$\bf\boldsymbol{>}$}}};

    \draw (0,0) node[anchor=north]  {\tiny{1}};
    \draw (2,0) node[anchor=north]  {\tiny{2}};
    \draw (4,0) node[anchor=north]  {\tiny{n-2}};
    \draw (6,0) node[anchor=north]  {\tiny{n-1}};
    \draw (8,-0.1) node[anchor=north]  {\tiny{n}};
    \draw[thick,dashed] (2.3,0)--(3.7,0);
    %\draw (4,3) node[anchor=north]  {\tiny{2}};
  \end{tikzpicture}
\end{center}
The highest root for a Lie algebra of type $\mathtt{B_n}$ is of the form $\theta=\alpha_1 +2\sum_{i=2}^{n}\alpha_i$. So we have $\theta^{\vee}=\alpha_1^{\vee}+2\sum_{i=2}^{n-1}\alpha_i^{\vee}+ \alpha_n^{\vee}$ and $\langle \omega,\theta^{\vee}\rangle=1\Leftrightarrow \omega\in \{\omega_1,\omega_n\}$.\\[2mm]
First we consider the case $\omega=\omega_1$. We want to consider the case $\mathtt{B_2}, w_1$ separately. Because there are not enough roots, this case does not fit in our general description of $\mathtt{B_n}, w_1$.
%but it is easy to calculate the polytope. 
We claim that the following polytope parametrizes a FFL basis of $V(m\omega_1), m \in \Z_{\geq 0}$:
\begin{equation*}
P(m\w_1)=\left\{ \textbf{x} \in \R_{\geq 0}^{3}\mid
\begin{aligned}
x_{2} +  x_{{1}}  \leq m \\
x_{2} +  x_{{3}}  \leq m \\
\end{aligned}
\right\}.
\end{equation*}
We fix $\beta_1 = (2,1), \beta_2 = (1,1), \beta_3 = (1,0)$ and the order $\beta_2 \prec \beta_1 \prec \beta_3$. Then with Proposition $\ref{helpfullprop}$ it is immediate that this polytope is normal. The following actions of the differential operators imply the spanning property in the sense of Section \ref{spanning} Proposition \ref{important}.
\begin{equation*}
\begin{aligned}
\partial_{\al_2}^{s_1} f_1^{s_1 + s_2} =&\  c_0f_1^{s_1} f_2^{s_2} + \textrm{smaller terms} \in I_{\lambda}\\
\partial_{\al_1}^{s_2 + 2s_3} f_1^{s_2 + s_3} =&\ c_1f_2^{s_2}f_3^{s_3} + \textrm{smaller terms} \in I_{\lambda},\ c_i\in\C\setminus\{0\}.
\end{aligned}
\end{equation*}
We conclude that $\{f^{\bs}v_{\w_1} \mid \bs \in S(m\omega_1)\}=\{\vwe,f_{1}\vwe,f_{2}\vwe,f_{3}\vwe,f_{1}f_{3}\vwe, \}$ is a spanning set of $V(\w_1)$.\\[2mm] 
Now we consider the case $n\geq 3$. If we construct $H(\Ln_{\omega_1}^-)_{\Lg}$ as in Section $\ref{test}$ we get a $n$-chain of length $2$. Therefore we choose a new order on the roots and change our Hasse diagram slightly to obtain a diagram without $k$-chains of length 2.
%The new order and the changed diagram shall satisfy the important property that for each pair $\beta_i \prec \beta_j$ and $\nu \in \Delta_{+}$ such that $\nu$ is a label at an edge starting in $\beta_i$ and at an edge starting in $\beta_j$, we have 
%\begin{equation}\label{important2}
%\beta_i \prec \beta_j \Rightarrow \beta_i - \nu \prec \beta_j - \nu.
%\end{equation}
We illustrate this procedure for $\Lg$ of type $\mathtt{B_3}$.
Then the roots $\Delta_{+}^{\omega_1}$ are given by
\begin{center} $\begin{array}{|l|l|l|l|l|}
\hline
\beta_1=(1,2,2) & \beta_2=(1,1,2) & \beta_{3}=(1,1,1) & \beta_4=(1,1,0) & \beta_5=(1,0,0) \\
\hline	
\end{array}$ \end{center}
%Usually we want $\prec$ to be consistent with the standard partial order.
We choose a new order 
\begin{equation*}
\beta_1 \prec \beta_2 \prec \beta_4 \prec \beta_5 \prec \beta_3,
\end{equation*}
and change the Hasse diagram
%\begin{equation*}
%\beta_i \prec \beta_j \Rightarrow \beta_i - \nu \prec \beta_j - \nu.
%\end{equation*}
%Now we consider the Hasse diagram as constructed in Section $\ref{test}$ and our adjustment.
\begin{center}
\begin{tikzpicture}
  \node (1) at (-6.75,4) {$\beta_1$};
  \node (2) at (-5.5,4) {$\beta_2$};
  \node (3) at (-4.25,4) {$\beta_3$};
  \node (4) at (-3,4) {$\beta_4$};
  \node (5) at (-1.75,4) {$\beta_5$};
	\node (6) at (0,4) {$\leadsto$};
	%\node (6) at (-0.5,4) {};
	\node (7) at (0.5,4) {};
	\node (8) at (1.75,4) {$\beta_1$};
  \node (9) at (3,4) {$\beta_3$};
%  \node (10) at (4,4.5) {$\beta_2$};
  \node (11) at (4,5) {$\beta_2$};
  \node (12) at (4,3) {$\beta_4$};
	\node (13) at (5,4) {$\beta_5.$};
%  \node (3) at (0,4) {$\beta_3$};
%  \node (a) at (-2,2) {$(0,1,1)$};
%  \node (b) at (0,2) {$(1,0,1)$};
%  \node (c) at (2,2) {$(1,1,0)$};
%  \node (d) at (-2,0) {$(0,0,1)$};
%  \node (e) at (0,0) {$(0,1,0)$};
%  \node (f) at (2,0) {$(1,0,0)$};
%  \node (min) at (0,-2) {$(0,0,0)$};
   %\draw (1) -- (2);
   \draw[->] (1) edge node[above] {\tiny{2}} (2);
   \draw[->] (2) edge node[above] {\tiny{3}} (3);
   \draw[->] (4) edge node[above] {\tiny{2}} (5);
   \draw[->] (3) edge node[above] {\tiny{3}} (4);

   %\draw[->] (6) --+ (7) ;

   \draw[->] (8) edge node[above] {\tiny{011}} (9);
   \draw[->] (9) edge node[right, pos=0.25] {\tiny{012}} (12);
   \draw[->] (12) edge node[right, pos=0.25] {\tiny{2}} (13);
   \draw[->] (9) edge node[right, pos=0.25] {\tiny{2}} (11);
   \draw[->] (11) edge node[right, pos=0.25] {\tiny{012}} (13);
  % \draw (0,4.5) node[anchor=west]  {\tiny{4}};
%  \draw (min) -- (d) -- (a)  -- (b) -- (f)
%  (e) -- (min) -- (f) -- (c) --  (d) -- (b);
%  \draw[preaction={draw=white, -,line width=6pt}] (a) -- (e) -- (c);
\end{tikzpicture}
\end{center}
First we check, if the new diagram has no $k$-chains. The first edge is labeled by $\al_2 + \al_3 = 011$ and we have $\beta_{3} - (\alpha_2 + \al_3) = \beta_{5}$. If we have a monomial $f_{1}^{k_1} f_{3}^{k_2} \in S(\mathfrak{n}_{\omega_1}^-), k_1,k_2 \geq 1$ and we act by $\partial_{\al_2 + \al_3}$ we get:
\begin{equation*}
c_0f_{1}^{k_1-1} f_{3}^{k_2+1} + c_1f_{1}^{k_1} f_{3}^{k_2 - 1} f_{5},\ c_i\in\C.
\end{equation*}
By the change of order $\beta_{3}$ is larger than $\beta_5$ and so $f_{1}^{k_1-1} f_{3}^{k_2+1} \succ f_{1}^{k_1} f_{3}^{k_2 - 1} f_{5}$.
Therefore we can neglect the edge between $\beta_3$ and $\beta_5$.\\
Now we consider $\partial_{\al_2}^{k_3} f_{1}^{k_1} f_{3}^{k_2}$. Because of $\partial_{\al_2}f_{3},\partial_{\al_2}f_{2}  = 0$ we get $f_{1}^{k_1-k_3} f_{3}^{k_2}f_{2}^{k_3},$ for $k_3 \leq k_1$.
So instead of drawing an edge directly from $\beta_1$ to $\beta_2$, we can draw an edge, labeled by 2, from $\beta_3$ to $\beta_2$. Similar, because of $\beta_1 - \al_2 - 2\al_3=\beta_4$, we can draw an edge labeled by $012$ from $\beta_3$ to $\beta_4$. The other edges do not cause any problems.\\
The second step is to show that the paths in the new diagram, define the actions by differential operators and the corresponding maximal elements like in Section \ref{spanning} Proposition \ref{important}. By the choice of order we get the following equalities:
\begin{equation*}
\begin{aligned}
\partial_{\al_2 + 2\al_3}^{s_5} \partial_2^{s_2} \partial_{\al_2 + \al_3}^{s_3} f_{1}^{s1 + s3 + s2 + s5} &= c_0f_1^{s_1} f_{3}^{s_3} f_2^{s_2} f_5^{s_5} + \textrm{smaller terms} \in I_{\lambda}\\
\partial_{\al_2}^{s_5} \partial_{\al_2 + 2\al_3}^{s_4} \partial_{\al_2 + \al_3}^{s_3} f_{1}^{s1 + s3 + s4 + s5} &= c_1f_1^{s_1} f_{3}^{s_3} f_4^{s_4} f_5^{s_5} + \textrm{smaller terms} \in I_{\lambda},
\end{aligned}
\end{equation*}
%We have to check if $\eqref{important2}$ is satisfied and that the new diagram can be regarded as a diagram without $i$-chains. These properties ensure that we can apply Section \ref{spanning} to this setup and that the proofs of Proposition $\ref{important}$ and Lemma $\ref{diffoperatordescription}$ hold.\\
%%For that purpose we refer to the notation of Section $\ref{spanning}$.\\
%The first edge is labeled by $\al_2 + \al_3 = 011$ and we have $\beta_{3} - (\alpha_2 + \al_3) = \beta_{5}$. If we have a monomial $f_{1}^{k_1} f_{3}^{k_2} \in \Snml, k_1,k_2 \geq 1$ and we act by $\partial_{\al_2 + \al_3}$ we get:
%\begin{equation*}
%f_{1}^{k_1-1} f_{3}^{k_2+1} + f_{1}^{k_1} f_{3}^{k_2 - 1} f_{5}
%\end{equation*}
%By the change of order $\beta_{3}$ is larger than $\beta_5$ and so $f_{1}^{k_1-1} f_{3}^{k_2+1} \succ f_{1}^{k_1} f_{3}^{k_2 - 1} f_{5}$.
%Therefore we can neglect the edge between $\beta_3$ and $\beta_5$ and similar we can neglect the edge from $\beta_2$ to $\beta_3$.\\
%Now we consider $\partial_{\al_2}^{k_3} f_{1}^{k_1} f_{3}^{k_2}$. Because of $\partial_{\al_2}f_{3},\partial_{\al_2}f_{2}  = 0$ we get $f_{1}^{k_1-k_3} f_{3}^{k_2}f_{2}^{k_3},$ for $k_3 \leq k_1$.
%So instead of drawing an edge directly from $\beta_1$ to $\beta_2$, we can draw an edge, labeled by 2, from $\beta_3$ to $\beta_2$. Similar, because of $\beta_1 - \al_2 - 2\al_3=\beta_4$, we can draw an edge labeled by $012$ from $\beta_3$ to $\beta_4$. The other edges do not cause any problems.\\
with $c_i\in\C\setminus\{0\}$. In the general case, for arbitrary $n > 3$, we have $N = 2n - 1$. Let $r := {\lceil N/2 \rceil}$, then $\Delta_{+}^{\omega_1}$ is given by:
\begin{center} $\begin{array}{|l|l|l|l|l|}
\hline
\beta_1=(1,2,2,\dots,2) & \,\,\,\beta_2 \,\,\, = (1,1,2,\dots,2,2) & \,\,\, \dots \,\,\, & \beta_{r - 1}=(1,1,\dots,1,2) \\
\hline
\beta_r=(1,1,1,\dots,1) & \beta_{r+1}=(1,1,1,\dots,1,0) & \,\,\, \dots \,\,\,& \,\, \beta_N \,\,=(1,0,\dots,0,0) \\
\hline	
\end{array}$ \end{center}
Then the only $n$-chain has the following form
%\begin{equation*}
%\beta_{\lceil N/2  \rceil + 1} = \overset{n-1}{\underset{i\, = \, 1}{\sum}} \al_i + 2\al_{n} \overset{n}{\rightarrow} \beta_{\lceil N/2  \rceil } = \al_1 + \dots + \al_n \overset{n}{\rightarrow} \beta_{\lceil N/2  \rceil -1} = \al_1 + \dots + \al_{n-1}
%\end{equation*}
%\begin{equation*}
%\left( \beta_{r - 1} = \overset{n-1}{\underset{i\, = \, 1}{\sum}} \al_i + 2\al_{n} \right) \overset{n}{\longrightarrow} \left( \beta_{r} = \overset{n}{\underset{i\, = \, 1}{\sum}} \al_i\right) \overset{n}{\longrightarrow} \left(\beta_{r + 1} = \overset{n-1}{\underset{i\, = \, 1}{\sum}} \al_i\right).
%\end{equation*}
$\beta_{r - 1}  \overset{n}{\longrightarrow} \beta_{r}  \overset{n}{\longrightarrow} \beta_{r + 1}$
We change the order from $\beta_1 \prec \beta_2 \prec \dots \prec \beta_N$ to
\begin{equation}\label{spanningorderBN}
\beta_1 \prec \beta_2 \prec \dots \prec \beta_{r -1}\prec\beta_{r +2}\preceq \dots \preceq \beta_{N-1} \preceq \beta_{r +1}   \prec \beta_N \prec \beta_{r }.
\end{equation}
The modifications of the diagram are similar to them in the case of $\mathtt{B_3}$, so the Hasse diagram for a Lie algebra of type $\mathtt{B_n}$ has the following shape
\begin{center}
\begin{tikzpicture}

  \node (1) at (-5.5,0) {$\beta_1$};
  \node (2) at (-4,0) {$\beta_3$};
	\node (11) at (-3,0) {$\beta_4$};
  \node (3) at (-2,0) {$. . .$};
  \node (4) at (-1,0) {$\beta_{r}$};
  \node (5) at (0.65,0) {$\beta_{r+1}$};
	\node (12) at (2.05,0) {$\beta_{r+2}$};
  \node (6) at (3.25,0) {$. . .$};
  \node (7) at (4.5,0) {$\beta_{N-2}$};
  \node (8) at (5.4,1) {$\beta_{2}$};
  \node (9) at (5.4,-1) {$\beta_{N-1}$};
  \node (10) at (6.3,0) {$\beta_{N}.$};
%	\node (6) at (-0.5,4) {};
%	\node (7) at (0.5,4) {};
%	\node (8) at (3,5.5) {$\beta_1$};
%  \node (9) at (3,4.5) {$\beta_3$};
%%  \node (10) at (4,4.5) {$\beta_2$};
%  \node (11) at (2,3.5) {$\beta_2$};
%  \node (12) at (4,3.5) {$\beta_4$};
%	\node (13) at (3,2.5) {$\beta_5$};
%%  \node (3) at (0,4) {$\beta_3$};
%%  \node (a) at (-2,2) {$(0,1,1)$};
%%  \node (b) at (0,2) {$(1,0,1)$};
%%  \node (c) at (2,2) {$(1,1,0)$};
%%  \node (d) at (-2,0) {$(0,0,1)$};
%%  \node (e) at (0,0) {$(0,1,0)$};
%%  \node (f) at (2,0) {$(1,0,0)$};
%%  \node (min) at (0,-2) {$(0,0,0)$};
%   %\draw (1) -- (2);
    \draw[->] (1) edge node[above] {\tiny{0110...0}} (2);
    \draw[->] (2) edge node[above] {\tiny{4}}(11);
		\draw[->] (11) edge node[above] {\tiny{5}}(3);
    \draw[->] (3) edge node[above] {\tiny{n}} (4);
    \draw[->] (4) edge node[above] {\tiny{001...12}} (5);
    \draw[->] (5) edge node[above] {\tiny{n-1}}(12);
		\draw[->] (12) edge node[above] {\tiny{n-2}} (6);
    \draw[->] (6) edge node[above] {\tiny{4}} (7);
    \draw[->] (7) edge node[left] {\tiny{2}} (8);
    \draw[->] (7) edge node[left] {\tiny{012...2}} (9);
    \draw[->] (8) edge node[right, pos=0.25] {\tiny{012...2}} (10);
    \draw[->] (9) edge node[right, pos=0.25] {\tiny{2}} (10);
%   \draw[->] (3) edge node[right] {\tiny{3}} (4);
%
%   \draw[->] (6) --+ (7) ;
%
%
%   \draw[->] (8) edge node[right] {\tiny{011}} (9);
%   \draw[->] (9) edge node[right, pos=0.25] {\tiny{012}} (12);
%   \draw[->] (12) edge node[right, pos=0.75] {\tiny{2}} (13);
%   \draw[->] (9) edge node[right, pos=0.75] {\tiny{2}} (11);
%   \draw[->] (11) edge node[right, pos=0.25] {\tiny{012}} (13);
  % \draw (0,4.5) node[anchor=west]  {\tiny{4}};
%  \draw (min) -- (d) -- (a)  -- (b) -- (f)
%  (e) -- (min) -- (f) -- (c) --  (d) -- (b);
%  \draw[preaction={draw=white, -,line width=6pt}] (a) -- (e) -- (c);
\end{tikzpicture}
\end{center}
Associated to the diagrams we get the following polytope for $m \in \Z_{\geq 0}$:
\begin{equation}\label{polytopeBN}
P(m\omega_1)=\left\{ \textbf{x} \in \R_{\geq 0}^{N}\mid
\begin{aligned}
x_{1} + x_{2} + \dots + x_{{N-2}} + x_{N} \leq m \\
x_{1} + x_{3} + \dots + x_{{N-1}} + x_{N} \leq m \\
\end{aligned}
\right\}.
\end{equation}
By Section $\ref{spanning}$, Corollary $\ref{Spanvmw}$ the elements
$$\vwe,f_{1}\vwe,f_{2}\vwe, \dots ,f_{N}\vwe,f_{2}f_{{N-1}}\vwe$$
span $\Vw$ and with \cite[p. 276]{Car05} we have $\dim V(\w_1) = 2n+1$. 
%From these observations we get immediately:
\begin{prop}
The vectors $f^{\bs}v_{\w_1}, \bs \in S(\w_1)$ are a FFL basis of $V(\w_1).\,\hfill \Box$
\end{prop}
\begin{proof}
The previous observations imply that $\{f^{\bs}v_{\w_1}, \bs \in S(\w_1)\}$ is a basis of $V(\omega_1)$. So it remains to show that $P(\w_1)$ is a normal polytope.\\
Because we changed the Hasse diagram we have to change the order of the roots to apply Section $\ref{helpful}$. One possible new order is given by:
\begin{center} 
$\beta_1 \prec \beta_3 \prec \beta_{4} \prec \dots \prec \beta_{N-2} \prec \beta_2 \prec \beta_{N-1} \prec \beta_N$.
\end{center}
Using this order we see immediately that $P(\omega_1)$ is a normal polytope.
\end{proof}
%\begin{rem}\label{orderchange}
%To apply Section \ref{helpful}, in particular to ensure that the paths in the Hasse diagram are Dyck paths in the sense of Definition \ref{helpfulldef}, we have to choose a different order than the order in $\eqref{spanningorderBN}$. One possible order is given by:
%\begin{center} 
%$\beta_1 \prec \beta_3 \prec \beta_{4} \prec \dots \prec \beta_{N-2} \prec \beta_2 \prec \beta_{N-1} \prec \beta_N$.
%\end{center}
%Using this order we see immediately that $P(\w_1)$ is a normal polytope.
%\end{rem}
%We saw that the order in Section $\ref{helpful}$ is independent from the order in Section $\ref{test}$. So we can choose a new order which is independent from the one chosen above. Hence with the order $\beta_1 \prec \beta_3 \prec \beta_4 \prec \beta_{N-2} \prec \beta_{N-1} \prec \beta_{2} \prec \beta_N$ the rewritten diagram fulfills the assumptions of $\eqref{helpfullprop}$. As a consequence we get
%\begin{prop}
%For all $m \in \Z_{\geq 1}$ we have
%\begin{equation*}
%S(m\omega) = \smmew + \sw.
%\end{equation*}
%\end{prop}
\vspace{2mm}

Now we consider the case $\omega=\omega_n$. In the following it will be again convenient to describe the roots and fundamental weights of $\mathtt{B_n}$ in terms of an orthogonal basis:
\begin{equation}\label{relrootsBn}
\rrn=\{\varepsilon_{i,j}=\varepsilon_{i}+\varepsilon_{j} \mid  1\leq i< j\leq n\}\cup \{\varepsilon_{k} \mid 1\leq k\leq n\}.
\end{equation}
The total order on $\rrn$ is obtained by considering the Hasse diagram. We begin with $\beta_1=\theta$ on the top and then labeling from left to right with increasing label on each level of the Hasse diagram, which correspond to the height of the roots in $\rrn$. For a concrete example see Figure \ref{hasseB4B5} in the Appendix. The corresponding polytope is defined as usual, see Table \ref{polyB4} for an example. The elements of $\rrn$ correspond to $\varepsilon_{i,j}= \sum_{r=i}^{j-1}\al_r+2\sum_{r=j}^{n}\al_r$ and $\varepsilon_{k}= \sum_{r=k}^n \al_r$. The highest weight of $V(\omega_n)$ has the description $\omega_n=\frac{1}{2}\sum_{r=1}^n \varepsilon_r$. Further the lowest weight is $-\omega_n = -\frac{1}{2}\sum_{r=1}^n \varepsilon_r.$ With this observation, the fact that $\w_n$ is minuscule and $\eqref{relrootsBn}$ we see that\\
%Moreover we know for the highest weight $\mathtt{B_n}$-module $V(\omega_n)$ the following basis (see Carter (xxx)):
\begin{equation}\label{basisBn}
\B_{V(\omega_n)} = \left\{ f_\alpha v_{\w_n} \mid 
\alpha = \frac{1}{2} \sum_{r=1}^{n}l_r\varepsilon_r , l_r \in \{-1,1\},\ \forall 1\leq r\leq n\right\}\subset V(\omega_n)
\end{equation}
is a basis.
We note that $|\B_{V(\omega_n)}|=2^n=\dim V(\omega_n)$.
% and that the highest weight of $V(\omega_n)$ has the description $\omega_n=\frac{1}{2}\sum_{r=1}^n \varepsilon_r$.
\begin{rem}\label{remBn}
For an arbitrary element $\bps\in \dswn^{\mathtt{B_n}}$ we have at most one root of the form $\varepsilon_k\in \bps$, because if there are $\varepsilon_{k_1},\varepsilon_{k_2}\in\bps$ (wlog $k_1<k_2$) we have: $\varepsilon_{k_1}-\varepsilon_{k_2}= \sum_{r=k_1}^{k_2-1}\al_r$. So with Remark \ref{defdyck} we know that there is a Dyck path $\bp\in D_{\omega_n}$ with $\varepsilon_{k_1},\varepsilon_{k_2}\in\bp$. This observation implies that the elements $\bps\in\dswn^{\mathtt{B_n}}$ have two possible forms:
\begin{equation}\label{b1b2}
(B_1)\ \bps=\{\varepsilon_{k},\varepsilon_{i_2,j_2},\dots,\varepsilon_{i_r,j_r}\}\ \ \textrm{or}\ \ (B_2)\  \bps=\{\varepsilon_{i_1,j_1},\dots,\varepsilon_{i_t,j_t}\}.
\end{equation}
\end{rem}
So we can characterize the elements $\bps\in\dswn^{\mathtt{B_n}}$ as follows.
\begin{prop}\label{chprobBn}
For $\bps\in\prn$ arbitrary we have:
\begin{equation}\label{dspropBn}
\bps\in\dswn^{\mathtt{B_n}} \Leftrightarrow \begin{cases} \bps\ \textrm{is of the form $(B_1)$, with (a) and (b)},\\
\bps\ \textrm{is of the form $(B_2)$, with (b)}.
\end{cases}
\end{equation}
In addition: $\ \bps\in\dswn^{\mathtt{B_n}} \Rightarrow \begin{cases} s\leq \lceil\frac{n}{2}\rceil,\ \bps\ \textrm{is of the form $(B_1)$},\\
s\leq \lfloor\frac{n}{2}\rfloor,\ \bps\ \textrm{is of the form $(B_2)$},
\end{cases}$\\ with $s=|\bps|$. The properties (a) and (b) are defined by
\begin{itemize}
\item[\textit{(a)}] $\forall \, 1\leq l\leq s:\ k<i_l<j_l,$
\item[\textit{(b)}] $\forall \, \ail,\aim\in\bps, i_l\leq i_m: i_l<i_m<j_m<j_l.$
\end{itemize}
\end{prop}
\begin{proof}
First we prove $\eqref{dspropBn}$: ``$\Leftarrow$'': Let $\bp=\{\varepsilon_{k},\varepsilon_{i_2,j_2},\dots,\varepsilon_{i_s,j_s}\}$ be an element of form $(B_1)$ with the properties $(a)$ and $(b)$. Assume there are two roots $x,y\in \bp$ such that there exists a Dyck path $\bq\in D_{\omega_n}$ containing them.
\vspace{2mm}

\textbf{1. Case:} $x=\varepsilon_{k}$ and $y=\varepsilon_{i_m,j_m}$, for $1\leq m\leq s$. Then we have
\begin{equation*}
\varepsilon_{i_m,j_m}-\varepsilon_{k}= \sum_{r=i_m}^{j_m-1}\alpha_r+2\sum_{r=j_m}^{n}\alpha_r-\sum_{r=k}^{n}\alpha_r=-\sum_{r=k}^{i_m-1}\al_r+\sum_{r=j_m}^{n }\al_r.
\end{equation*}
Hence there is no Dyck path $\bq\in D_{\omega_n}$ such that $x$ and $y$ are contained in $\bq$. This is a contradiction to the assumption.
\vspace{2mm}

\textbf{2. Case:} $x=\varepsilon_{i_m,j_m}$ and $y=\varepsilon_{i_l,j_l}$, wlog $i_l<i_m$. Then we have
\begin{equation*}
\varepsilon_{i_l,j_l}-\varepsilon_{i_m,j_m}= \sum_{r=i_l}^{j_l-1}\alpha_r+2\sum_{r=j_l}^{n}\alpha_r-\sum_{r=i_m}^{j_m-1}\alpha_r-2\sum_{r=j_m}^{n}\alpha_r=\sum_{r=i_l}^{i_m-1}\alpha_r-\sum_{r=j_m}^{j_l-1}\alpha_r.
\end{equation*}
This is a contradiction to our assumption and hence: $\bp\in\dswn^{\mathtt{B_n}}$.
\vspace{2mm}

Let $\bp$ be of form $(B_2)$ with property $(b)$, and assume there are two roots $x,y\in \bp$ such that there exists a Dyck path $\bq\in D_{\omega_n}$ containing them. Like in the second case of our previous consideration the assumption is false and therefore: $\bp\in\dswn^{\mathtt{B_n}}$.
\vspace{2mm}

``$\Rightarrow$'': Let $\bp\in\dswn^{\mathtt{B_n}}$. Then we know from Remark \ref{remBn} that $\bp$ is of the form $(B_1)$ or $(B_2)$. Let $\bp=\{\varepsilon_{k},\varepsilon_{i_1,j_1},\dots,\varepsilon_{i_s,j_s} \}$ be of form $(B_1)$, with $i_l<j_l$ for all $1\leq l\leq s$.
\vspace{2mm}

\textbf{1. Step:} Assume $\exists\ 1\leq m\leq s:\ k>i_m$. Then we have:
\begin{equation*}
\varepsilon_{i_m,j_m}-\varepsilon_{k}= \sum_{r=i_m}^{j_m-1}\alpha_r+2\sum_{r=j_m}^{n}\alpha_r-\sum_{r=k}^n \al_r=\sum_{r=i_m}^{k-1}\al_r + \sum_{r=j_m}^{n}\al_r.
\end{equation*}
So by Remark \ref{defdyck} this contradicts $\bp\in\dswn^{\mathtt{B_n}}$. Hence: $k<i_m$ for all $1\leq m\leq s$.
\vspace{2mm}

Let $\varepsilon_{i_l,j_l},\varepsilon_{i_m,j_m}\in\bp$ be two roots with $\varepsilon_{i_l,j_l}\neq\varepsilon_{i_m,j_m}$. We assume wlog $i_l\leq i_m$.
\vspace{2mm}

\textbf{2. Step:} Assume $i_l=i_m=:y$. Set $x=\min\{j_l,j_m\}$ and $\bar{x}=\max\{j_l,j_m\}$:
\begin{equation*}
\varepsilon_{y,x}-\varepsilon_{y,\bar{x}}= \sum_{r=y}^{x-1}\alpha_r+2\sum_{r=x}^{n}\alpha_r-\sum_{r=y}^{\bar{x}-1} \al_r -2\sum_{r=\bar{x}}^{n}\al_r=\sum_{r=x}^{\bar{x}}\al_r.
\end{equation*}
Again by Remark \ref{defdyck} this contradicts $\bp\in\dswn^{\mathtt{B_n}}$ and we have: $i_l<i_m$.
\vspace{2mm}

\textbf{3. Step:} Let $i_l<i_m$ and assume $j_l=j_m=:x$, we consider:
\begin{equation*}
\varepsilon_{i_l,x}-\varepsilon_{i_m,x}= \sum_{r=i_l}^{x}\alpha_r+2\sum_{r=x}^{n}\alpha_r-\sum_{r=i_m}^{x} \al_r -2\sum_{r=x}^{n}\al_r=\sum_{r=i_l}^{i_m-1}\al_r.
\end{equation*}
This contradicts $\bp\in\dswn^{\mathtt{B_n}}$ by Remark \ref{defdyck}, so: $j_l\neq j_m$.
\vspace{2mm}

\textbf{4. Step:} $(i_l<i_m<j_m)\wedge (i_l<j_l)$. So there are three possible cases:
\begin{center}
(a) $i_l<j_l<i_m<j_m$, (b) $i_l<i_m<j_l<j_m$ and (c) $i_l<i_m<j_m<j_l$.
\end{center}
Let us assume $\varepsilon_{i_l,j_l}$ and $\varepsilon_{i_m,j_m}$ have the property of case (a):
\begin{equation*}
\varepsilon_{i_l,j_l}-\varepsilon_{i_m,j_m}= \sum_{r=i_l}^{j_l-1}\alpha_r+2\sum_{r=j_l}^{n}\alpha_r-\sum_{r=i_m}^{j_m-1} \al_r -2\sum_{r=j_m}^{n}\al_r=\sum_{r=i_l}^{j_m-1}\al_r+2\sum_{r=j_l}^{i_m-1}\al_r+\sum_{r=i_m}^{j_m-1}\al_r.
\end{equation*}
This contradicts $\bp\in\dswn^{\mathtt{B_n}}$ by Remark \ref{defdyck}. We assume now that $\varepsilon_{i_l,j_l}$ and $\varepsilon_{i_m,j_m}$ have the property of case (b):
\begin{equation*}
\varepsilon_{i_l,j_l}-\varepsilon_{i_m,j_m}= \sum_{r=i_l}^{j_l-1}\alpha_r+2\sum_{r=j_l}^{n}\alpha_r-\sum_{r=i_m}^{j_m-1} \al_r -2\sum_{r=j_m}^{n}\al_r=\sum_{r=i_l}^{i_m-1}\al_r+\sum_{r=j_l}^{j_m-1}\al_r.
\end{equation*}
Again by Remark \ref{defdyck} this contradicts $\bp\in\dswn^{\mathtt{B_n}}$. Finally we conclude that two roots $\varepsilon_{i_l,j_l},\varepsilon_{i_m,j_m}\in\bp$, with $i_l\leq i_l$, satisfy (c): $i_l<i_m<j_m<j_l$. To prove this statement for a $\bp\in\dswn^{\mathtt{B_n}}$ of form $(B_2)$ we only have to restrict our consideration to the second, third and fourth step.
\vspace{2mm}

It remains to show that the cardinality $s$ of $\bp$ is bounded by $\lceil \frac{n}{2}\rceil$ respectively $\lfloor \frac{n}{2}\rfloor$. Again we consider the two possible cases:
\vspace{2mm}

\textbf{1. Case:} $\bp=\{\varepsilon_{k},\varepsilon_{i_2,j_2},\dots,\varepsilon_{i_s,j_s}\}$ is of the form $(B_1)$ and we assume $|\bp|=s>\lceil \frac{n}{2}\rceil$. Then we know from our previous consideration that after reordering the roots in $\bp$ we have a strictly increasing chain of integers:
\begin{equation}\label{cp1}
C_{\bp}:\ k<i_2<i_3\cdots<i_s<j_s<j_{s-1}<\cdots<j_3<j_2.
\end{equation}
So there are $2s-1$ different integers, where each of these correspond to a $\varepsilon_i$ for $1\leq i \leq n$. By assumption we know $2s-1\geq 2(\lceil \frac{n}{2}\rceil+1)-1\geq n+1$, but there are only $n$ different elements in $\{\varepsilon_r \mid 1\leq r\leq n\}$. So this is a contradiction and hence: $|\bp|=s\leq \lceil \frac{n}{2}\rceil$.
\vspace{2mm}

\textbf{2. Case:} $\bp=\{\varepsilon_{i_1,j_1},\dots,\varepsilon_{i_s,j_s}\}$ is of the form $(B_2)$ and we assume $|\bp|=s>\lfloor \frac{n}{2}\rfloor$. As in the first case we have a strictly increasing chain of integers:
\begin{equation}\label{cpBn2}
C_{\bp}:\ i_1<i_2\cdots<i_s<j_s<j_{s-1}<\cdots<j_2<j_1.
\end{equation}
So we have $2s$ different integers corresponding to at most $n$ different elements in $\{\varepsilon_r \mid 1\leq r\leq n\}$, but by assumption we have $2s\geq 2(\lfloor \frac{n}{2}\rfloor+1)\geq n+1$. Again we have a contradiction and therefore: $|\bp|=s\leq \lfloor \frac{n}{2}\rfloor$.
\end{proof}
Because of Corollary \ref{Spanvmw} we know that the elements $\{f^{\bs}v_{\omega_n} \mid \bs \in S(\omega_n)\}$ span $V(\omega_n)$ and by Proposition $\ref{bijSWDL}$ there is a bijection between $S(\omega_n)$ and $\dswn^{\mathtt{B_n}}$.
% span the highest weight $\mathtt{A_n}$-module $\Vw$.
We want to show that these elements are linear independent. To achieve that we will show that $|\dswn^{\mathtt{B_n}}|=\dim V(\w_n) $. To be more explicit:
\begin{prop}\label{proBn}
$|\dswn^{\mathtt{B_n}}|=\dim V(\omega_n)=2^n$.
\end{prop}
\begin{proof}
We know from (\ref{dspropBn}) that for an arbitrary element $\bps\in\dswn^{\mathtt{B_n}}$ the number of roots $s$ in $\bps$ is bounded by $\lceil \frac{n}{2}\rceil$ respective by $\lfloor \frac{n}{2}\rfloor$. So the number of integers occurring in $C_{\bps}$ (see (\ref{cp1}) and (\ref{cpBn2})) is also bounded:
\begin{equation}
|C_{\bps}|=\begin{cases} 2s-1\leq 2\lceil \frac{n}{2}\rceil-1\leq n,&\ \textrm{$\bps$ is of the form $(B_1)$},\\
2s\leq 2\lfloor \frac{n}{2}\rfloor\leq n,&\ \textrm{$\bps$ is of the form $(B_2)$}.\end{cases}
\end{equation}
In order to simplify our notation, we define $l:=|C_{\bps}|$, so we have for an arbitrary $\bps\in\dswn^{\mathtt{B_n}}$: $0\leq l\leq n$. Further we define the subsets $\dswn^{\mathtt{B_n}}(l)\subset \dswn^{\mathtt{B_n}}$:
\begin{equation}\label{defdswnkBn}
\dswn^{\mathtt{B_n}}(l):=\{ \bps\in\dswn^{\mathtt{B_n}}\mid|C_{\bps}|=l\},\ \forall\ 0\leq l\leq n.
\end{equation}
So the elements in $\dswn^{\mathtt{B_n}}(l)$ are parametrized by $l$ totally ordered integers $u_i$ in $\{r\ |\ 1\leq r\leq n\}$, $\forall\ 1\leq i\leq l$. Hence we conclude: $|\dswn^{\mathtt{B_n}}(l)|\leq\binom{n}{l}$, $\forall\ 1\leq l\leq n$ and so
\begin{equation}\label{dimabBn}
|\dswn^{\mathtt{B_n}}|=|\bigcup_{l=1}^n \dswn^{\mathtt{B_n}}(l)|=\sum_{l=0}^n |\dswn^{\mathtt{B_n}}(l)|\leq \sum_{l=0}^n \binom{n}{l}=2^n.
\end{equation}
We also know from Corollary \ref{Spanvmw} that we have $|\dswn^{\mathtt{B_n}}|\geq \dim V(\omega_n)=\binom{n}{l}=2^n$. Finally we conclude: $|\dswn^{\mathtt{B_n}}|=2^n$.\end{proof}
\begin{exa} The polytope $P(m\omega_3)$ in the case $\Lg = \mf{so}_7$ has the following shape.
\begin{equation*}
P(m\omega_3)=\left\{\bx\in\R^{6}_{\geq 0}\mid \begin{aligned}
&x_1+x_2+x_3+x_5+x_6\leq m\\
&x_1+x_2+x_4+x_5+x_6\leq m
\end{aligned}
\right\}.
\end{equation*}
\end{exa}
Proposition \ref{proBn} implies immediately:
\begin{prop}\label{basisBnwn}
The vectors $f^{\bs}v_{\w_n}, \bs \in S(\w_n)$ are a FFL basis of $V(\w_n).\,\hfill \Box$
\end{prop}
%\begin{proof}
%The Prop. \ref{proBn} implies that the elements of $\dswn^{\mathtt{B_n}}$ parametrize a PBW-basis for $V(\omega_n)^{a}$. The observation that $V(\w)^{a}\cong V(\w)$ as vector spaces proves the statement.
%\end{proof}

\subsection{Type \texorpdfstring{$\mathtt{C_n}$}{Cn}}
Let $\Lg$ be a simple Lie algebra of type $\mathtt{C_n}$ for $n \geq 3$ with the associated Dynkin diagram
\begin{center}
\begin{tikzpicture}[scale=.4]
    \draw (-1,0) node[anchor=east]  {$\mathtt{C_n}$};
    \foreach \x in {0,...,4}
    \draw[thick,xshift=\x cm] (\x cm,0) circle (1 mm);
    \foreach \y in {0,...,0,2}
    \draw[thick,xshift=\y cm] (\y cm,0) ++(.3 cm, 0) -- +(14 mm,0);

		\draw[thick] (6.3,0.1) -- +(1.4 cm, 0);
		\draw[thick] (6.3,-0.1) -- +(1.4 cm, 0);

		\draw (7.8,0) node[anchor=east] {\small{{$\bf\boldsymbol{<}$}}};

    \draw (0,0) node[anchor=north]  {\tiny{1}};
    \draw (2,0) node[anchor=north]  {\tiny{2}};
    \draw (4,0) node[anchor=north]  {\tiny{n-2}};
    \draw (6,0) node[anchor=north]  {\tiny{n-1}};
    \draw (8,-0.1) node[anchor=north]  {\tiny{n}};
    \draw[thick,dashed] (2.3,0)--(3.7,0);
    %\draw (4,3) node[anchor=north]  {\tiny{2}};
  \end{tikzpicture}
\end{center}For all fundamental weights $\omega_k$ we have $\langle \omega_k,\theta^{\vee}\rangle = 1$, where $\theta = (2,2,\dots,2,1)$ is the highest root and $\theta^{\vee} =(1,1,\dots,1)$ the corresponding coroot. But only for $\omega_1$ the associated Hasse diagram $H(\Ln_{\omega_1}^-)_{\Lg}$ has no $i$-chains.  In fact for $1 \leq k \leq n$, $H(\Ln_{\omega_k}^-)_{\Lg}$ has $k-1$ different $i$-chains, with $1 \leq i \leq k-1$.
% In these cases our approach does not succeed.
%Like in the case of $\mathtt{B_n}, \omega_1$ we have chains in some other cases, where we are able to rewrite the diagram.
The following example explains, why we are not able to rewrite the diagram in these cases, with our approach.\\
For all $\omega_k$ with $k \neq 1$ we have the following $1$-chain.
\begin{center}
\begin{tikzpicture}
\node (1) at (0,0) {$\beta_1$};
  \node (2) at (1,0) {$\beta_2$};
	\node (3) at (2,0) {$\beta_3.$};
	\draw[->] (1) edge node[above] {\tiny{1}} (2);
	\draw[->] (2) edge node[above] {\tiny{1}} (3);
\end{tikzpicture}
\end{center}
Here $\beta_1 = 2\al_1 + \dots +2\al_{n-1} + \al_n$ is the highest root, $\beta_2 = \al_1 + 2\al_2 + \dots + 2\al_{n-1} + \al_n$ and $\beta_3 = 2\al_2 +  \dots + 2\al_{n-1} + \al_n$. Note that $\beta_1 - \beta_3 = 2 \al_1$, which is not a root. Further, because $\beta_1$ is the highest root, there are no roots $\gamma \in \rrp, \nu \in \rrk$ with $\partial_{\gamma}f_{\nu} = f_{3}$, except for $\nu = \beta_2.$ Hence it is more involved to rewrite the diagram into a diagram without $k$-chains such that there is a path connecting $\beta_1$ and $\beta_3$.\\
Nevertheless, in \cite{FFoL11b} similar statements to Theorem A and Theorem B were proven for arbitrary dominant integral weights.\\[2mm]
Now we consider $\omega = \omega_1$. Then we have $2n - 1 = N$ and $\Delta_+^{\omega}$ is given by
\begin{center} $\begin{array}{|l|l|l|l|l|}
\hline
\beta_1\,\,\,\,\,\,=(2,2,\dots,2,1) & \,\,\,\beta_2 \,\,\, = (1,2,\dots,2,1) & \,\,\, \dots \,\,\, & \beta_{n}\,=(1,1,\dots,1,1) \\
\hline
\beta_{n+1}=(1,1,\dots,1,0) & \beta_{n+2}=(1,\dots,1,0,0) & \,\,\, \dots \,\,\,&  \beta_N=(1,0,\dots,0,0) \\
\hline	
\end{array}$ \end{center}
The diagram $H(\Ln_{\omega}^-)_{\Lg}$ has the following form.
\begin{center}
\begin{tikzpicture}
\node (1) at (0,0) {$\beta_1$};
  \node (2) at (1,0) {$\beta_2$};
	\node (3) at (2,0) {$\beta_3$};
	\node (4) at (3,0) {$...$};
	\node (5) at (4.25,0) {$\beta_{n-1}$};
	\node (6) at (5.5,0) {$\beta_{n}$};
	\node (7) at (6.75,0) {$\beta_{n+1}$};
	\node (8) at (8,0) {$...$};
	\node (9) at (9,0) {$\beta_N.$};
	\draw[->] (1) edge node[above] {\tiny{1}} (2);
	\draw[->] (2) edge node[above] {\tiny{2}} (3);
	\draw[->] (3) edge node[above] {\tiny{3}} (4);
	\draw[->] (4) edge node[above] {\tiny{n-2}} (5);
	\draw[->] (5) edge node[above] {\tiny{n-1}} (6);
	\draw[->] (6) edge node[above] {\tiny{n}} (7);
	\draw[->] (7) edge node[above] {\tiny{n-1}} (8);
	\draw[->] (8) edge node[above] {\tiny{2}} (9);
\end{tikzpicture}
\end{center}
There are no $k$-chains and the associated polytope is given by
\begin{equation*}
P(m\omega)= \{ \mathbf{x} \in \R_{\geq 0}^{N} \mid x_1 + x_2 + \dots + x_N \leq m\}.
\end{equation*}
By Corollary $\ref{Spanvmw}$ the elements $\vw,f_{1}\vw,f_{2}\vw, \dots, f_{N}\vw$ span $\Vw$ and with \cite[p295]{Car05} we know $\dim \Vw = 2n$. From these observations we get immediately:
% for $S(m\omega) = \pmw \cap \Z_{\geq 0}^N$
\begin{prop}
The set $\B_{\omega} = \{f^{\bs}v_{\w} \mid \bs \in S(\w)\}$ is a FFL basis of $V(\w)$.\hfill $\Box$
\end{prop}
\subsection{Type \texorpdfstring{$\mathtt{D_n}$}{Dn}}
Let $\Lg$ be a simple Lie algebra of type $\mathtt{D_n}$ with associated Dynkin diagram
\begin{center}
\begin{tikzpicture}[scale=.4]
    \draw (-1,0) node[anchor=east]  {$\mathtt{D_n}$};
    \foreach \x in {0,...,3}
    \draw[thick,xshift=\x cm] (\x cm,0) circle (1 mm);
    \foreach \y in {0,...,0,2}
    \draw[thick,xshift=\y cm] (\y cm,0) ++(.3 cm, 0) -- +(14 mm,0);

		\draw[thick] (6.3,0.1) -- +(1.4 cm, 0.5);
		\draw[thick] (6.3,-0.1) -- +(1.4 cm, -0.5);
    \draw[thick] (8 ,0.7) circle (1mm);
		\draw[thick] (8 ,-0.7) circle (1mm);
		%\draw (7.8,0) node[anchor=east] {$\small{{\bf\boldsymbol{<}}}$};

    \draw (0,0) node[anchor=north]  {\tiny{1}};
    \draw (2,0) node[anchor=north]  {\tiny{2}};
    \draw (4,0) node[anchor=north]  {\tiny{n-3}};
    \draw (6,0) node[anchor=north]  {\tiny{n-2}};
		\draw (8,0.7) node[anchor=north]  {\tiny{n-1}};
    \draw (8,-0.7) node[anchor=north]  {\tiny{n}};
    \draw[thick,dashed] (2.3,0)--(3.7,0);
    %\draw (4,3) node[anchor=north]  {\tiny{2}};
  \end{tikzpicture}
\end{center}
The highest root in type $\mathtt{D_n}$ is of the form $\theta=\alpha_1 +2\sum_{i=2}^{n-2}\alpha_i +\al_{n-1}+\al_{n}$. Since $\Lg$ is simply-laced we have $\theta^{\vee}=\alpha_1^{\vee}+2\sum_{i=2}^{n-2}\alpha_i^{\vee}+ \alpha_{n-1}^{\vee}+\alpha_n^{\vee}$.
Hence $\langle \omega,\theta^{\vee}\rangle=1\Leftrightarrow \omega\in \{\omega_1,\omega_{n-1},\omega_n\}$.
%For our considerations it will be more convinient to describe the roots and fundamental weights of $\mathtt{D_n}$ in terms of an orthogonal basis $\{\varepsilon_i|\ 1\leq i\leq n\}$.
\vspace{2mm}

First we consider the case $\omega = \omega_1$. Then we have $2n-2 = N$ and $\Delta_+^{\w_1}$ has the following form:
\begin{center}$
\begin{array}{|l|l|l|l|l|}
\hline
\,\,\,\,\beta_1=\!(1,2,2\dots,2,1,1)\!\! &\!\! \beta_2  =\! (1,1,2,\dots,2,1,1)\!\! &\!\!  \dots\! \! &\!\! \beta_{n-2}=\!(1,1,1\dots,1,1,1)\!\! \\
\hline
\!\!\beta_{n-1}=\!(1,1,1\dots,1,0,1)\!\! &\!\! \beta_{n}=\!(1,1,1,\dots,1,1,0)\!\! &\!\!  \dots \!\!& \,\,\, \beta_{N}=\!(1,0,0\dots,0,0,0)\!\! \\
\hline	
\end{array}$
\end{center} %\subsection{Type $\mathtt{E_6}$ and $\mathtt{E_7}$}
The Hasse diagram has no $k$-chain. In addition in $\overline{D}_{\omega_1}$ there are only co-chains of cardinality at most 1, except for one with cardinality 2.
\begin{center}
\begin{tikzpicture}
\node (1) at (0,0) {$\beta_1$};
  \node (2) at (1,0) {$\beta_2$};
	\node (3) at (2,0) {$\beta_3$};
	\node (4) at (3,0) {$...$};
	\node (5) at (4.25,0) {$\beta_{n-2}$};
	\node (6) at (5.25,1) {$\beta_{n-1}$};
	\node (7) at (5.25,-1) {$\beta_{n}$};
	\node (10) at (6.25,0) {$\beta_{n+1}$};
	\node (11) at (7.75,0) {$\beta_{n+2}$};
	\node (8) at (9,0) {$...$};
	\node (9) at (10,0) {$\beta_N.$};
	\draw[->] (1) edge node[above] {\tiny{2}} (2);
	\draw[->] (2) edge node[above] {\tiny{3}} (3);
	\draw[->] (3) edge node[above] {\tiny{4}} (4);
	\draw[->] (4) edge node[above] {\tiny{n-2}} (5);
	\draw[->] (5) edge node[left] {\tiny{n-1}} (6);
	\draw[->] (5) edge node[left] {\tiny{n}} (7);
	\draw[->] (6) edge node[left] {\tiny{n}} (10);
	\draw[->] (7) edge node[left] {\tiny{n-1}} (10);
	\draw[->] (10) edge node[above] {\tiny{n-2}} (11);
	\draw[->] (11) edge node[above] {\tiny{n-3}} (8);
%	\draw[->] (7) edge node[above] {\tiny{n-1}} (8);
	\draw[->] (8) edge node[above] {\tiny{2}} (9);
\end{tikzpicture}
\end{center}
Associated to this diagram we get the following polytope for $m \in \Z_{\geq 0}$:
\begin{equation*}
\pmw=\left\{ \textbf{x} \in \R_{\geq 0}^{N}\mid
\begin{aligned}
x_{1} +  \dots + x_{{n-2}} &+ x_{{n-1}}\!\!\!\!\! &+\, x_{{n+1}} + \dots + x_{N} \leq m \\
x_{1} + \dots + x_{{n-2}} &+ x_{{n}}\!\!\!\!\!\! &+\, x_{{n+1}}+  \dots +  x_{N} \leq m \\
\end{aligned}
\right\}.
\end{equation*}
By Corollary $\ref{Spanvmw}$ the elements
$\B_{\omega_1}=\{v_{\w_1},f_{1}v_{\w_1},f_{2}v_{\w_1}, \dots ,f_{N}v_{\w_1},f_{{n-1}}f_{{n}}v_{\w_1}\}$
span $V(\w_1)$ and with \cite[p. 280]{Car05} we have $\dim V(\w_1) = 2n$. From these observations we get immediately.
\begin{prop}
The vectors $f^{\bs}v_{\w_1},\bs \in S(\w_1)$ are a FFL basis of $V(\w_1).\,\hfill \Box$
\end{prop}
\vspace{2mm}

For most of the proofs of the statements in the case $\w=\w_{n-1},\w_{n}$ we will refer to the proofs of the corresponding statements for type $\mathtt{B_n}$. 
%because the ideas and structure of the proofs are the same.
\vspace{2mm}

Now we consider the case $\omega=\omega_{n-1}$. For further considerations it will be convenient to describe the roots and fundamental weights of $\Lg$ in terms of an orthogonal basis $\{\varepsilon_i \mid 1\leq i\leq n\}$. Then $\rrnme$ is given by
\begin{equation}\label{relrootsDnm1}
\{\varepsilon_{i,j}=\varepsilon_{i}+\varepsilon_{j}\mid 1\leq i< j\leq n-1\}\cup \{\varepsilon_{k,\overline{n}}=\varepsilon_{k}-\varepsilon_{n}\mid 1\leq k\leq n-1\}.
\end{equation}
The total order on $\rrnme$ is defined like in the $\mathtt{B_n}$, $\w_n$-case (see Figure \ref{hasseB4B5}).
The elements of $\rrnme$ correspond to $\varepsilon_{i,j}= \sum_{r=i}^{j-1}\al_r+2\sum_{r=j}^{n-2}\al_r+\al_{n-1}+\al_n$ and $\varepsilon_{k,\overline{n}}= \sum_{r=k,}^{n-1} \al_r$. The highest weight of $V(\omega_{n-1})$ has the description $\omega_{n-1}=\frac{1}{2}\left(\sum_{r=1}^{n-1} \varepsilon_r -\varepsilon_n \right)$. Further the lowest weight is $-\omega_{n-1}=-\frac{1}{2}\left(\sum_{r=1}^{n-1} \varepsilon_r -\varepsilon_n \right)$. With this observation, the fact that $\w_{n-1}$ is minuscule and $\eqref{relrootsDnm1}$ we see that 
\begin{equation*}\label{basisDnme}
\B_{V(\omega_{n-1})}=\left\{f_\alpha v_{\w_{n-1}} \mid \alpha = \frac{1}{2} \sum_{r=1}^{n}l_r\varepsilon_r, l_r = \pm 1,\ \forall 1\leq r\leq n,\ 2\nmid\#\{l_r\mid l_r=-1\}\right\}
\end{equation*}
is a basis of $V(\omega_{n-1})$. We note that $|\B_{V(\omega_{n-1})}|=2^{n-1}=\dim V(\omega_{n-1})$.
\begin{rem}\label{remDnme}
Similar arguments as in Remark \ref{remBn} show that the elements $\bps\in\dswnme^{\mathtt{D_n}}$ have two possible forms:
%For an arbitrary element $\bps\in \dswnme^{\mathtt{D_n}}$ we have at most one root of the form $\varepsilon_{k,\overline{n}}\in \bps$. Because if there are $\varepsilon_{k_1,\overline{n}},\varepsilon_{k_2,\overline{n}}\in\bps$ (wlog $k_1<k_2$) we have: $\varepsilon_{k_1,\overline{n}}-\varepsilon_{k_2,\overline{n}}= \sum_{r=k_1}^{k_2-1}\al_r$, so with Remark \ref{defdyck} we know that there is a Dyck path $\bp\in D_{\omega_{n-1}}$ with $\varepsilon_{k_1,\overline{n}},\varepsilon_{k_2,\overline{n}}\in\bp$.
%This observation implies,
\begin{equation}
(D_1)\ \bps=\{\varepsilon_{k,\overline{n}},\varepsilon_{i_2,j_2},\dots,\varepsilon_{i_r,j_r}\}\ \ \textrm{or}\ \ (D_2)\  \bps=\{\varepsilon_{i_1,j_1},\dots,\varepsilon_{i_t,j_t}\}.
\end{equation}
\end{rem}
We denote with $\textbf{1}_{2\nmid n}:\Z_{\geq 0}\rightarrow \{0,1\}$ (respective $\textbf{1}_{2\mid n}$) the Indicator function for the odd (respective even) integers, which is defined by $\textbf{1}_{2\nmid n}(n)=1$ if $2\nmid n$ (respective $\textbf{1}_{2\mid n}(n)=1$ if $2\mid n$) and $0$ otherwise.
So we can characterize the elements $\bps\in\dswnme^{\mathtt{D_n}}$ as follows
\begin{prop}\label{chproDnme}
For $\bps\in\mathcal{P}(\Delta_{+}^{\omega_{n-1}})$ arbitrary we have:
\begin{equation}\label{dspropDnme}
\bps\in\dswnme^{\mathtt{D_n}} \Leftrightarrow \begin{cases} \bps\ \textrm{is of the form $(D_1)$, with (a) and (b)},\\
\bps\ \textrm{is of the form $(D_2)$, with (b)}.
\end{cases}
\end{equation}
In addition: $\ \bps\in\dswnme^{\mathtt{D_n}} \Rightarrow \begin{cases} s\leq \lceil\frac{n}{2}\rceil-\textbf{1}_{2\nmid n}(n),\ \bps\ \textrm{is of the form $(D_1)$},\\
s\leq \lfloor\frac{n}{2}\rfloor-\textbf{1}_{2\mid n}(n),\ \bps\ \textrm{is of the form $(D_2)$},
\end{cases}$\\
%\bps\in\dswnme^{\mathtt{D_n}} \Leftrightarrow \begin{cases} \bps\ \textrm{is of the form $(D_1)$, with (a),(b), $s\leq \lceil\frac{n}{2}\rceil-\textbf{1}_{2\nmid n}(n)$},\\
%\bps\ \textrm{is of the form $(D_2)$, with (b), } s\leq \lfloor\frac{n}{2}\rfloor-\textbf{1}_{2\mid n}(n).
%\end{cases}
%\end{equation}
with $s=|\bps|$. The properties (a) and (b) are defined by
\begin{itemize}
\item[\textit{(a)}] $\forall \, 1\leq l\leq s:\ k<i_l<j_l,$
\item[\textit{(b)}] $\forall \, \ail,\aim\in\bps, i_l\leq i_m: i_l<i_m<j_m<j_l.$
\end{itemize}
%We denote with $\textbf{1}_{2\nmid n}(n)$ (respective $\textbf{1}_{2\mid n}(n)$) the Indicator function for the odd (respective even) integers.
\end{prop}
\begin{proof}
To prove this statement we adapt the idea of Proposition \ref{chprobBn}. We use exactly the same approach but we consider $\Delta_{+}^{\w_{n-1}}$ of type $\mathtt{D_n}$.\\
To check that that the cardinality $s$ of an arbitrary element $\bps\in \dswnme^{\mathtt{D_n}}$ is bounded, like we claim on the rhs of (\ref{dspropDnme}), we use only fundamental combinatorics, again analogue to the idea of the proof of Proposition \ref{chprobBn}.
\end{proof}
Because of Corollary \ref{Spanvmw} we know that the elements $\{f^{\bs}v_{\omega_{n-1}} \mid \bs \in S({\omega_{n-1}})\}$ span $V({\omega_{n-1}})$ and by Proposition $\ref{bijSWDL}$ there is a bijection between $S({\omega_{n-1}})$ and $\dswnme^{\mathtt{D_n}}$.
% span the highest weight $\mathtt{A_n}$-module $\Vw$.
We want to show that these elements are linear independent. To achieve that we will show that $|\dswnme^{\mathtt{D_n}}|=\dim V({\omega_{n-1}}) $. To be more explicit:
%Because of Theorem \ref{thmspan} we know that the elements of $\dswnme^{\mathtt{D_n}}$ spann the highest weight $\mathtt{D_n}$-module $V(\omega_{n-1})$. But we still have to show that these elements are linear independent. To achieve that we will show:
\begin{prop}\label{proDnme}
$|\dswnme^{\mathtt{D_n}}|=\dim V(\omega_{n-1})=2^{n-1}$.
\end{prop}
\begin{proof}
This is a direct consequence of Lemma \ref{zweihochme} and Proposition \ref{proBn}.
\end{proof}
Proposition \ref{proDnme} implies immediately
\begin{prop}\label{basisDnwnme}
$\B_{\omega_{n-1}} = \{f^{\bs}v_{\w_{n-1}} \mid \bs \in S(\w_{n-1})\}$ is a basis for $V(\w_{n-1}).\hfill \Box$
\end{prop}
\vspace{2mm}

Finally we consider the case $\omega=\omega_n$. For the proofs of the statements in this case we refer to the proofs of the analogous statements in the previous case $\omega=\omega_{n-1}$ and the $\mathtt{B_n}, \omega_n$-case.\\The set of roots $\rrn$, where $\al_n=\varepsilon_{n-1}+\varepsilon_{n}$ is a summand, is given by:
\begin{equation}\label{relrootsDn}
\{\varepsilon_{i,j}=\varepsilon_{i}+\varepsilon_{j} \mid 1\leq i< j\leq n-1\}\cup \{\varepsilon_{k,n}=\varepsilon_{k}+\varepsilon_{n}\mid 1\leq k\leq n-1\}.
\end{equation}
Again the total order on $\rrn$ is defined like in the $\mathtt{B_n}, \w_n$-case (see Figure \ref{hasseB4B5}), where the elements of $\rrn$ correspond to $\varepsilon_{i,j}= \sum_{r=i}^{j-1}\al_r+2\sum_{r=j}^{n-2}\al_r+\al_{n-1}+\al_n$ and $\varepsilon_{k,n}= \sum_{r=k,\ r\neq n-1}^n \al_r$.
The highest weight of $V(\omega_{n})$ has the description $\omega_{n}=\frac{1}{2}\left(\sum_{r=1}^{n} \varepsilon_r\right)$. Further the lowest weight is $-\omega_{n}=-\frac{1}{2}\left(\sum_{r=1}^{n} \varepsilon_r\right)$. As before we see that
\begin{equation}\label{basisDn}
\B_{V(\w_n)}=\left\{f_\alpha v_{\w_n}\mid \alpha = \frac{1}{2} \sum_{r=1}^{n}l_r\varepsilon_r, l_r\{-1,1\},\ \forall 1\leq r\leq n,\ 2\mid \#\{l_r\mid l_r=-1\}\right\}
\end{equation}
is a basis of $V(\omega_n)$. We note that $|\Bwn|=2^{n-1}=\dim V(\omega_n)$.
\begin{rem}\label{remDn}
Similar arguments as in Remark \ref{remBn} show that the elements $\bps\in\dswn^{\mathtt{D_n}}$ have two possible forms:
%For an arbitrary element $\bps\in \dswn^{\mathtt{D_n}}$ we have at most one root of the form $\varepsilon_{k,n}\in \bps$. Because if there are $\varepsilon_{k_1,n},\varepsilon_{k_2,n}\in\bps$ (wlog $k_1<k_2$) we have: $\varepsilon_{k_1}-\varepsilon_{k_2}= \sum_{r=k_1}^{k_2-1}\al_r$, so with Remark \ref{defdyck} we know that there is a Dyck path $\bp\in D_{\omega_n}$ with $\varepsilon_{k_1,n},\varepsilon_{k_2,n}\in\bp$. This observation implies,

\begin{equation}
(D_1^*)\ \bps=\{\varepsilon_{k,n},\varepsilon_{i_2,j_2},\dots,\varepsilon_{i_s,j_s}\}\ \ \textrm{and}\ \ (D_2^*)\  \bps=\{\varepsilon_{i_1,j_1},\dots,\varepsilon_{i_s,j_s}\}.
\end{equation}
\end{rem}
So we can characterize the elements $\bps\in\dswn^{\mathtt{D_n}}$ as follows:
\begin{prop}\label{chproDn}
For $\bps\in\mathcal{P}(\Delta_{+}^{\omega_{n}})$ arbitrary we have:
\begin{equation}\label{dspropDn}
\bps\in\dswn^{\mathtt{D_n}} \Leftrightarrow \begin{cases} \bps\ \textrm{is of the form $(D_1^*)$, with (a) and (b)},\\
\bps\ \textrm{is of the form $(D_2^*)$, with (b)}.
\end{cases}
\end{equation}
In addition: $\ \bps\in\dswn^{\mathtt{D_n}} \Rightarrow \begin{cases} s\leq \lceil\frac{n}{2}\rceil-\textbf{1}_{2\nmid n}(n),\ \bps\ \textrm{is of the form $(D_1^*)$},\\
s\leq \lfloor\frac{n}{2}\rfloor-\textbf{1}_{2\mid n}(n),\ \bps\ \textrm{is of the form $(D_2^*)$},
\end{cases}$\\
with $s=|\bps|$. 
The properties (a) and (b) are defined by
\begin{itemize}
\item[\textit{(a)}] $\forall \, 1\leq l\leq s:\ k<i_l<j_l,$
\item[\textit{(b)}] $\forall \, \ail,\aim\in\bps, i_l\leq i_m: i_l<i_m<j_m<j_l.$
\end{itemize}
%We denote with $\textbf{1}_{2\nmid n}(n)$ (respective $\textbf{1}_{2\mid n}(n)$) the Indicator function for the odd (respective even) integers.
\end{prop}
\begin{proof}
To prove this statement we refer to the proof of Proposition \ref{chproDnme}.
\end{proof}
Because of Corollary \ref{Spanvmw} we know that the elements of $\dswn^{\mathtt{D_n}}$ span the highest weight module $V(\omega_n)$. But we still have to show that these elements are linear independent. To achieve that we will show:
\begin{prop}\label{proDn}
$|\dswn^{\mathtt{D_n}}|=\dim V(\omega_n)=2^{n-1}$.
\end{prop}
\begin{proof}
This is a direct consequence of Lemma \ref{zweihochme} and Proposition \ref{proBn}.\end{proof}
Proposition \ref{proDn} implies immediately
\begin{prop}\label{basisDnwn}
The set $\B_{\omega_n} = \{f^{\bs}v_{\w_n} \mid \bs \in S(\w_n)\}$ is a basis for $V(\w_n)$.\hfill $\Box$
\end{prop}
%\begin{proof}
%The proposition (\ref{proDn}) implies that the elements of $\dswn^{\mathtt{D_n}}$ parametrize a PBW-basis for $V(\omega_n)^{a}$. The observation that $V(\w_n)^{a}\cong V(\w_n)$ proves the statement.
%\end{proof}
The following Lemma gives us a very useful connection between the co-chains of $\Lg$ of type $\mathtt{B_{n-1}}$ and $\mathtt{D_n}$:
\begin{lem}\label{zweihochme}
We have: $|\dswnme^{\mathtt{D_n}}|=|\dswnme^{\mathtt{B_{n-1}}}|$ and $|\dswn^{\mathtt{D_n}}|=|\dswnme^{\mathtt{B_{n-1}}}|$.
\end{lem}
\begin{proof}
We only use basic combinatorics to prove this statement.
\end{proof}

\subsection{Type \texorpdfstring{$\mathtt{E_6}$}{E6}}
Let $\Lg$ be a simple Lie algebra of type $\mathtt{E_6}$ with associated Dynkin diagram
\begin{center}
  \begin{tikzpicture}[scale=.4]
    \draw (-1,0) node[anchor=east]  {$\mathtt{E_6}$};
    \foreach \x in {0,...,4}
    \draw[thick,xshift=\x cm] (\x cm,0) circle (1 mm);
    \foreach \y in {0,...,3}
    \draw[thick,xshift=\y cm] (\y cm,0) ++(.3 cm, 0) -- +(14 mm,0);
    \draw[thick] (4 cm,2 cm) circle (1 mm);
    \draw[thick] (4 cm, 3mm) -- +(0, 1.4 cm);
    \draw (0,0) node[anchor=north]  {\tiny{1}};
    \draw (2,0) node[anchor=north]  {\tiny{3}};
    \draw (4,0) node[anchor=north]  {\tiny{4}};
    \draw (6,0) node[anchor=north]  {\tiny{5}};
    \draw (8,0) node[anchor=north]  {\tiny{6}};
    \draw (4,3) node[anchor=north]  {\tiny{2}};
  \end{tikzpicture}
\end{center}
We have $\langle \w,\theta^{\vee}\rangle=1\Leftrightarrow \omega = \omega_1,\omega_6$ and first we fix $\w$ to be $\omega_6$.
The set is $\Delta_+^{\omega_6}$ given as follows:
\begin{center} $\begin{array}{|l|l|}
\hline
\beta_1=(1,2,2,3,2,1) & \beta_{9} \,\,=(1,1,1,1,1,1)\\
\beta_2=(1,1,2,3,2,1) & \beta_{10}=(0,1,1,1,1,1) \\
\beta_3=(1,1,2,2,2,1) & \beta_{11}=(1,0,1,1,1,1) \\
\beta_4=(1,1,1,2,2,1) & \beta_{12}=(0,0,1,1,1,1) \\
\beta_5=(1,1,2,2,1,1) & \beta_{13}=(0,1,0,1,1,1) \\
\beta_6=(0,1,1,2,2,1) & \beta_{14}=(0,0,0,1,1,1) \\
\beta_7=(1,1,1,2,1,1) & \beta_{15}=(0,0,0,0,1,1) \\
\beta_8=(0,1,1,2,1,1) & \beta_{16}=(0,0,0,0,0,1) \\
\hline	
\end{array}$ \end{center}
The Hasse diagram $H(\Ln_{\omega_6}^-)_{\mathtt{E_6}}$ has no $k$-chains and the maximal cardinality of a co-chain of $H(\Ln_{\omega_6}^-)_{\mathtt{E_6}}$ is two (see Appendix, Figure \ref{hasseE6}). The associated polytope is given for $m \in \Z_{\geq 0}$ by:
\begin{equation*}
P(m\w_6)=\{ \textbf{x} \in \R_{\geq 0}^{16} \mid \sum_{\beta_j \, \in \, \bp} x_{j} \leq m,\ \forall \bp\in D_{\omega_6} \},
\end{equation*}
in particular see Appendix, Table \ref{polyE6} for the non-redundant inequalities.
\begin{prop} The set $\B_{\omega_6} = \{f^{\bs}v_{\w_6} \mid \bs \in S(\w_6)\}$ is a FLL basis of $V(\w_6).$
\end{prop}
\begin{proof}
The co-chains of the Hasse diagram give us immediately:
\begin{eqnarray*}
\B_{\omega_6}=\{\!\!\!\!\!\!\!\!\!\!\!&&v_{\omega_6}, f_{1}v_{\omega_6},f_{2}v_{\omega_6}, \dots, f_{{16}}v_{\omega_6}, f_{{4}}f_{{5}}v_{\omega_6}, f_{{5}}f_{{6}}v_{\omega_6}, f_{{6}}f_{{7}}v_{\omega_6}, f_{{6}}f_{{9}}v_{\omega_6}, \\&&f_{{8}}f_{{9}}v_{\omega_6},f_{{8}}f_{{10}}v_{\omega_6}, f_{{8}}f_{{11}}v_{\omega_6}, f_{{10}}f_{{11}}v_{\omega_6}, f_{{11}}f_{{13}}v_{\omega_6}, f_{{12}}f_{{13}}v_{\omega_6} \}.
\end{eqnarray*}
Note that there are 27 elements in $\B_{\omega_6}$. By Corollary $\ref{Spanvmw}$, we get that $\B_{\omega_6}$ is a spanning set of $V(\omega_6)$. By \cite[p. 303]{Car05} we have $\dim V(\omega_6) = 27$ and therefore the claim holds.
\end{proof} 
It is shown in Figure \ref{hasseE6} that the Hasse diagrams $H(\mf{n}_{\w_1}^-)_{\mathtt{E_6}}$ and $H(\mf{n}_{\w_6}^-)_{\mathtt{E_6}}$ have a very similar shape. So with same arguments as above we conclude:
\begin{prop} The vectors $f^{\bs}v_{\w_1}$, $\bs \in S(\w_1)$ are a FLL basis of $V(\w_1). \,\Box$
\end{prop}

\subsection{Type \texorpdfstring{$\mathtt{E_7}$}{E7}}
Let $\Lg$ be the simple Lie algebra of type $\mathtt{E_7}$ with associated Dynkin diagram
\begin{center}
  \begin{tikzpicture}[scale=.4]
    \draw (-1,0) node[anchor=east]  {$\mathtt{E_7}$};
    \foreach \x in {0,...,5}
    \draw[thick,xshift=\x cm] (\x cm,0) circle (1 mm);
    \foreach \y in {0,...,4}
    \draw[thick,xshift=\y cm] (\y cm,0) ++(.3 cm, 0) -- +(14 mm,0);
    \draw[thick] (4 cm,2 cm) circle (1 mm);
    \draw[thick] (4 cm, 3mm) -- +(0, 1.4 cm);
    \draw (0,0) node[anchor=north]  {\tiny{1}};
    \draw (2,0) node[anchor=north]  {\tiny{3}};
    \draw (4,0) node[anchor=north]  {\tiny{4}};
    \draw (6,0) node[anchor=north]  {\tiny{5}};
    \draw (8,0) node[anchor=north]  {\tiny{6}};
		\draw (10,0) node[anchor=north]  {\tiny{7}};
    \draw (4,3) node[anchor=north]  {\tiny{2}};
  \end{tikzpicture}
\end{center}
In this case $\omega = \omega_7$ is the only fundamental weight satisfying $\langle \omega,\theta^{\vee}\rangle = 1$.
%Then $\omega = \omega_7$ is the only minuscule fundamental weight and let $v_\omega$ be a highest weight vector for the irreducible representation $V(\omega)$.
%The set $\rr$ is given as follows.
\begin{center} $\begin{array}{|l|l|l|}
\hline
\beta_1=(2,2,3,4,3,2,1) & \beta_{10}=(1,1,2,3,2,1,1) & \beta_{19}=(1,1,1,1,1,1,1)\\
\beta_2=(1,2,3,4,3,2,1) & \beta_{11}=(1,1,1,2,2,2,1) & \beta_{20}=(0,1,1,1,1,1,1)\\
\beta_3=(1,2,2,4,3,2,1) & \beta_{12}=(1,1,2,2,2,1,1) & \beta_{21}=(1,0,1,1,1,1,1)\\
\beta_4=(1,2,2,3,3,2,1) & \beta_{13}=(0,1,1,2,2,2,1) & \beta_{22}=(0,0,1,1,1,1,1)\\
\beta_5=(1,1,2,3,3,2,1) & \beta_{14}=(1,1,1,2,2,1,1) & \beta_{23}=(0,1,0,1,1,1,1)\\
\beta_6=(1,2,2,3,2,2,1) & \beta_{15}=(1,1,2,2,1,1,1) & \beta_{24}=(0,0,0,1,1,1,1)\\
\beta_7=(1,1,2,3,2,2,1) & \beta_{16}=(0,1,1,2,2,1,1) & \beta_{25}=(0,0,0,0,1,1,1)\\
\beta_8=(1,2,2,3,2,1,1) & \beta_{17}=(1,1,1,2,1,1,1) & \beta_{26}=(0,0,0,0,0,1,1)\\
\beta_9=(1,1,2,2,2,2,1) & \beta_{18}=(0,1,1,2,1,1,1) & \beta_{27}=(0,0,0,0,0,0,1)\\
\hline	
\end{array}$ \end{center}
As in the $\mathtt{E_6}$-case the Hasse diagram has no $k$-chains. In addition there are only co-chains of cardinality at most 2, except for one with cardinality 3 (see Appendix, Figure \ref{hasseE7}).
As before the polytope is defined by the paths in the Hasse diagram. For $m \in \Z_{\geq 0}$ we have:
\begin{equation*}
\pmw=\{ \textbf{x} \in \R_{\geq 0}^{27}\mid \sum_{\beta_j \, \in \, \bp} x_{j} \leq m,\ \forall \bp\in D_{\omega} \}.
\end{equation*}
Because the polytope is defined by 77 non-redundant inequalities we will not state it explicitly.
\begin{prop} The set $\B_{\omega} = \{f^{\bs}v_{\w} \mid \bs \in S(\w)\}$ is a FFL basis of $V(\w)$.
\end{prop}
\begin{proof}
The co-chains of the Hasse diagram give us immediately:
\begin{eqnarray*}
\B_{\omega}=\{\!\!\!\!\!\!\!\!\!\!\!&&v_{\omega}, f_{1}v_{\omega},f_{2}v_{\omega}, \dots, f_{{27}}\vw, f_{{5}}f_{{6}}\vw, f_{{5}}f_{{8}}\vw,f_{{7}}f_{{8}}\vw, f_{{8}}f_{{9}}\vw,\\&& f_{{9}}f_{{10}}\vw,f_{{8}}f_{{11}}\vw, f_{{10}}f_{{11}}\vw, f_{{11}}f_{{12}}\vw, f_{{8}}f_{{13}}\vw, f_{{10}}f_{{13}}\vw,\\&& f_{{12}}f_{{13}}\vw,f_{{13}}f_{{14}}\vw,f_{{11}}f_{{15}}\vw,f_{{13}}f_{{15}}\vw,f_{{14}}f_{{15}}\vw,f_{{15}}f_{{16}}\vw,\\&&f_{{13}}f_{{17}}\vw,f_{{16}}f_{{17}}\vw,f_{{13}}f_{{19}}\vw,f_{{16}}f_{{19}}\vw,f_{{18}}f_{{19}}\vw,f_{{13}}f_{{21}}\vw,\\&&f_{{16}}f_{{21}}\vw,f_{{18}}f_{{21}}\vw,f_{{20}}f_{{21}}\vw,f_{{21}}f_{{23}}\vw,f_{{22}}f_{{23}}\vw,f_{{13}}f_{{14}}f_{{15}}\vw\}.
\end{eqnarray*}
Note that there are 56 elements in $\B_{\omega}$. By Corollary $\ref{Spanvmw}$, we get that this is a spanning set of $V(\omega)$. By \cite[p. 303]{Car05} we have $\dim V(\omega) = 56$ and therefore that $\B_{\omega}$  is a basis.
\end{proof}
\subsection{Type \texorpdfstring{$\mathtt{F_4}$}{F4}}
Let $\Lg$ be the simple Lie algebra of type $\mathtt{F_4}$ with associated Dynkin diagram
\begin{center}
  \begin{tikzpicture}[scale=.4]
    \draw (-3,0) node[anchor=east]  {$\mathtt{F_4}$};
		\draw[thick] (-2 ,0) circle (1mm);
		\draw[thick] (0: -1.7cm) -- +(1.4 cm, 0);
		\draw[thick] (4 ,0) circle (1mm);
    \draw[thick] (0 ,0) circle (1mm);
    \draw[thick] (2 cm,0) circle (1mm);
		\draw[thick] (0: 2.3cm) -- +(1.4 cm, 0);
    \draw[thick] (20: 3mm) -- +(1.5 cm, 0);
    %\draw[thick] (0: 2.8 mm) -- +(0.7 cm, 0);
		\draw (1.8,0) node[anchor=east] {\small{{$\bf\boldsymbol{<}$}}};
		%\draw[thick] (1mm) -- +(0.7 cm, 0);
    %\draw[thick] (0.9 cm, 0)[line width=0.5mm][<-] -- +(0.01 cm, 0);
    %\draw[thick] (0.9 cm, 0) -- +(0.88 cm, 0);
    \draw[thick] (-20: 3 mm) -- +(1.5 cm, 0);
				\draw (-2,0) node[anchor=north]  {\tiny{1}};

		\draw (0,0) node[anchor=north]  {\tiny{2}};
		\draw (2,0) node[anchor=north]  {\tiny{3}};
		\draw (4,0) node[anchor=north]  {\tiny{4}};
  \end{tikzpicture}
\end{center}
The highest root is of the form $\theta=2\alpha_1 +3\al_2+4\al_4+2\al_4$. And we have $\theta^{\vee}=2\alpha_1^{\vee}+3\alpha_2^{\vee}+2\alpha_3^{\vee}+\alpha_4^{\vee}$.  So $\langle \omega,\theta^{\vee}\rangle=1\Leftrightarrow \omega=\w_4$, so we consider the case $\omega=\omega_4$. If we construct $H(\Ln_{\omega}^-)_{\mathtt{F_4}}$ as in Section $\ref{test}$ we get a 3-chain of length 2, but here we are able to solve this problem. Therefore we will change the order of the roots such that we can draw a new diagram without any $k$-chains. As usual we start with the set of roots $\rr$:
%In this subsection we will consider highest weight representations for the Lie algebra $\mathfrak{g}$ of type $\mathtt{F_4}$ and the highest weight $\omega=\omega_4$. First of all we have to remark that our usual procedure will not give us a basis of $\Vw=\unm\vw$. The reason for that is the $\partial_3$-chain of length 2 in the relevant part of the Hasse diagram. If we ignore this chain and start our standard procedure we would get the following total order of relevant positive roots:
\begin{center} $\begin{array}{|l|l|l|}
\hline
\beta_1=(2,3,4,2) & \beta_6=(1,2,3,1) & \beta_{11}=(0,1,2,1) \\
\beta_2=(1,3,4,2) & \beta_7=(1,1,2,2) & \beta_{12}=(1,1,1,1) \\
\beta_3=(1,2,4,2) & \beta_8=(1,2,2,1) & \beta_{13}=(0,1,1,1) \\
\beta_4=(1,2,3,2) & \beta_9=(0,1,2,2) & \beta_{14}=(0,0,1,1) \\
\beta_5=(1,2,2,2) & \beta_{10}=(1,1,2,1) & \beta_{15}=(0,0,0,1) \\
\hline	
\end{array}$ \end{center}
Here we have $\beta_i\succ \beta_j \Leftrightarrow i>j$. With this order we are not able to find relations derived from differential operators (see Section \ref{spanning}), which include the rootvector $f_{{4}}$ (see $\eqref{probichain}$). In order to find relations including $f_4$ we adjust the order on the roots in this case as follows:\\
\begin{equation*}
\beta_1 \prec \beta_2 \prec \beta_3 \prec \beta_5 \prec \beta_4\prec \beta_6\prec \beta_7\prec \cdots\prec \beta_{15}.
\end{equation*}
So we just switched the positions of $\beta_4$ and $\beta_5$.
Now we consider our Hasse diagram constructed as in Section $\ref{test}$ and the diagram we obtain by changing the order of the roots and by using differential operators corresponding to non-simple roots, see Figure \ref{hasseF4}.\\
The idea of this adjustment is that we split up the $3$-chain by using the non-simple differential operators mentioned above. After this we still want to get as many roots as possible on each path. To do so we use two non-simple differential operators: $\partial_{0110}=\partial_{\al_2+\al_3} $ and $\partial_{0011}=\partial_{\al_3+\al_4}$. In the adjusted diagram also occurs a directed edge labeled by $\mathfrak a$ from $\beta_2$ to $\beta_5$ and a second labeled by $\mathfrak b$ from $\beta_5$ to $\beta_4$. We cannot label the second edge with a differential operator, because there is no element $\gamma\in\Delta_+$ satisfying: $\beta_5-\gamma=\beta_4$. We will use the following observation to explain the existence of these edges and labels. For $a_0,b_0 \in \C \setminus \{0\}$ we have:
\begin{align*}
\partial_{3}^{n_2+2n_3}\partial_{2}^{n_2+n_3}\partial_{1}^{n_1}f_1^{m+1}=&\ \partial_3^{n_2+2n_3}(a_0 f_3^{n_2+n_3}f_2^{n_1-n_2-n_3}f_1^{m+1-n_1})\\
=&\  b_0 f_5^{n_3}f_4^{n_2}f_2^{n_1-n_2-n_3}f_1^{m+1-n_1}+ \text{smaller terms}.
\end{align*}
That means we can replace in the path consisting of $\beta_1$, $\beta_2$, $\beta_3$ and $\beta_4$ the root $\beta_3$ by $\beta_5$. Furthermore the differential operators $\partial_{\al_2+\al_3} $ and $\partial_{\al_3+\al_4}$ have no influence on $\beta_5$. That is the reason for the directed edge labeled by $\mathfrak b$ from $\beta_{5}$ to $\beta_{4}$. The reason for the edge between $\beta_2$ and $\beta_5$ is that we want to visualize the co-chain which we construct at this point. We label this edge with $\mathfrak a$ to prevent confusions about the applied differential operators, where $\mathfrak a$ corresponds to $\partial_{3}^{n_2+2n_3}$. We note that the changed Hasse diagram gives us directly the inequalities of $P(\lambda)$, but in this case it does not describe in general the action of the differential operators.\\ If we now follow our standard procedure with the adjusted Hasse diagram the next step is to define the polytope associated to the set of Dyck paths $D_{\omega}$ and $m\in\Z_{\geq 0}$:
\begin{equation*}
\pmw=\{\bx\in \R_{\geq 0}^{15}\mid \sum_{\beta_j \in \bp} x_{j} \leq m,\ \forall \bp\in D_{\w}\}.
\end{equation*}
More explicitly: $\pmw$ is the set of all elements $\bx\in\R_{\geq 0}^{15}$ such that the 12 inequalities, which can be found in the Appendix, Figure \ref{polyF4}, are satisfied.\\ The set $\Bw=\{f^{\bs}\vw \mid \bs\in \sw\}\subset \Vw$ is given by:
\begin{align*}
\Bw=& \{\vw,f_{{1}}\vw,f_{{2}}\vw,\dots,f_{{15}}\vw,f_{{3}}f_{{5}}\vw,f_{{4}}f_{{6}}\vw,f_{{5}}f_{{6}}\vw,f_{{6}}f_{{7}}\vw,\\
&\ f_{{7}}f_{{8}}\vw,f_{{6}}f_{{9}}\vw,f_{{8}}f_{{9}}\vw,f_{{9}}f_{{10}}\vw,f_{{9}}f_{{12}}\vw,f_{{11}}f_{{12}}\vw
\}.
\end{align*}
%With this knowledge and the earlier results it will be possible to proof the following statement:
\begin{prop}
The set $\B_{\omega} = \{f^{\bs}v_{\w} \mid \bs \in S(\w)\}$ is a FFL basis of $V(\w)$.
\end{prop}
\begin{proof}
By Corollary \ref{Spanvmw} we conclude that $\Bw$ spans the vector space $\Vw$. In addition we know by \cite[p. 303]{Car05} that $\dim V(\omega)=26=|\Bw|$. Hence the set $\Bw$ is a basis.
\end{proof}

\begin{center}
\begin{minipage}{\linewidth}
\centering
\begin{tikzpicture}[scale=1]
  \node (1) at (-3.75,-3) {$\beta_1$};
  \node (2) at (-2.5,-3) {$\beta_2$};
  \node (3) at (-1.25,-3) {$\beta_3$};
  \node (4) at (0,-3) {$\beta_4$};
  \node (5) at (-1,-4) {$\beta_5$};
  \node (6) at (1,-4) {$\beta_6$};
  \node (7) at (-2,-5) {$\beta_7$};
	\node (8) at (0,-5) {$\beta_8$};
  \node (9) at (-3,-6) {$\beta_9$};
  \node (10) at (-1,-6) {$\beta_{10}$};
  \node (11) at (-2,-7) {$\beta_{11}$};
  \node (12) at (0,-7) {$\beta_{12}$};
  \node (13) at (-1,-8) {$\beta_{13}$};
  \node (14) at (0.25,-8) {$\beta_{14}$};
  \node (15) at (1.5,-8) {$\beta_{15}$};

	\node (a) at (2,-5) {$\leadsto$};
	\node (b) at (2.5,-5) {};
	
	\node (16) at (3.75,-2) {$\beta_1$};
  \node (17) at (5,-2) {$\beta_2$};
  \node (18) at (4,-3) {$\beta_3$};
  \node (19) at (6,-3) {$\beta_5$};
  \node (20) at (3,-4) {$\beta_6$};
  \node (21) at (5,-4) {$\beta_4$};
  \node (22) at (4,-5) {$\beta_8$};
	\node (23) at (6,-5) {$\beta_7$};
  \node (24) at (5,-6) {$\beta_{10}$};
  \node (25) at (7,-6) {$\beta_{9}$};
  \node (26) at (4,-7) {$\beta_{12}$};
  \node (27) at (6,-7) {$\beta_{11}$};
  \node (28) at (5,-8) {$\beta_{13}$};
  \node (29) at (6.25,-8) {$\beta_{14}$};
  \node (30) at (7.5,-8) {$\beta_{15}$};

   \draw[->] (1) edge node[above] {\tiny{1}} (2);
   \draw[->] (2) edge node[above] {\tiny{2}} (3);
   \draw[->] (3) edge node[above] {\tiny{3}} (4);
   \draw[->] (4) edge node[left,pos=0.2] {\tiny{3}} (5);
   \draw[->] (4) edge node[right,pos=0.2] {\tiny{4}} (6);
   \draw[->] (5) edge node[left,pos=0.2] {\tiny{2}} (7);
   \draw[->] (5) edge node[right,pos=0.2] {\tiny{4}} (8);
   \draw[->] (6) edge node[left,pos=0.2] {\tiny{3}} (8);
   \draw[->] (7) edge node[left,pos=0.2] {\tiny{1}} (9);
   \draw[->] (7) edge node[right,pos=0.2] {\tiny{4}} (10);
   \draw[->] (8) edge node[left,pos=0.2] {\tiny{2}} (10);
   \draw[->] (9) edge node[right,pos=0.2] {\tiny{4}} (11);
   \draw[->] (10) edge node[left,pos=0.2] {\tiny{1}} (11);
   \draw[->] (10) edge node[right,pos=0.2] {\tiny{3}} (12);
   \draw[->] (11) edge node[right,pos=0.2] {\tiny{3}} (13);
   \draw[->] (12) edge node[left,pos=0.2] {\tiny{1}} (13);
   \draw[->] (13) edge node[above] {\tiny{2}} (14);
   \draw[->] (14) edge node[above] {\tiny{3}} (15);

   %\draw[->] (a) --+ (b);

  \draw[->] (16) edge node[above] {\tiny{1}} (17);
   \draw[->] (17) edge node[left,pos=0.2] {\tiny{2}} (18);
   \draw[->] (17) edge node[right,pos=0.2] {\tiny{$\mathfrak{a}$}} (19);
   \draw[->] (18) edge node[left,pos=0.2] {\tiny{0011}} (20);
   \draw[->] (18) edge node[right,pos=0.2] {\tiny{3}} (21);
   \draw[->] (19) edge node[left,pos=0.2] {\tiny{$\mathfrak{b}$}} (21);
   \draw[->] (20) edge node[right,pos=0.2] {\tiny{3}} (22);
   \draw[->] (21) edge node[left,pos=0.2] {\tiny{0011}} (22);
   \draw[->] (21) edge node[right,pos=0.2] {\tiny{0110}} (23);
   \draw[->] (22) edge node[right,pos=0.2] {\tiny{2}} (24);
   \draw[->] (23) edge node[left,pos=0.2] {\tiny{4}} (24);
   \draw[->] (23) edge node[right,pos=0.2] {\tiny{1}} (25);
   \draw[->] (24) edge node[left,pos=0.2] {\tiny{3}} (26);
   \draw[->] (24) edge node[right,pos=0.2] {\tiny{1}} (27);
   \draw[->] (25) edge node[left,pos=0.2] {\tiny{4}} (27);
   \draw[->] (26) edge node[right,pos=0.2] {\tiny{1}} (28);
   \draw[->] (27) edge node[left,pos=0.2] {\tiny{3}} (28);
   \draw[->] (28) edge node[above] {\tiny{2}} (29);
   \draw[->] (29) edge node[above] {\tiny{3}} (30);
\end{tikzpicture} 
\captionof{figure}{$H(\Ln_{\omega}^-)_{\mathtt{F_4}}$}
    \label{hasseF4}
  \end{minipage}
\end{center}

\subsection{Type \texorpdfstring{$\mathtt{G_2}$}{G2}}
Let $\Lg$ be the simple Lie algebra of type $\mathtt{G_2}$ with associated Dynkin diagram
\begin{center}
  \begin{tikzpicture}[scale=.4]
    \draw (-1,0) node[anchor=east]  {$\mathtt{G_2}$};
    \draw[thick] (0 ,0) circle (1mm);
    \draw[thick] (2 cm,0) circle (1mm);
    \draw[thick] (20: 3mm) -- +(1.5 cm, 0);
    \draw[thick] (0: 2.8 mm) -- +(0.7 cm, 0);
		\draw (1.8,0) node[anchor=east] {\small{{$\bf\boldsymbol{<}$}}};
		%\draw[thick] (1mm) -- +(0.7 cm, 0);
    %\draw[thick] (0.9 cm, 0)[line width=0.5mm][<-] -- +(0.01 cm, 0);
    \draw[thick] (0.9 cm, 0) -- +(0.88 cm, 0);
    \draw[thick] (-20: 3 mm) -- +(1.5 cm, 0);
		\draw (0,0) node[anchor=north]  {\tiny{1}};
		\draw (2,0) node[anchor=north]  {\tiny{2}};
  \end{tikzpicture}
\end{center}
For the highest root $\theta = 3 \alpha_1 + 2 \al_2$ we have $\theta^{\vee} = \al_1^{\vee} + 2\al_2^{\vee}$.
%$\langle \omega_1, \theta^{\vee} \rangle = 1$ and $\langle \omega_2, \theta^{\vee} \rangle = 2$.
So we consider $\omega = \omega_1$. In this case the Hasse diagram has one $1$-chain. We will rewrite $H(\Ln_{\omega}^-)_{\mathtt{G_2}}$ into a diagram without any $k$-chains. Consider the following order on $\rr$:
\begin{equation*}
\beta_1 \prec \beta_2 \prec \beta_4 \prec \beta_5 \prec \beta_3,
\end{equation*}
where
\begin{center} $\begin{array}{|l|l|l|l|l|}
\hline
\beta_1=(3,2) & \beta_2=(3,1) & \beta_{3}=(2,1) & \beta_{4}=(1,1) & \beta_{5} =(1,0) \\
\hline	
\end{array}$\end{center}
So we obtain the following diagrams:
\begin{center}
\begin{tikzpicture}
  \node (1) at (-6.75,4) {$\beta_1$};
  \node (2) at (-5.5,4) {$\beta_2$};
  \node (3) at (-4.25,4) {$\beta_3$};
  \node (4) at (-3,4) {$\beta_4$};
  \node (5) at (-1.75,4) {$\beta_5$};
	%\node (6) at (-0.5,4) {};
	\node (6) at (0,4) {$\leadsto$};
	\node (7) at (0.5,4) {};
	\node (8) at (1.75,4) {$\beta_1$};
  \node (9) at (3,4) {$\beta_3$};
%  \node (10) at (4,4.5) {$\beta_2$};
  \node (11) at (4,5) {$\beta_2$};
  \node (12) at (4,3) {$\beta_4$};
	\node (13) at (5,4) {$\beta_5$,};
%  \node (3) at (0,4) {$\beta_3$};
%  \node (a) at (-2,2) {$(0,1,1)$};
%  \node (b) at (0,2) {$(1,0,1)$};
%  \node (c) at (2,2) {$(1,1,0)$};
%  \node (d) at (-2,0) {$(0,0,1)$};
%  \node (e) at (0,0) {$(0,1,0)$};
%  \node (f) at (2,0) {$(1,0,0)$};
%  \node (min) at (0,-2) {$(0,0,0)$};
   %\draw (1) -- (2);
   \draw[->] (1) edge node[above] {\tiny{2}} (2);
   \draw[->] (2) edge node[above] {\tiny{1}} (3);
   \draw[->] (4) edge node[above] {\tiny{2}} (5);
   \draw[->] (3) edge node[above] {\tiny{1}} (4);

   %\draw[->] (6) --+ (7) ;

   \draw[->] (8) edge node[above] {\tiny{11}} (9);
   \draw[->] (9) edge node[right, pos=0.25] {\tiny{21}} (12);
   \draw[->] (12) edge node[right, pos=0.25] {\tiny{2}} (13);
   \draw[->] (9) edge node[right, pos=0.25] {\tiny{2}} (11);
   \draw[->] (11) edge node[right, pos=0.25] {\tiny{21}} (13);
  % \draw (0,4.5) node[anchor=west]  {\tiny{4}};
%  \draw (min) -- (d) -- (a)  -- (b) -- (f)
%  (e) -- (min) -- (f) -- (c) --  (d) -- (b);
%  \draw[preaction={draw=white, -,line width=6pt}] (a) -- (e) -- (c);
\end{tikzpicture}
\end{center}
%We have to check that the new diagram can be regarded as a diagram with no chains.
%The first edge is labeled by $11 = \al_1 + \al_2$ and we have $\beta_{3} - (\alpha_1 + \al_2) = \beta_{5}$. So to be consistent we should draw an edge from $\beta_{3}$ to $\beta_{5}$. But by the change of order $\beta_{3}$ is larger than $\beta_5$ and therefore we can neglect this edge. With the same argument, we can forget about the edge from $\beta_2$ to $\beta_3$ which would be otherwise labeled by $1$.\\
%We have $\beta_1 - \al_2 = \beta_2$ and $\beta_4 - \al_2 = \beta_5$. If we subtract $\al_2$ from another root we will not end up in a root. So instead of drawing an edge directly from $\beta_1$ to $\beta_2$, we can draw an edge, labeled by 2, from $\beta_3$ to $\beta_2$. The other edges do not cause any problems.
Very similar arguments as in the case of $\mathtt{B_3}, \w_{1}$ show that we can apply the results of section $\ref{spanning}$ to the rewritten diagram. We  consider the polytope associated to the new diagram for $m \in \Z_{\geq 0}$:
%\begin{equation*}
%\pmw=
%%\left\{
%\{
 %\textbf{x} \in \R_{\geq 0}^{5}\mid
%%\begin{aligned}
%%x_{\beta_1} + x_{\beta_2} + x_{\beta_3} + x_{\beta_5} \leq m\\
%%x_{\beta_1} + x_{\beta_3} + x_{\beta_4} + x_{\beta_5} \leq m
%%\end{aligned}
%x_{\beta_1} + x_{\beta_2} + x_{\beta_3} + x_{\beta_5} \leq m,
%x_{\beta_1} + x_{\beta_3} + x_{\beta_4} + x_{\beta_5} \leq m
%%\right\}
%\}
%\end{equation*}
\begin{equation*}\pmw=\left\{\textbf{x} \in \R_{\geq 0}^{N}\mid
\begin{aligned}
x_{1} + x_{2} + x_{3} + x_{5} \leq m \\
x_{1} + x_{3} + x_{4} + x_{5} \leq m \\
\end{aligned}
\right\}.\end{equation*}
By Section $\ref{spanning}$ the elements $\vw$, $f_{1}\vw$, $f_{2}\vw$, $f_{3}\vw$, $f_{4}\vw$, $f_{5}\vw, f_{2}f_{4}\vw$ span $\Vw$ and with \cite[p. 316]{Car05} we know $\dim \Vw = 7$. 
\begin{prop}
The set $\B_{\omega} = \{f^{\bs}v_{\w} \mid \bs \in S(\w)\}$ is a FFL basis of $V(\w)$.\hfill $\Box$
\end{prop}
\begin{proof}
The previous observations imply that $\{f^{\bs}v_{\w} \mid \bs \in S(\w)\}$ is a  basis of $V(\w)$. It remains to show that $P(\omega)$ is a normal polytope.\\ Like in the case of $(\mathtt{B_n}, \w_1)$ we have to change the order of the roots to apply Section $\ref{helpful}$. One possible order is $\beta_1 \prec \beta_3 \prec \beta_4 \prec \beta_2 \prec \beta_5$. Using this order we conclude that $P(\w)$ is a normal polytope.
\end{proof}
%\begin{rem}
%Like in the case of $(\mathtt{B_n}, \w_1)$, in particular analogue to Remark \ref{orderchange}, we have to choose a different order to apply Section \ref{helpful}. We choose the order $\beta_1 \prec \beta_3 \prec \beta_4 \prec \beta_2 \prec \beta_5$.
%\end{rem}
%Again as in the $\mathtt{B_3}, \w_1$ case, with the order $\beta_1 \prec \beta_3 \prec \beta_2 \prec \beta_4 \prec \beta_5$  we get
%\begin{prop}
%For all $m \in \Z_{\geq 1}$ we have
%\begin{equation*}
%S(m\omega) = \smmew + \sw.
%\end{equation*}
%\end{prop} 

%% file: PBW-Basis_linearindependence.tex
\section{Linear Independence}\label{linearinde} 
%\subsection{}
We refer to the notation of Section \ref{test}, especially Subsection \ref{applications}. Throughout the Section we assume the vectors $f^{\bp}\vl\in\Vl$ to be ordered as in $\eqref{fixedorder}$ and we fix $\lambda = m\omega$ where $\omega$ appears in Table \ref{fundamentalweights}.\\
We want to investigate the connection between our polytope $P(\lambda)$ and the essential multi-exponents. Via this connection and with the results from Section \ref{spanning} we want to prove that $\{f^{\bs}\vl\mid \bs\in S(\lambda)\}$ provides a FFL basis of $\Vl$.\\
Note that one can define essential monomials like in Subsection \ref{applications} for an arbitrary total order on $\rrl$. Hence for the following statements it is very important that we kept in Subsection \ref{applications} the total order introduced in Subsection \ref{Definitions}. 
\setcounter{thm}{0}
\setcounter{subsection}{1}
\begin{lem}\label{lemequa}
If $\{f^{\bs}v_{\lambda} \mid \bs \in \smw\}$ is linear independent in $\Vmw$, then 
\begin{equation*}
\smw = \es (\Vmw).
\end{equation*}
\end{lem}
\begin{proof}
Let $\bs \in \es (\Vmw) = \{\bp \in \Z_{\geq 0}^N \mid f^{\bp}\vmw\notin \textrm{span}\{f^{\bq}\vmw \mid \bq \prec \bp\}\}$ and assume $\bs \notin \smw$. By Proposition $\ref{important}$ we can rewrite $f^{\bs}$ such that
\begin{equation*}
f^{\bs}\vmw = \underset{\bt \, \prec \, \bs}{\sum} c_{\bt} f^{\bt}\vmw, c_{\bt} \in \C
\end{equation*}
and we get immediately a contradiction, so $\bs \in \smw$.\\
Now let $\bs \in \smw$ and $\bs \notin \es (\Vmw)$. Then $f^{\bs}\vmw \in \textrm{span}\{f^{\bq}\vmw \mid \bq \prec \bs\}$ and so 
\begin{equation}\label{bla23}
f^{\bs}\vmw = \underset{\bq\, \prec \, \bs}{\sum} c_{\bq}f^{\bq}\vmw ,
%=\underset{\bq' \, \in \, \smw}{\sum} c_{\bq'}f^{\bq'}\vmw, 
\end{equation}
for some $c_{\bq} \in \C$.
%,$ c_{\bq'} \in \C$
%By Corollary $\ref{Spanvmw}$ $\smw$ induces a spanning set, in particular by the proof of Proposition $\ref{important}$, we rewrite every element as a sum of \textbf{smaller} elements until we get a sum of multi-exponents $\bq' \in \smw$. 
We rewrite each $f^{\bq}v_\lambda$ in terms of basis elements $f^{\bt}v_\lambda, \bt \in \smw$. Because of the linear independence all prefactors are zero, meaning that ${\bs} = 0$. But this is a contradiction to $\bs \notin \es \Vmw$.
%Let $\bs \in \es (\Vmw) = \{\bp \in \Z_{\geq 0}^N \mid f^{\bp}\vmw\notin \text{span}\{f^{\bq}\vmw \mid \bq < \bp\}\}$ and assume $\bs \notin \smw$. By $\ref{important}$ we can rewrite $f^{\bs}$ as
%\begin{equation*}
%f^{\bs}\vmw = \underset{\bt \, \prec \, \bs}{\sum} c_{\bt} f^{\bt}\vmw, c_{\bt} \in \C.
%\end{equation*}
%and we get immediately a contradiction, so $\bs \in \smw$.
\end{proof}
\begin{thm}\label{theoremBB}
The elements $\{f^{\bs}(v_{\lambda -\omega} \otimes v_{\omega}) \mid \bs \in \smw \} \, \subset V(\lambda - \omega) \odot V(\omega)$ are linearly independent and $\mathbb{B}_{\lambda} = \{f^{\bs}v_{\lambda} \mid \bs \in \smw\}$ is a FFL basis of $V(\lambda)$. 
\end{thm}
\begin{proof} We want to prove this statement by induction on $m \in \Z_{\geq 1}$. For $m = 1$ we saw in Section $\ref{uebergreifend2}$ that $\mathbb{B}_{\omega} = \{f^{\bs}\vw \mid \bs \in \sw\}$ is a basis for $V(\omega)$ in each type.\\
So let $m \in \Z_{\geq 2}$ be arbitrary and we assume that the claim holds for all $m' < m$. By induction the set $\mathbb{B}_{\lambda - \omega} = \{f^{\bs}v_{\lambda -\omega} \mid \bs \in S(\lambda - \omega)\}$ is a basis of $V(\lambda -\omega)$. So we have by Lemma $\ref{lemequa}$
\begin{equation}\label{Induction}
\es (V(\lambda - \omega) = S(\lambda - \omega) \text{ and } \es (\Vw) = \sw.
\end{equation}
But then with \cite[Prop. 1.11]{FFoL13a}: 
\begin{equation*}
\es (V(\lambda - \omega) + \es (\Vw) \subset \es (V(\lambda - \omega) \odot V(\omega))
\end{equation*}
and so we get the linearly independence of
%and the Minkowski sum property
%\begin{equation*}
%\smw= S((m-1)\omega) + \sw = \es (V((m-1)\omega) + \es (\Vw) \subset \es (\Vmw).
%\end{equation*}
\begin{equation*}
\{f^{\bs}(v_{\lambda - \omega} \otimes v_{\omega}) \mid \bs \in \es (V(\lambda - \omega) + \es (\Vw)\} \subset V(\lambda - \omega) \odot V(\omega)
\end{equation*}
With the equalities in $\eqref{Induction}$ and Section $\ref{helpful}$ where we proved $S(\lambda - \omega) + \sw = \smw$, we conclude that the set
\begin{equation*}
\{f^{\bs}(v_{\lambda - \omega} \otimes v_{\omega}) \mid \bs \in \smw\} \subset V(\lambda - \omega) \odot V(\omega)
\end{equation*}
is linearly independent. So we get $\dim \Vmw \geq |\smw|$ and with the spanning property Corollary $\ref{Spanvmw}$ we have
$|\smw| \geq \dim \Vmw.$ Finally we get 
\begin{equation*}
|\smw| = \dim \Vmw
\end{equation*}
and that $\mathbb{B}_{\lambda}$ is a FFL basis of $\Vmw$ as claimed.
\end{proof}
\begin{rem}
The basis $\mathbb{B}_{\lambda}$ is a monomial basis, so we get an induced FFL basis of $\Vmwa$.
\end{rem}
\begin{thm}\label{thereomAA}
Let $\Vmwa \cong \Sn /\imw$. Then the ideal $\imw$ is generated by
\begin{equation*}
\Unp \circ \mathrm{span}\{f_{\beta}^{\langle \lambda,\beta^{\vee} \rangle + 1} \mid \beta \in \Delta_{+}\}
\end{equation*}
as $\Sn$ ideal.\\
In particular we have that $\imw = \Sn(\Unp \circ \mathrm{span}\{f_\beta, f_{\theta}^{m+1} \mid \beta \in \Delta_{+}\!\setminus\Delta_{+}^{\lambda}\})$.
\end{thm}
\begin{proof} Let $I$ be an Ideal generated by $\Unp \circ \textrm{span}\{f_{\beta}^{\langle \lambda,\beta^{\vee}\rangle + 1} \mid \beta \in \Delta_{+}\}$ as  $\Sn$ ideal. By $Iv_{\lambda} = \{0\}$ we have $I \subset \Imw$, so there is a canonical projection:
\begin{equation*}
\phi: \Sn / I \rightarrow \Sn / \imw \cong \Vmwa
\end{equation*}
Let $f^{\bt} = 0$ in $\Sn / \imw$. Because we have a basis of $\Vmwa$ we can rewrite $f^{\bt}$ as follows:
\begin{equation}\label{unbedeutend1}
f^{\bt} = \underset{\bs \, \in \, \smw}{\sum} c_{\bs} f^{\bs} \in \Sn / \imw
\end{equation}
for some $c_{\bs} \in \C$. In the proof of Theorem $\ref{thmspan}$ we already saw that the relations obtained by $I$ are sufficient to achieve $\eqref{unbedeutend1}$. So $0 = f^{\bt} = \underset{\bs \, \in \, \smw}{\sum} c_{\bs} f^{\bs} \in \Sn / I$. Therefore $\phi$ is injective.\\
In the proof of Proposition $\ref{important}$
% we use arbitrary powers of $f_{\theta}$ greater than $\langle \lambda,\beta^{\vee}\rangle$. We can apply the relations obtained by $f_{\theta}^{\langle \lambda,\beta^{\vee}\rangle + 1}$ multiple times, so it is enough to consider only this power. Again from the proof, it is clear that
we do not need powers $f_{\beta}$ for $\beta \in \rrl \setminus \{\theta\}$.
\end{proof}

%% file: PBW-Basis_Appendix.tex
\section*{Appendix}\label{appendix}
In this section we want to present the Hasse diagrams $H(\Ln_{\omega_6}^-)_{\mathtt{E_6}}$ and $H(\Ln_{\omega_7}^-)_{\mathtt{E_7}}$ for a better understanding of our work. In addition to illustrate the ordering of the roots for the classical types $\mathtt{A_n}$, $\mathtt{B_n}$ and $\mathtt{D_n}$ we give in Figure \ref{hasseA4} the complete Hasse diagram of $\mf{sl}_5$ and in Figure \ref{hasseB4B5} a concrete example of the Hasse diagram in the $\mathtt{D_n}$, $\omega_n$-case, for $n=5,6$. We remark that the shape of the Hasse diagram $H(\mf{n}_{\omega_{n-1}}^-)_{\mf{so}_{2n}}$ and $H(\mf{n}_{\omega_{n}}^-)_{\mf{so}_{2n}}$ is equal to the shape of $H(\mf{n}_{\omega_{n-1}}^-)_{\mf{so}_{2(n-1)+1}}$. So Figure \ref{hasseB4B5} shows also the shape of the Hasse diagrams $H(\mf{n}_{\omega_{4}}^-)_{\mf{so}_{10}}$, $H(\mf{n}_{\omega_{5}}^-)_{\mf{so}_{10}}$ and $H(\mf{n}_{\omega_{5}}^-)_{\mf{so}_{12}}$, $H(\mf{n}_{\omega_{6}}^-)_{\mf{so}_{12}}$. Furthermore we state the explicit polytopes for $\mathtt{E_6}$ (Table \ref{polyE6}), $\mathtt{F_4}$ (Table \ref{polyF4}) and for the special cases: $\mathtt{B_4}$, $\w_4$ ($\mathtt{D_5}$, $\w_4$) and $\mathtt{D_5}$ $\w_5$ (Table \ref{polyB4}).
\begin{center}
\begin{minipage}{\linewidth}
\centering
  \begin{tikzpicture}
  \node (1) at (0,0) {$\beta_1$};
  \node (2) at (-1,-1) {$\beta_2$};
  \node (3) at (1,-1) {$\beta_3$};
  \node (4) at (-2,-2) {$\beta_4$};
  \node (5) at (0,-2) {$\beta_5$};
  \node (6) at (2,-2) {$\beta_6$};
  \node (7) at (-3,-3) {$\beta_7$};
  \node (8) at (-1,-3) {$\beta_8$};
  \node (9) at (1,-3) {$\beta_9$};
  \node (10) at (3,-3) {$\beta_{10}$};

   \draw[->] (1) edge node[left, pos=0.25] {\tiny{1}} (2);
   \draw[->] (1) edge node[right, pos=0.25] {\tiny{4}} (3);
   \draw[->] (2) edge node[left, pos=0.25] {\tiny{2}}(4);
   \draw[->] (2) edge node[right, pos=0.25] {\tiny{4}} (5);
   \draw[->] (3) edge node[left, pos=0.25] {\tiny{1}} (5);
   \draw[->] (3) edge node[right, pos=0.25] {\tiny{3}} (6);
   \draw[->] (4) edge node[left, pos=0.25] {\tiny{3}} (7);
   \draw[->] (4) edge node[right, pos=0.25] {\tiny{4}} (8);
   \draw[->] (5) edge node[left, pos=0.25] {\tiny{2}} (8);
   \draw[->] (5) edge node[right, pos=0.25] {\tiny{3}} (9);
   \draw[->] (6) edge node[left, pos=0.25] {\tiny{1}} (9);
   \draw[->] (6) edge node[right, pos=0.25] {\tiny{2}} (10);

\end{tikzpicture}
    \captionof{figure}{Complete Hasse diagram of $\Lg=\mf{sl}_5$.}
    \label{hasseA4}
  \end{minipage}
\end{center}

\begin{center}
\begin{minipage}{\linewidth}
\centering
  \begin{tikzpicture}
  \node (1) at (0,0) {$\beta_1$};
  \node (2) at (1,-1) {$\beta_2$};
  \node (3) at (0,-2) {$\beta_3$};
  \node (4) at (2,-2) {$\beta_4$};
  \node (5) at (1,-3) {$\beta_5$};
  \node (6) at (3,-3) {$\beta_6$};
  \node (7) at (0,-4) {$\beta_7$};
  \node (8) at (2,-4) {$\beta_8$};
  \node (9) at (1,-5) {$\beta_9$};
  \node (10) at (0,-6) {$\beta_{10}$};

	\node (11) at (5,1) {$\beta_1$};
  \node (12) at (6,0) {$\beta_2$};
  \node (13) at (5,-1) {$\beta_3$};
  \node (14) at (7,-1) {$\beta_4$};
  \node (15) at (6,-2) {$\beta_5$};
  \node (16) at (8,-2) {$\beta_6$};
  \node (17) at (5,-3) {$\beta_7$};
  \node (18) at (7,-3) {$\beta_8$};
  \node (19) at (9,-3) {$\beta_9$};
  \node (20) at (6,-4) {$\beta_{10}$};
  \node (21) at (8,-4) {$\beta_{11}$};
  \node (22) at (5,-5) {$\beta_{12}$};
  \node (23) at (7,-5) {$\beta_{13}$};
  \node (24) at (6,-6) {$\beta_{14}$};
  \node (25) at (5,-7) {$\beta_{15}$};

   \draw[->] (1) edge node[right, pos=0.25] {\tiny{2}} (2);
   \draw[->] (2) edge node[left, pos=0.25] {\tiny{1}} (3);
   \draw[->] (2) edge node[right, pos=0.25] {\tiny{3}}(4);
   \draw[->] (3) edge node[right, pos=0.25] {\tiny{3}} (5);
   \draw[->] (4) edge node[left, pos=0.25] {\tiny{1}} (5);
   \draw[->] (4) edge node[right, pos=0.25] {\tiny{4}} (6);
   \draw[->] (5) edge node[left, pos=0.25] {\tiny{2}} (7);
   \draw[->] (5) edge node[right, pos=0.25] {\tiny{4}} (8);
   \draw[->] (6) edge node[left, pos=0.25] {\tiny{1}} (8);
   \draw[->] (7) edge node[right, pos=0.25] {\tiny{4}} (9);
   \draw[->] (8) edge node[left, pos=0.25] {\tiny{2}} (9);
   \draw[->] (9) edge node[left, pos=0.25] {\tiny{3}} (10);

   \draw[->] (11) edge node[right, pos=0.25] {\tiny{2}} (12);
   \draw[->] (12) edge node[left, pos=0.25] {\tiny{1}} (13);
   \draw[->] (12) edge node[right, pos=0.25] {\tiny{3}}(14);
   \draw[->] (13) edge node[right, pos=0.25] {\tiny{3}} (15);
   \draw[->] (14) edge node[left, pos=0.25] {\tiny{1}} (15);
   \draw[->] (14) edge node[right, pos=0.25] {\tiny{4}} (16);
   \draw[->] (15) edge node[left, pos=0.25] {\tiny{2}} (17);
   \draw[->] (15) edge node[right, pos=0.25] {\tiny{4}} (18);
   \draw[->] (16) edge node[left, pos=0.25] {\tiny{1}} (18);
   \draw[->] (16) edge node[right, pos=0.25] {\tiny{5}} (19);
   \draw[->] (17) edge node[right, pos=0.25] {\tiny{4}} (20);
   \draw[->] (18) edge node[left, pos=0.25] {\tiny{2}} (20);
   \draw[->] (18) edge node[right, pos=0.25] {\tiny{5}} (21);
   \draw[->] (19) edge node[left, pos=0.25] {\tiny{1}} (21);
   \draw[->] (20) edge node[left, pos=0.25] {\tiny{3}} (22);
   \draw[->] (20) edge node[right, pos=0.25] {\tiny{5}} (23);
   \draw[->] (21) edge node[left, pos=0.25] {\tiny{2}} (23);
   \draw[->] (22) edge node[right, pos=0.25] {\tiny{5}} (24);
   \draw[->] (23) edge node[left, pos=0.25] {\tiny{3}} (24);
   \draw[->] (24) edge node[left, pos=0.25] {\tiny{4}} (25);
\end{tikzpicture}
    \captionof{figure}{$H(\mf{n}_{\omega_{4}}^-)_{\mf{so}_{9}}$ , $H(\mf{n}_{\omega_{5}}^-)_{\mf{so}_{11}}$}
    \label{hasseB4B5}
  \end{minipage}
\end{center}

\begin{center}
\begin{minipage}{\linewidth}
\centering
\begin{align*}
&x_1+x_2+x_3+x_5+x_7+x_9+x_{10}\leq m\\
&x_1+x_2+x_3+x_5+x_8+x_9+x_{10}\leq m\\
&x_1+x_2+x_4+x_5+x_7+x_9+x_{10}\leq m\\
&x_1+x_2+x_4+x_5+x_8+x_9+x_{10}\leq m\\
&x_1+x_2+x_4+x_6+x_8+x_9+x_{10}\leq m
\end{align*}
    \captionof{table}{Polytope $P(m\omega_4)$ corresponding to $\Lg=\mf{so}_{9}$ and $P(m\omega_4)$, $P(m\omega_5)$ corresponding to $\Lg=\mf{so}_{10}$.}
    \label{polyB4}
  \end{minipage}
\end{center}

\begin{center}
\begin{minipage}{\linewidth}
\centering
  \begin{tikzpicture}
  \node (1) at (0,6) {$\beta_1$};
  \node (2) at (1,6) {$\beta_2$};
  \node (3) at (2,6) {$\beta_3$};
  \node (4) at (2,5) {$\beta_4$};
  \node (5) at (3,5) {$\beta_5$};
  \node (6) at (2,4) {$\beta_6$};
  \node (7) at (3,4) {$\beta_7$};
  \node (8) at (2,3) {$\beta_8$};
  \node (9) at (3,3) {$\beta_9$};
  \node (10) at (2,2) {$\beta_{10}$};
  \node (11) at (3,2) {$\beta_{11}$};
  \node (12) at (2,1) {$\beta_{12}$};
  \node (13) at (1,1) {$\beta_{13}$};
  \node (14) at (2,0) {$\beta_{14}$};
  \node (15) at (1,0) {$\beta_{15}$};
  \node (16) at (0,0) {$\beta_{16}$};
	\node (111) at (9,6) {$\beta_1$};
  \node (211) at (8,6) {$\beta_2$};
  \node (311) at (7,6) {$\beta_3$};
  \node (411) at (6,5) {$\beta_4$};
  \node (511) at (7,5) {$\beta_5$};
  \node (611) at (6,4) {$\beta_7$};
  \node (711) at (7,4) {$\beta_6$};
  \node (811) at (6,3) {$\beta_9$};
  \node (911) at (7,3) {$\beta_8$};
  \node (1011) at (6,2) {$\beta_{11}$};
  \node (1111) at (7,2) {$\beta_{10}$};
  \node (1211) at (8,1) {$\beta_{12}$};
  \node (1311) at (7,1) {$\beta_{13}$};
  \node (1411) at (7,0) {$\beta_{14}$};
  \node (1511) at (8,0) {$\beta_{15}$};
  \node (1611) at (9,0) {$\beta_{16}$};
%  \node (3) at (0,4) {$\beta_3$};
%  \node (a) at (-2,2) {$(0,1,1)$};
%  \node (b) at (0,2) {$(1,0,1)$};
%  \node (c) at (2,2) {$(1,1,0)$};
%  \node (d) at (-2,0) {$(0,0,1)$};
%  \node (e) at (0,0) {$(0,1,0)$};
%  \node (f) at (2,0) {$(1,0,0)$};
%  \node (min) at (0,-2) {$(0,0,0)$};
   %\draw (1) -- (2);
   \draw[->] (1) edge node[above] {\tiny{2}} (2);
   \draw[->] (2) edge node[above] {\tiny{4}} (3);
   \draw[->] (3) -- (4) node[left, pos=0.25] {\tiny{3}};
   \draw[->] (3) edge node[right, pos=0.25] {\tiny{5}} (5);
   \draw[->] (4) edge node[left] {\tiny{1}} (6);
   \draw[->] (5) edge node[right] {\tiny{3}} (7);
   \draw[->] (4) edge node[right] {\tiny{5}} (7);
   \draw[->] (6) edge node[left] {\tiny{5}} (8);
   \draw[->] (7) edge node[right] {\tiny{4}} (9);
   \draw[->] (7) edge node[right] {\tiny{1}} (8);
   \draw[->] (8) edge node[left] {\tiny{4}} (10);
   \draw[->] (9) edge node[right] {\tiny{2}} (11);
   \draw[->] (9) edge node[right] {\tiny{1}} (10);
    \draw[->] (10) edge node[right] {\tiny{2}} (12);
    \draw[->] (10) edge node[right] {\tiny{3}} (13);
%   \draw[->] (11) edge node[right] {\tiny{5}} (13);
   \draw[->] (11) edge node[right] {\tiny{1}} (12);
   \draw[->] (12) edge node[right] {\tiny{3}} (14);
   \draw[->] (13) edge node[right] {\tiny{2}} (14);
   \draw[->] (14) edge node[below] {\tiny{4}} (15);
   \draw[->] (15) edge node[below] {\tiny{5}} (16);
	
	\draw[->] (111) edge node[above] {\tiny{2}} (211);
   \draw[->] (211) edge node[above] {\tiny{4}} (311);
   \draw[->] (311) -- (411) node[left, pos=0.25] {\tiny{3}};
   \draw[->] (311) edge node[right, pos=0.25] {\tiny{5}} (511);
   \draw[->] (411) edge node[left] {\tiny{5}} (611);
   \draw[->] (511) edge node[right] {\tiny{6}} (711);
   \draw[->] (511) edge node[right] {\tiny{3}} (611);
   \draw[->] (611) edge node[left] {\tiny{4}} (811);
   \draw[->] (711) edge node[right] {\tiny{3}} (911);
   \draw[->] (611) edge node[right] {\tiny{6}} (911);
   \draw[->] (811) edge node[left] {\tiny{2}} (1011);
   \draw[->] (811) edge node[right] {\tiny{6}} (1111);
   \draw[->] (911) edge node[right] {\tiny{4}} (1111);
    \draw[->] (1111) edge node[right] {\tiny{5}} (1211);
    \draw[->] (1011) edge node[right] {\tiny{6}} (1311);
%   \draw[->] (11) edge node[right] {\tiny{5}} (13);
   \draw[->] (1111) edge node[right] {\tiny{2}} (1311);
   \draw[->] (1211) edge node[right] {\tiny{2}} (1411);
   \draw[->] (1311) edge node[right] {\tiny{5}} (1411);
   \draw[->] (1411) edge node[below] {\tiny{4}} (1511);
   \draw[->] (1511) edge node[below] {\tiny{3}} (1611);
  % \draw (0,5.5) node[anchor=west]  {\tiny{2}};
  % \draw (0,4.5) node[anchor=west]  {\tiny{4}};
%  \draw (min) -- (d) -- (a)  -- (b) -- (f)
%  (e) -- (min) -- (f) -- (c) --  (d) -- (b);
%  \draw[preaction={draw=white, -,line width=6pt}] (a) -- (e) -- (c);
\end{tikzpicture}
    \captionof{figure}{$H(\Ln_{\omega_6}^-)_{\mathtt{E_6}}$ and $H(\Ln_{\omega_1}^-)_{\mathtt{E_6}}$}
    \label{hasseE6}
  \end{minipage}
\end{center}
%\begin{figure}
%  \centering
%  \begin{tikzpicture}
%    % hier kommt die Grafik
%  \end{tikzpicture}
%  \caption{M1}
%  \label{hasseE6}
%\end{figure}
%\begin{center}
%  \begin{minipage}{\linewidth}
%    \centering
%    \begin{tikzpicture}
%    % hier kommt die Grafik
%    \end{tikzpicture}
%    \captionof{figure}{$H(\mathtt{E_6},\omega_1 )$}
%    \label{hasseE6}
%  \end{minipage}
%\end{center}
\begin{center}
\begin{minipage}{\linewidth}
\centering
\begin{eqnarray*}
x_{1} + x_{2} + x_{3} + x_{4} +  x_{6} + x_{8} + x_{{10}} +  x_{{13}} + x_{{14}} + x_{{15}} + x_{{16}} \leq m\\
x_{1} + x_{2} + x_{3} + x_{4} +  x_{6} + x_{8} + x_{{10}} +  x_{{12}} + x_{{14}} + x_{{15}} + x_{{16}} \leq m\\
x_{1} + x_{2} + x_{3} + x_{4} +  x_{7} + x_{8} + x_{{10}} +  x_{{13}} + x_{{14}} + x_{{15}} + x_{{16}} \leq m\\
x_{1} + x_{2} + x_{3} + x_{4} +  x_{7} + x_{8} + x_{{10}} +  x_{{12}} + x_{{14}} + x_{{15}} + x_{{16}} \leq m\\
x_{1} + x_{2} + x_{3} + x_{4} +  x_{7} + x_{9} + x_{{10}} +  x_{{13}} + x_{{14}} + x_{{15}} + x_{{16}} \leq m\\
x_{1} + x_{2} + x_{3} + x_{4} +  x_{7} + x_{9} + x_{{10}} +  x_{{12}} + x_{{14}} + x_{{15}} + x_{{16}} \leq m\\
x_{1} + x_{2} + x_{3} + x_{4} +  x_{7} + x_{9} + x_{{11}} +  x_{{12}} + x_{{14}} + x_{{15}} + x_{{16}} \leq m\\
x_{1} + x_{2} + x_{3} + x_{5} +  x_{7} + x_{8} + x_{{10}} +  x_{{13}} + x_{{14}} + x_{{15}} + x_{{16}} \leq m\\
x_{1} + x_{2} + x_{3} + x_{5} +  x_{7} + x_{8} + x_{{10}} +  x_{{12}} + x_{{14}} + x_{{15}} + x_{{16}}\leq m\\
x_{1} + x_{2} + x_{3} + x_{5} +  x_{7} + x_{9} + x_{{10}} +  x_{{13}} + x_{{14}} + x_{{15}} + x_{{16}}\leq m\\
x_{1} + x_{2} + x_{3} + x_{5} +  x_{7} + x_{9} + x_{{10}} +  x_{{12}} + x_{{14}} + x_{{15}} + x_{{16}}\leq m\\
x_{1} + x_{2} + x_{3} + x_{5} +  x_{7} + x_{9} + x_{{11}} +  x_{{12}} + x_{{14}} + x_{{15}} + x_{{16}}\leq m\end{eqnarray*}
  \captionof{table}{Polytope $P(m)$ corresponding to $\mathtt{E_6}$}
    \label{polyE6}
  \end{minipage}
\end{center}

\begin{center}
\begin{minipage}{\linewidth}
\centering
\begin{tikzpicture}
  \node (1) at (0,3) {$\beta_1$};
  \node (2) at (1,3) {$\beta_2$};
  \node (3) at (2,3) {$\beta_3$};
  \node (4) at (3,3) {$\beta_4$};
  \node (5) at (3,2) {$\beta_5$};
  \node (6) at (4,2) {$\beta_6$};
  \node (7) at (4,1) {$\beta_7$};
  \node (8) at (5,1) {$\beta_8$};
  \node (9) at (4,0) {$\beta_9$};
  \node (10) at (5,0) {$\beta_{10}$};
  \node (11) at (4,-1) {$\beta_{11}$};
  \node (12) at (5,-1) {$\beta_{12}$};
  \node (13) at (4,-2) {$\beta_{13}$};
  \node (14) at (5,-2) {$\beta_{14}$};
   \node (15) at (6,-2) {$\beta_{15}$};
   \node (16) at (4,-3) {$\beta_{16}$};
  \node (17) at (5,-3) {$\beta_{17}$};
  \node (18) at (4,-4) {$\beta_{18}$};
  \node (19) at (5,-4) {$\beta_{19}$};
  \node (20) at (4,-5) {$\beta_{20}$};
  \node (21) at (5,-5) {$\beta_{21}$};
  \node (22) at (4,-6) {$\beta_{22}$};
  \node (23) at (3,-6) {$\beta_{23}$};
  \node (24) at (3,-7) {$\beta_{24}$};
  \node (25) at (2,-7) {$\beta_{25}$};
  \node (26) at (1,-7) {$\beta_{26}$};
  \node (27) at (0,-7) {$\beta_{27}$};
%  \node (3) at (0,4) {$\beta_3$};
%  \node (a) at (-2,2) {$(0,1,1)$};
%  \node (b) at (0,2) {$(1,0,1)$};
%  \node (c) at (2,2) {$(1,1,0)$};
%  \node (d) at (-2,0) {$(0,0,1)$};
%  \node (e) at (0,0) {$(0,1,0)$};
%  \node (f) at (2,0) {$(1,0,0)$};
%  \node (min) at (0,-2) {$(0,0,0)$};
   %\draw (1) -- (2);
   \draw[->] (1) edge node[above] {\tiny{1}} (2);
   \draw[->] (2) edge node[above] {\tiny{3}} (3);
   \draw[->] (3) edge node[above] {\tiny{4}} (4);
   \draw[->] (4) edge node[left] {\tiny{2}} (5);
   \draw[->] (4) edge node[right, pos=0.5] {\tiny{5}} (6);
   \draw[->] (5) edge node[left, pos=0.5] {\tiny{5}} (7);
   \draw[->] (6) edge node[right, pos=0.5] {\tiny{6}} (8);
   \draw[->] (6) edge node[right, pos=0.5] {\tiny{2}} (7);
    \draw[->] (7) edge node[left, pos=0.5] {\tiny{4}} (9);
   \draw[->] (8) edge node[right, pos=0.5] {\tiny{2}} (10);
   \draw[->] (7) edge node[right, pos=0.5] {\tiny{6}} (10);
   \draw[->] (9) edge node[left, pos=0.5] {\tiny{3}} (11);
   \draw[->] (10) edge node[right, pos=0.5] {\tiny{4}} (12);
   \draw[->] (9) edge node[right, pos=0.5] {\tiny{6}} (12);
   \draw[->] (11) edge node[left, pos=0.5] {\tiny{1}} (13);
   \draw[->] (12) edge node[right, pos=0.5] {\tiny{3}} (14);
   \draw[->] (11) edge node[right, pos=0.5] {\tiny{6}} (14);
   \draw[->] (12) edge node[right, pos=0.5] {\tiny{5}} (15);
   \draw[->] (13) edge node[left, pos=0.5] {\tiny{6}} (16);
   \draw[->] (14) edge node[right, pos=0.5] {\tiny{5}} (17);
   \draw[->] (14) edge node[left, pos=0.5] {\tiny{1}} (16);
   \draw[->] (15) edge node[right, pos=0.5] {\tiny{3}} (17);
   \draw[->] (16) edge node[left, pos=0.5] {\tiny{5}} (18);
   \draw[->] (17) edge node[right, pos=0.5] {\tiny{4}} (19);
   \draw[->] (17) edge node[left, pos=0.5] {\tiny{1}} (18);
   \draw[->] (18) edge node[left, pos=0.5] {\tiny{4}} (20);
   \draw[->] (19) edge node[right, pos=0.5] {\tiny{2}} (21);
   \draw[->] (19) edge node[left, pos=0.5] {\tiny{1}} (20);
   \draw[->] (20) edge node[left, pos=0.5] {\tiny{2}} (22);
   \draw[->] (21) edge node[right, pos=0.5] {\tiny{1}} (22);
   \draw[->] (20) edge node[left, pos=0.5] {\tiny{3}} (23);
   \draw[->] (23) edge node[left, pos=0.5] {\tiny{2}} (24);
   \draw[->] (22) edge node[right, pos=0.5] {\tiny{3}} (24);
   \draw[->] (24) edge node[below, pos=0.5] {\tiny{4}} (25);
   \draw[->] (25) edge node[below, pos=0.5] {\tiny{5}} (26);
   \draw[->] (26) edge node[below, pos=0.5] {\tiny{6}} (27);
  % \draw (0,5.5) node[anchor=west]  {\tiny{2}};
  % \draw (0,4.5) node[anchor=west]  {\tiny{4}};
%  \draw (min) -- (d) -- (a)  -- (b) -- (f)
%  (e) -- (min) -- (f) -- (c) --  (d) -- (b);
%  \draw[preaction={draw=white, -,line width=6pt}] (a) -- (e) -- (c);
\end{tikzpicture}
    \captionof{figure}{$H(\Ln_{\omega_7}^-)_{\mathtt{E_7}}$}
    \label{hasseE7}
  \end{minipage}
\end{center}

\begin{center}
\begin{minipage}{\linewidth}
\centering
\begin{align*}
 &x_1+x_2+x_3+x_4+x_8+x_{10}+x_{11}+x_{13}+x_{14}+x_{15} &\leq m\\
 &x_1+x_2+x_3+x_4+x_8+x_{10}+x_{12}+x_{13}+x_{14}+x_{15} &\leq m\\
 &x_1+x_2+x_3+x_4+x_7+x_{9}+x_{11}+x_{13}+x_{14}+x_{15} &\leq m\\
 &x_1+x_2+x_3+x_4+x_7+x_{10}+x_{11}+x_{13}+x_{14}+x_{15} &\leq m\\
 &x_1+x_2+x_3+x_4+x_7+x_{10}+x_{12}+x_{13}+x_{14}+x_{15} &\leq m\\
 &x_1+x_2+x_4+x_5+x_8+x_{10}+x_{11}+x_{13}+x_{14}+x_{15} &\leq m\\
 &x_1+x_2+x_4+x_5+x_8+x_{10}+x_{12}+x_{13}+x_{14}+x_{15} &\leq m\\
 &x_1+x_2+x_4+x_5+x_7+x_{9}+x_{11}+x_{13}+x_{14}+x_{15} &\leq m\\
 &x_1+x_2+x_4+x_5+x_7+x_{10}+x_{11}+x_{13}+x_{14}+x_{15} &\leq m\\
 &x_1+x_2+x_4+x_5+x_7+x_{10}+x_{12}+x_{13}+x_{14}+x_{15} &\leq m\\
 &x_1+x_2+x_3+x_6+x_8+x_{10}+x_{11}+x_{13}+x_{14}+x_{15} &\leq m\\
 &x_1+x_2+x_3+x_6+x_8+x_{10}+x_{12}+x_{13}+x_{14}+x_{15} &\leq m\\
\end{align*}
    \captionof{table}{Polytope $P(m\omega_4)$ corresponding to $\mathtt{F_4}$}
    \label{polyF4}
  \end{minipage}
\end{center}

%\end{document} 

%% file: PBW-Basis_acknowledments.tex
\section*{Acknowledgments}
We are grateful to Ghislain Fourier, Deniz Kus and Peter Littelmann for very inspiring discussions. For many calculations on our polytopes we used the program LattE.  We thank the developers.\\ The work of T. B. was funded by the DFG Priority Program SPP 1388 "Representation Theory", C.D. was partially funded by this program.